\documentclass[11pt,reqno]{amsart}

\usepackage{amsmath}
\usepackage{amsfonts}
\usepackage{amssymb}
\usepackage{bm}
\newcommand{\vertiii}[1]{{\left\vert\kern-0.25ex\left\vert\kern-0.25ex\left\vert #1 
		\right\vert\kern-0.25ex\right\vert\kern-0.25ex\right\vert}}
\usepackage{amsthm}
\usepackage{mathrsfs}
\usepackage{latexsym}

\usepackage{cite} 

\usepackage{esint}

\numberwithin{equation}{section}

\usepackage{graphicx,subfigure}

\usepackage{ifthen} 

\provideboolean{shownotes} 
\setboolean{shownotes}{true} 
\newcommand{\margnote}[1]{
	\ifthenelse{\boolean{shownotes}}%
	{\marginpar{\raggedright\tiny\texttt{#1}}}%
	{}%
}

\newcommand{\hole}[1]{
	\ifthenelse{\boolean{shownotes}}%
	{\begin{center} \fbox{ \rule {.25cm}{0cm}
				\rule[-.1cm]{0cm}{.4cm} \parbox{.85\textwidth}{\begin{center}
						\texttt{#1}\end{center}} \rule {.25cm}{0cm}}\end{center}}
	{}
}

\usepackage[colorlinks=true,urlcolor=blue,
citecolor=red,linkcolor=blue,linktocpage,pdfpagelabels,
bookmarksnumbered,bookmarksopen]{hyperref}

\usepackage{fullpage}

\theoremstyle{plain}

\newtheorem{lemma}{Lemma}[section]
\newtheorem{theorem}[lemma]{Theorem}
\newtheorem{proposition}[lemma]{Proposition}
\newtheorem{corollary}[lemma]{Corollary}

\theoremstyle{definition}
\newtheorem{remark}[lemma]{Remark}
\newtheorem{definition}[lemma]{Definition}

\theoremstyle{remark}

\newcommand{\bfv}{\boldsymbol{v}}
\newcommand{\bfq}{\boldsymbol{q}}
\newcommand{\bfu}{\boldsymbol{u}}
\newcommand{\bfw}{\boldsymbol{w}}

 \newcommand{\bSigma}{\mathbf{\Sigma}}

\newcommand{\R}{\mathbb{R}}
\newcommand{\C}{\mathbb{C}}
\newcommand{\Z}{\mathbb{Z}}
\newcommand{\N}{\mathbb{N}}
\newcommand{\bbS}{\mathbb{S}}

\newcommand{\bbO}{0}
\newcommand{\bbI}{I}

\newcommand{\tirho}{\widetilde{\rho}}

\newcommand{\tithe}{\widetilde{\theta}}
\newcommand{\tibu}{\widetilde{\bm{u}}}

\newcommand{\tiK}{\widetilde{K}}

\newcommand{\tiA}{\widetilde{A}}
\newcommand{\tiB}{\widetilde{B}}

\newcommand{\cN}{{\mathcal{N}}}

\newcommand{\cD}{{\mathcal{D}}}
\newcommand{\cR}{{\mathcal{R}}}

\newcommand{\cE}{{\mathcal{E}}}

\newcommand{\vep}{\varepsilon}

\renewcommand{\Re}{\mathrm{Re}\,} 
\renewcommand{\Im}{\mathrm{Im}\,}

\newcommand{\btau}{\overline{\tau}}
\newcommand{\bbu}{\overline{\bm{u}}}
\newcommand{\bu}{\overline{u}}
\newcommand{\bv}{\overline{v}}

\newcommand{\bU}{\overline{U}}
\newcommand{\bmu}{\overline{\mu}}

\newcommand{\bk}{\overline{k}}
\newcommand{\bnu}{\overline{\nu}}
\newcommand{\bla}{\overline{\lambda}}
\newcommand{\bK}{\overline{K}}

\newcommand{\bB}{\overline{B}}

\newcommand{\bA}{\overline{A}}

\newcommand{\brho}{\overline{\rho}}
\newcommand{\bthe}{\overline{\theta}}
\newcommand{\balp}{\overline{\alpha}}
\newcommand{\bp}{\overline{p}}
\newcommand{\be}{\overline{e}}
\newcommand{\bsi}{\overline{\sigma}}

\newcommand{\hW}{\widehat{W}}
\newcommand{\hU}{\widehat{U}}
\newcommand{\hV}{\widehat{V}}

\newcommand{\<}{\langle}
\renewcommand{\>}{\rangle}

\begin{document}
	
\title[Dissipative structure of higher order systems]{Dissipative structure of higher order regularizations of hyperbolic systems of conservation laws in several space dimensions}

	
\author[F. Angeles]{Felipe Angeles}
	
\address{{\rm (F. Angeles)} Instituto de 
Investigaciones en Matem\'aticas Aplicadas y en Sistemas\\Universidad Nacional Aut\'onoma de 
M\'exico\\ Circuito Escolar s/n, Ciudad Universitaria, C.P. 04510\\Cd. de M\'{e}xico (Mexico)}
	
\email{felipe.angeles@iimas.unam.mx}
	
\author[R. G. Plaza]{Ram\'on G. Plaza}
	
\address{{\rm (R. G. Plaza)} Instituto de 
Investigaciones en Matem\'aticas Aplicadas y en Sistemas\\Universidad Nacional Aut\'onoma de 
M\'exico\\ Circuito Escolar s/n, Ciudad Universitaria, C.P. 04510\\Cd. de M\'{e}xico (Mexico)}
	
\email{plaza@aries.iimas.unam.mx}
	
\author[J. M. Valdovinos]{Jos\'{e} M. Valdovinos}
	
\address{{\rm (J. M. Valdovinos)} Instituto de 
Investigaciones en Matem\'aticas Aplicadas y en Sistemas\\Universidad Nacional Aut\'onoma de 
M\'exico\\ Circuito Escolar s/n, Ciudad Universitaria, C.P. 04510\\Cd. de M\'{e}xico (Mexico)}
	
	
\email{valdovinos94@comunidad.unam.mx}
	
\begin{abstract}
This work studies the dissipative structure of regularizations of any order of hyperbolic systems of conservation laws in several space dimensions. It is proved that the seminal equivalence theorem by Kawashima and Shizuta \cite{ShKa85}, which relates the strictly dissipative structure of second-order (viscous) systems to a genuine coupling condition of algebraic type, can be extended to higher-order multidimensional systems. For that purpose, the symbolic formulation of the genuine coupling condition by Humpherys \cite{Hu05} for linear operators of any order in one dimension, is adopted and extrapolated. Therefore, the concepts of symbol symmetrizability and genuine coupling are extended to the most general setting of differential operators of any order in several space dimensions. Applications to many viscous-dispersive systems of physical origin, such as compressible viscous-capillar fluids of Korteweg type, the dispersive Navier-Stokes-Fourier system and the equations of quantum hydrodynamics, illustrate the relevance of this extension.
%
\end{abstract}
	
	
	
\maketitle
	
\setcounter{tocdepth}{1}
	
	
	
\section{Introduction}
	
In this paper, we consider higher-order linear systems of partial differential equations of evolution in several space dimensions of the form,
\begin{equation}
\label{linsyst}
U_t + \sum_{0 \leq |\alpha| \leq m} L^\alpha D^\alpha U = 0,
\end{equation}
where $t \geq 0$ and $x = (x_1, \ldots, x_d)^\top \in \R^d$, $d \in \N$, $d \geq 2$, denote time and space variables, respectively; $ U = U(x,t) \in \R^n$, $ n \geq 1$, is a vector of $n$ unknowns and $m \in \N$ is the (arbitrary) highest order of differentiation. For any multi-index $\alpha = (\alpha_1, \ldots, \alpha_d)^\top$, $\alpha_j \in \Z$, $\alpha_j \geq 0$, $L^\alpha$ denotes a constant $n \times n$ real matrix. 
According to custom, $D^\alpha$ denotes the standard differential operator with respect to the space variables.
	
Systems of the form \eqref{linsyst} arise, for example, as linearizations around equilibrium states, $\bU \in \R^n$, of higher order regularizations of hyperbolic systems of conservation laws. The latter include viscous systems of conservation laws of the form,
\begin{equation}
\label{vscl}
U_t + \sum_{j=1}^d f^j(U)_{x_j} = \sum_{i,k = 1}^d \big( B^{ik}(U) U_{x_i} \big)_{x_k},
\end{equation}
viscous-dispersive systems,
\begin{equation}
\label{visccap}
U_t + \sum_{j=1}^d f^j(U)_{x_j} = \sum_{i,k = 1}^d \big( B^{ik}(U) U_{x_i} \big)_{x_k} + \sum_{i,j,k = 1}^d \big( C^{ijk}(U) U_{x_i x_j} \big)_{x_k},
\end{equation}
or relaxation systems,
\begin{equation}
\label{relax}
U_t + \sum_{j=1}^d f^j(U)_{x_j} = Q(U)
\end{equation}
(here it is assumed that the equilibrium state $\bU \in \R^n$ solves $Q(\bU)=0$, known as the Maxwellian or equilibrium manifold; see, e.g., \cite{MaN10}), as well as other higher order regularizations and their combinations.

For these systems of equations, it is of special interest to understand how the interaction between hyperbolic-convective effects and dissipative terms of zero or higher orders affects the behavior of solutions as time goes on. For example, the presence of diffusion terms in hyperbolic-parabolic systems and the dissipation due to the source term in a hyperbolic system with relaxation can compensate the loss of regularity -by the development of singularities- derived from the quasilinear hyperbolic structure of the system (see \cite{ShKa85, KaTh83, KaSh88, LiuZ97, HaNa03, KY09} and the references cited therein).

In the literature, there exist two fundamental contributions that established the theoretical foundations to determine whether a linearized system around an arbitrary constant state exhibits some abstract symmetrizability and dissipative properties, which can be extrapolated to the nonlinear problem: first, the seminal work by Kawashima and Shizuta \cite{KaSh88,ShKa85}, which studied the strict dissipativity of a large class of second order systems which include viscous and/or relaxation effects; and, second, the work by Humpherys \cite{Hu05}, who extended Kawashima and Shizuta’s notions of strict dissipativity, genuine coupling and symmetrizability to viscous-dispersive one-dimensional systems of any order. As a motivation for the reader, in the remainder of the Introduction we examine both contributions and define the purpose of this work. 


\subsection{Motivation: the viscous case}

In physical systems, some components of the unknown variables obey hyperbolic conservation laws without relaxation or diffusion terms and therefore, only partial dissipation is present. For instance, many systems describing the dynamics of compressible viscous heat conducting fluids in space are of the form \eqref{vscl}, where the matrices $B^{ik}$ have the following block structure,
\[
B^{ik}(U)= \begin{pmatrix}
\bbO_{n-p} & 0\\
0 &\widetilde{B}^{ik}(U)
\end{pmatrix},
\]
where $\bbO_{n-p}$ denotes the zero $(n-p) \times (n-p)$ block matrix and the $p \times p$ block $\widetilde{B}^{ik}(U)$ is assumed to be positive definite for all $U$. Thus, the full viscosity matrix symbol, $B(U;\xi):=\sum_{ik}B^{ik}(U)\xi_{i}\xi_{k}$, has rank $p$ for every $\xi=(\xi_{1},..,\xi_{d})^\top\in\mathbb{R}^{d}\setminus\{0\}$. The prototype for this kind of systems is the compressible Navier-Stokes-Fourier (NSF) system of equations, describing compressible viscous heat-exchanging fluids (see, e.g., Serre \cite{Ser10b}). In the analysis of the large time behavior of solutions that are perturbations of a constant state, the assumption that,
\begin{itemize}
	\item [(C$_1$)] \emph{no inviscid characteristic direction is contained in the null space of the viscosity matrix,}
\end{itemize}
guarantees the dissipativity of the system (see, e.g., Liu and Zumbrun \cite{LiuZ97}). On the other hand, many models fit into the framework of equations \eqref{relax}, as it is the case of the compressible Euler system for isentropic flows with damping (cf. \cite{HsLiu92,SiTW03}). For this kind of systems,
\begin{itemize}
	\item [(C$_2$)] \emph{ the presence of an equilibrium manifold made up by nothing but constant equilibrium solutions,}
\end{itemize}
plays a key role in preventing the existence of persistent signals, a necessary condition for the asymptotic stability of constant states (see Mascia and Natalini \cite[Section 2]{MaN10}). In both the viscous and relaxation cases, we refer to (C$_1$) and (C$_2$) as \emph{coupling conditions} and, once imposed in the linearizations around constant states, they ensure the total dissipation of the solutions. 

Linear systems of the form \eqref{linsyst} can also come from quasilinear systems of the form
\begin{equation}\label{quas-hyp-par-sys}
\begin{aligned}
&A_{1}^{0}(u,v)u_{t} + \sum_{j=1}^{d}A^{j}_{11}(u,v)u_{x_{j}} = f_{1}(u,v,D_{x}u), \\
&A_{2}^{0}(u,v)v_{t} - \sum_{j,k=1}^{d}B^{jk}(u,v)v_{x_{j}x_{k}}= f_{2}(u,v,D_{x}u,D_{x}v),
\end{aligned}
\end{equation}
after we linearize them around a constant equilibrium state $\bU=(\bu, \bv)^\top$, that is, $\bu \in \R^{n_{1}}$ and $\bv \in \R^{n_{2}}$ such that $f_{1}(\bu, \bv, 0)=0$ and $f_{2}(\bu, \bv, 0, 0)=0$. General systems of this form were studied in detail by Kawashima \cite{KaTh83}. The author proved  that the Cauchy problem for initial data close to a constant state is globally well-posed in Sobolev spaces and the solutions approach the equilibrium solution as time goes to infinity, provided that the system satisfies a \textit{dissipation condition}. This condition guarantees that the (second order) linearized system around the constant state possess  a \textit{dissipative structure}, and it is precisely this structure what enables the author to get the global existence and the asymptotic behaviour of solutions. It is worth mentioning that viscous systems of conservations laws \eqref{vscl} can be put into the form \eqref{quas-hyp-par-sys} if the system has a convex entropy and satisfies a condition called the \emph{$N$ condition} (see Kawashima and Shizuta \cite{KaSh88} for further details). 

Let us briefly explain the meaning of the dissipation condition for systems of the form \eqref{quas-hyp-par-sys}. Consider a constant state $\bU=(\bu, \bv)^\top$ and linearize the system around it to obtain

\begin{align}
\label{eq:2order}
A^{0}U_{t}+\sum_{j=1}^{d}A^{j}U_{x_{j}}-\sum_{j,k=1}^{d}B^{jk}U_{x_{j}x_{k}}+LU=0,
\end{align}
for which we assume the matrix coefficients $A^{0}$, $L$, $A^{j}$, $B^{j,k}$, $j,k=1,\ldots,d$, are symmetric with $A^{0}>0$, $L\geq 0$, and $\sum_{j,k=1}^{d}B^{jk}\omega_{j}\omega_{k}\geq 0$ for all $\omega \in \bbS^{d-1}$. Thus the dissipation condition reads: 
\begin{itemize}
\item[(K$_1$)] \emph{Dissipation condition:} There exist square matrices $K^{j}$, $j=1,\ldots,d$, of order $n$ such that
\begin{itemize}
\item[(i)] $K^{j}A^{0}$, $j=1,\ldots,d$, is a real skew-symmetric matrix. 
\item[(ii)] The matrix
\[
\sum_{j,k=1}^{d}\left( [K^{j}A^{k}]^{s} + B^{jk} \right)\omega_{j}\omega_{k} + L
\]
is positive definite for all $\omega \in \bbS^{d-1}$. Here $[M]^{s}$ denotes the symmetric part of the matrix $M$.  
\end{itemize}  
\end{itemize}

The matrices $K^{j}$, $j=1,\ldots,d$, serve as a computational tool, allowing us to obtain energy decay rates for the solutions of \eqref{eq:2order} (see, for example, \cite{AnMP20, KaTh83, UKS84}). Indeed, this condition is motivated by the work of Umeda, Kawashima and Shizuta \cite{UKS84}, where they showed that it is satisfied for the linearized equations of electro-magneto-fluid dynamics in three space dimensions ($d = 3$), and that it implies that solutions of this system decay at a rate of $(1+t)^{-3/4}$ in $L^{2}(\R^{3})$ for initial data in $L^{2}(\R^{3}) \cap L^{1}(\R^{3})$.

In a subsequent work, Shizuta and Kawashima \cite{ShKa85} proved that the dissipation condition (K$_1$) is equivalent to each of the following conditions:

\begin{itemize}
	\item [(K$_2$)] \emph{Strict dissipativity}: If $\lambda=\lambda(\xi)$ is an eigenvalue of system \eqref{eq:2order} in the Fourier space, that is, if $\lambda$ satisfies	
\begin{equation}\label{eig-prob-2ord}
\left( \lambda A^{0} + i \vert \xi \vert A(\omega) + \vert \xi \vert^2 B(\omega) + L\right) \phi=0, \quad \phi
\in \C^{n}, \quad \phi \neq 0,
\end{equation}
for  $\xi \in \R^{d} \setminus \{ 0 \}$, $\omega = \xi/ \vert \xi \vert$, then $\operatorname{Re}\lambda(\xi)<0$. Here $A(\omega)$ and $B(\omega)$ are given by
\[
A(\omega)=\sum_{j=1}^{d}A^{j}\omega_{j}, \quad 
B(\omega)=\sum_{j,k=1}^{d}B^{jk}\omega_{j}\omega_{k},\quad \omega \in \bbS^{d-1}. 
\]
\item [(K$_3$)] \emph{Genuine coupling condition}: If for some $\omega \in \bbS^{d-1}$ there holds and $\sum_{j,k}B^{jk}\omega_{j}\omega_{k}v=0$ for $v\in\mathbb{R}^{n}$, $v\neq 0$, then $\mu A^{0}v+\sum_{j}A^{j}\xi_{j}v\neq 0$ for any $\mu\in\mathbb{R}$.

\item[(K$_4$)] There exists a positive constant $c$ such that for every $\xi \in \R^{d}$, $\xi \neq 0 $, we have 
\[
\operatorname{Re} \lambda(\xi) \leq - c \frac{ \vert \xi \vert^{2}}{1 + \vert \xi \vert^{2}},
\]
for eigenvalues $\lambda = \lambda(\xi)$ for \eqref{eig-prob-2ord}.  

\end{itemize}
The authors also pointed out that the dissipation condition can be replaced by the following one.

\begin{itemize}
\item [(K$_1'$)] \emph{Existence of a compensating function}: There exists a matrix valued function $K(\omega)$, $\omega \in \bbS^{d-1}$, such that
\begin{itemize}
\item[(i)] $K(\omega)A^{0}$ is skew-symmetric for all $\omega \in \bbS^{d-1}$, and
\item[(ii)] the matrix
\[
[K(\omega)A(\omega)]^{s} + B(\omega) + L 
\]
is positive definite for all $\omega \in \bbS^{d-1}$. 
\end{itemize}
The matrix $K(\omega)$ is called a \emph{compensating function} for system \eqref{eq:2order}. 
\end{itemize}


Regarding the compensating function, it should be noted that it has been used not only to study the asymptotic behaviour of constant equilibrium solutions of the systems \eqref{vscl}, \eqref{relax}, and \eqref{quas-hyp-par-sys}, but also in the stability theory of small-amplitude shock profiles for systems of viscous, relaxation and/or radiating conservations laws (see, e.g., \cite{HuZ02, MZ09, NPZ10}). 

\subsection{Extension to viscous-dispersive systems}

Another sort of models whose linearization leads to systems of the form \eqref{linsyst} are viscous-capillar systems of conservation laws, which have the general form of a third order system \eqref{visccap}.
Examples of the latter (modulo some nonlinear terms) come from compressible fluid dynamic models that go beyond the Navier-Stokes equations: the model describing a compressible, heat conducting, and viscous fluid with capillarity effects which was proposed by Korteweg \cite{Kortw1901} and later reformulated by Dunn and Serrin \cite{DS85}, known as \emph{compressible fluids of Korteweg type}; the dispersive correction to the NSF system proposed by Levermore \emph{et al.} \cite{Lev03rep, LeSu11}; quantum hydrodynamics (QHD) models \cite{Mdlng27} with the Bohm potential \cite{Boh52a}; and the ``modified'' Navier-Stokes equations proposed by Brenner \cite{Brnnr05a, Brnnr05b} (see also \cite{FeiVass10}), among many others.


Just like in the viscous case, for systems of the form  \eqref{visccap} it is important to understand the interaction between the convective/dispersive mechanisms and the ones associated to dissipation. In this direction we have to mention the seminal work by Humpherys \cite{Hu05}, where the equivalence theorem by Kawashima and Shizuta was extended to higher order \emph{one-dimensional} linear systems of the form  
\begin{align}
\label{eq:highorder}
V_{t}=-\sum_{k=0}^{m}D_{k}\partial_{x}^{k}V,\quad x\in\mathbb{R},\quad t>0,\quad V\in\mathbb{R}^{n}.
\end{align}

In order to establish appropriate coupling and dissipation conditions for this case, Humpherys groups together all the terms involving space derivatives of even order and associates a \emph{generalized viscosity}  symbol matrix of the form,
\begin{align*}
B(\xi)=\sum_{k~\text{even}}(-1)^{k/2}D_{k}\xi^{k},
\end{align*}
and regarding the remaining odd terms, the \emph{generalized transport symbol} is given by
\begin{align*}
A(\xi)=\sum_{k~\text{odd}}D_{k}(i\xi)^{k-1}.
\end{align*}

In this fashion, we say that \eqref{eq:highorder} is genuinely coupled \emph{in the sense of Humpherys} if no eigenvector of $A(\xi)$ belongs to the kernel of $B(\xi)$, whereas an smooth matrix valued function $K(\xi)$ is said to be a \emph{compensating matrix symbol} for system \eqref{eq:highorder} if it is skew-symmetric for all $\xi \in \R$ and satisfies
\[
[K(\xi) A(\xi)]^{s} + B(\xi) >0,\quad \forall \xi \in \R.
\] 
Note that the characteristic equation can be written as
\begin{align*}
\det\left(\lambda+i\xi A(\xi)+B(\xi)\right)=0.
\end{align*}
Then, the equivalence between the strict dissipativity, the existence of a compensating matrix and Humpherys' coupling condition is proved under the assumptions:
\begin{itemize}
	\item [(H$_1$)] $A(\xi)$ is symmetric and with eigenvalues of constant multiplicities.
	\item [(H$_2$)] $B(\xi)$ is symmetric and positive semi-definite.
\end{itemize}

In contrast with the result of Kawashima and Shizuta, assumption (H$_1$) allows the explicit computation of the compensating matrix as the Drazin inverse of the commutator operator $[A(\xi),\cdot]$, which is given by the formula
\begin{align}
K(\xi)=\sum_{i\neq j}\frac{\pi_{i}(\xi)B(\xi)\pi_{j}(\xi)}{\mu_{i}(\xi)-\mu_{j}(\xi)},\label{eq:humpcompensation}
\end{align}
where $\{\mu_{j}(\xi)\}_{j=1}^{r}$ is the set of distinct eigenvalues of $A(\xi)$ with corresponding eigenprojections $\{\pi_{j}(\xi)\}_{j=1}^{r}$. In particular, $K(\xi)$ is real analytic. Also observe that the constant multiplicity is needed to keep the denominator in \eqref{eq:humpcompensation} away from zero.

Regarding the extendibility of this equivalence result to symmetrizable systems, Humpherys argues against the limitation of considering \emph{term wise symmetrizers} (also known as \emph{Friedrichs' symmetrizers}; cf. \cite{Frd54,FLa67,Godu61a}) for higher order systems, since, in some cases, the conditions for simultaneous symmetrizability of each constant coefficient matrix make the existence of classical Friedrichs' symmetrizers impossible. Consequently, Humpherys introduces the notion of a \emph{symbolic symmetrizer} for both symbols $A(\xi)$ and $B(\xi)$. Symbol symmetrizability is, therefore, an important hypothesis to prove the equivalence theorem. Examples of physical, one-dimensional systems which are symbol symmetrizable but not symmetrizable in the sense of Friedrichs include isothermal compressible fluids with viscosity and capillarity (cf. Humpherys \cite{Hu05}, Plaza and Valdovinos \cite{PlV22}), heat conducting viscous capillar compressible fluids (cf. Plaza and Valdovinos \cite{PlV2}) and isentropic quantum hydrodynamics systems with artificial viscosity (cf. Plaza and Zhelyazov \cite{PlZ24}). 


At this point, we recall the recent work by Kawashima \emph{et al.} \cite{KSX22} on the dissipative structure of viscous-capillar systems of the form \eqref{visccap}, where the authors apply their results to some versions of the compressible viscous-capillar fluids of Korteweg type. It is to be observed that, in \cite{KSX22}, no equivalence theorem for those systems is presented. Instead, their analysis is based on some \emph{ad hoc} ``craftsmanship conditions'' which allow the authors to handle the non-symmetric dispersive terms. Under such conditions, they were able to show that the linear systems for the isothermal and heat conducting compressible viscous-capillar fluid of Korteweg type in several space dimensions are uniformly dissipative, satisfying
\[
\operatorname{Re} \lambda(\xi) \leq -c \vert \xi \vert^{2}, \quad \xi \in \R^{d} \setminus \{ 0 \},
\] 
and which they call dissipative of the \emph{regularity-gain type}, as it corresponds to the case $p > q$ in \eqref{diss-(k,l)} ($p = 1$ and $q = 0$; see Remark \ref{Rmr-reg-typ} below). Similar results have been also obtained in the case of one space dimension by two of the authors of the present paper (cf. \cite{PlV22,PlV2}) under the framework of Humpherys' equivalence theorem and symbol symmetrizability. Moreover, the results of \cite{PlV22,PlV2} also include the nonlinear stability analysis of  equilibrium constant states under small perturbations. (The results of \cite{KSX22} hold at a linear level, only.) It is to be noticed that another model considered in \cite{KSX22} is the system that describes a compressible, heat conducting, \emph{non-viscous}, capillar fluid of Korteweg type, also referred to as the Euler-Fourier-Korteweg (EFK) system, but only in one space dimension. The authors showed that this system is uniformly dissipative of the \emph{standard type} (see Remark \ref{Rmr-reg-typ} below). The reason why the multidimensional version of this system was not included in the analysis of \cite{KSX22} was not explained in that reference. In this paper we prove that, in several space dimensions, the strict dissipativity of the EFK system actually \emph{does not hold} (see Section \ref{secEFKmultid} below).

The potential applications of Humpherys' extension are numerous, not only because it pertains to linear operators of any order, but also due to the flexibility and generality of the symbol formulation. Even though Humpherys did not pursue his studies at the nonlinear level, we regard Humpherys' work as a fundamental contribution to the theory. 


\subsection{Contributions of this work}

In this paper, we extend the main result of Humpherys \cite{Hu05} to several space dimensions, a formulation that demands attention due to its potential applications to physical systems and which, as far as we know, has not been reported before in the literature. Albeit such extension may seem quite natural, in this paper we make some technical improvements to the one-dimensional proof; for example, we circumvent the constant multiplicity assumption (H$_1$) for the ``hyperbolic'' symbol $A(\xi)$. In this fashion, we consider our extension not merely incremental but a technical contribution as well (for a discussion of the relevance of avoiding the constant multiplicity assumption in the context of shock propagation see \cite{Ng09,MeZ05}). We study higher order linear systems of the form \eqref{linsyst} and apply Fourier's transform to obtain the equivalent system,
\begin{align*}
\widehat{U}_{t}=P(i\xi)\widehat{U},\quad\mbox{with}\quad P(i\xi)=-i|\xi|A(\xi)-B(\xi).
\end{align*}

Similarly to the one dimensional case, the symbol $|\xi|A(\xi)$ involves all the odd powers of $|\xi|$, with $\xi\in\mathbb{R}^{d}\setminus\{0\}$, and the remaining terms, with even powers, are contained in $B(\xi)$ (see equations \eqref{gen-flux-visc}). Then, with these \emph{generalized transport} and \emph{generalized viscosity} matrices, a genuinely coupling condition is established (see Definition \ref{gencoupling} below). Nonetheless, we \emph{do not make assumption} (H$_1$) in order to prove the equivalence between the genuinely coupling condition, the strict dissipativity and the existence of a compensating function. Instead, we proceed as in \cite{ShKa85}, that is, we begin by associating the corresponding symmetrized matrices $A_S(\xi)$ and $B_{S}(\xi)$ to $A(\xi)$ and $B(\xi)$, respectively. In some neighborhood of each $\xi_{0}\in\mathbb{R}^{d}\setminus\{0\}$, $A_S(\xi)$ can be approximated by a matrix symbol $\widetilde{A}_{S}(\xi)$ with constant multiplicities and, for all $\xi\in\mathbb{R}^{d}\setminus\{0\}$, $B_{S}(\xi)$ is similar to $B(\xi)$. Moreover, locally, the coupling condition holds true for the triplet $(I,\widetilde{A}_{S}(\xi),B_{S}(\xi))$ provided that it holds true for the triplet $(I, A_S(\xi), B_S(\xi))$. At this point, we construct the compensating matrix in a small neighborhood of $\xi_{0}$ by using Humpherys' formula \eqref{eq:humpcompensation}. Then we obtain the global form of $K(\xi)$ by means of a partition of unity argument. In this fashion, we combine the ideas presented in \cite{Hu05} and \cite{ShKa85}, in order to provide the most general formulation of the equivalence theorem (i.e., without the constant multiplicity assumption) for linear systems involving high order terms that underlie symbolic symmetrizers.

To exhibit the relevance of the extension, we examine many physical examples in several space dimensions, ranging from compressible viscous fluids with capillarity to models in quantum hydrodynamics. In each case, we look for a symmetrizer and verify the genuinely coupling condition. Whenever the genuine coupling condition is satisfied, we can apply Theorem \ref{equivalencetheo} to conclude the existence of a compensating symbol function. Even though the systems that satisfy the constant multiplicity assumption have a global compensating function given by formula \eqref{eq:humpcompensation}, in some cases, the computation of the eigenprojectors $\{\pi_{j}(\xi)\}_{j=1}^{r}$ is not possible, due to the non-availability of analytic expressions for the eigenvectors of $A(\xi)$, as functions of $\xi\in\mathbb{R}^{d}\setminus\{0\}$. Still, a compensating function can be found (perhaps more conveniently) by the method of inspection.

Let us summarize the contributions of this paper:
\begin{itemize}
\item[-] We extend the notions of symbol symmetrizability, genuine coupling, strict dissipativity and of the compensating matrix symbol to linear systems of any order of differentiation in several space dimensions.

\item[-] We combine the explicit formula for the compensating matrix symbol, say $K(\xi)$, and the constant multiplicity of the symmetrized symbols $\widetilde{A}_S(\xi)$ (see \eqref{hA} below) to verify the equivalence of the existence of $K(\xi)$ and the genuine coupling condition, circumventing in this fashion the constant multiplicity assumption (H$_1$) for the original symbol.

\item[-] We complete the proof of the equivalence theorem by showing the equivalence in several space dimensions of genuine coupling and the strict dissipativity of the system.

\item[-] Based on suitable assumptions on the compensating matrix symbol we deduce pointwise estimates in the Fourier space for the solutions to the linearized system. These pointwise estimates will eventually lead to decaying estimates on the physical space. The latter will depend, on a case by case basis, on the symmetrizer and on the compensating matrix symbols, which determine the appropriate energy spaces.

\item[-] We examine several examples from physical theories which satisfy the assumptions of the equivalence Theorem \ref{equivalencetheo}. These examples include: isothermal compressible fluids exhibiting viscosity and capillarity \cite{DS85,PlV22}; heat-conducting, viscous compressible fluids of Korteweg type \cite{DS85,PlV2,Kot10}; one-dimensional inviscid heat conducting fluids with capillarity \cite{KSX22}; the dispersive Navier-Stokes system \cite{Le96,Lev03rep,LeSu11}; and, hydrodynamics systems from quantum semiconductor theory \cite{GarC94}.

\item[-] Finally, we present a couple of physical systems in several space dimensions which do not satisfy the hypotheses of the equivalence Theorem \ref{equivalencetheo}: the full quantum hydrodynamics systems with energy exchanges \cite{GarC94} and the Euler-Korteweg-Fourier system in several space dimensions (the inviscid version of the Kortweg system derived by Dunn and Serrin \cite{DS85}).
\end{itemize}

\subsection*{Plan of the paper} This paper is structured as follows. The preliminary Section \ref{secprel} contains the extensions to multi-dimensions of the definitions of genuine coupling, symbol symmetrizability and strict dissipativity for systems of arbitrary order of differentiation. Section \ref{secequiv} is devoted to the proof of the equivalence theorem (see Theorem \ref{equivalencetheo} below). Section \ref{secenergy} describes how to obtain pointwise energy estimates in the Fourier space for strictly dissipative systems. In Section \ref{secappl} we describe in detail some systems of physical origin and verify the hypotheses of the equivalence theorem. These examples include compressible fluids of Korteweg type, systems in quantum hydrodynamics and the dispersive Navier-Stokes equations. Moreover, we employ the estimates in Fourier space to provide decaying energy estimates for the linearized operator in appropriate physical energy spaces. Finally, in Section \ref{secnoway} we exhibit a couple of examples of models in continuum mechanics which \emph{do not} satisfy the hypotheses of the equivalence theorem.
	
\subsection*{Notation}
	
The dimension of the physical space is $d \geq 2$. The Euclidean inner product in $\C^n$ is denoted as $\langle v, u \rangle = \sum_{j=1}^n v_j^* u_j$ where $\lambda^*$ denotes complex conjugation of any $\lambda \in \C$. All vectors $u \in \R^d$ are \emph{column} vectors, so that $u = (u_1, \ldots, u_d)^\top$. We employ the following (standard) notation: for any multi-index $\alpha = (\alpha_1, \ldots, \alpha_d)^\top$, $\alpha_j \in \Z$, $\alpha_j \geq 0$ we define $|\alpha| = \sum_{j=1}^d \alpha_j$, $\alpha ! = \alpha_1 ! \cdots \alpha_d!$ and $\zeta^\alpha = \zeta_1^{\alpha_1} \zeta_2^{\alpha_2} \cdots \zeta_d^{\alpha_d}$ for any $\zeta \in \C^d$. Higher-order derivatives are expressed by multi-indices as
\[
D^\alpha = \prod_{j=1}^d \Big( \frac{\partial}{\partial x_j}\Big)^{\alpha_j} = \frac{\partial^{}|\alpha|}{\partial x_1^{\alpha_1} \cdots \partial x_d^{\alpha_d}}.
\]
In particular, if $|\alpha| = 0$ then $D^\alpha$ is the identity operator. The set of linear isomorphisms from $\R^n$ to $\R^n$ is denoted as $GL_{n}$. $\bbO_{p \times q}$ will denote the zero $p \times q$ block matrix, for any $p, q \in \N$. The square zero and identity $p \times p$ matrices, with $p \in \N$, will be written as $\bbO_{p}$ and $\bbI_{p}$, respectively. For any real matrix $M \in \R^{n \times n}$, $[M]^{s} := \frac{1}{2}(M+M^\top)$ denotes its symmetric part. 
	
\section{Preliminaries}
\label{secprel}
	
In this section we extend to several space dimensions the notions of strict dissipativity, symmetrizability and genuine coupling of higher order systems introduced by Kawashima and Shizuta for constant coefficient second order systems \cite{KaSh88} and by Humpherys in the context of one space dimensional higher order systems (cf. \cite{Hu05}). Take the Fourier transform to system \eqref{linsyst}. The result is
\begin{equation}
\label{Fouriersyst}
\hU_t + \sum_{0 \leq |\alpha| \leq m} (i\xi)^\alpha L^\alpha \hU = 0, \quad \xi \in \R^d,
\end{equation}
where $\hU = \hU(\xi,t)$ denotes the Fourier transform of $U = U(x,t)$. For each $\xi\in \R^d$, $\xi \neq 0$, let us denote
\[
\omega = \omega(\xi) := \frac{\xi}{|\xi|} \in \bbS^{d-1}.
\]
Following Humpherys \cite{Hu05}, we split the symbol into odd and even terms. For that purpose, we define
\begin{equation}
\label{evenodd}
\left\{
\begin{aligned}
A^{\alpha}(\omega) &:= (-1)^{\tfrac{1}{2}(|\alpha|-1)} \omega^\alpha L^\alpha, & \quad |\alpha| \; \text{odd, } 1 \leq |\alpha| \leq m, \\
B^{\alpha}(\omega) &:= (-1)^{|\alpha|/2} \omega^\alpha L^\alpha, & \quad  |\alpha| \; \text{even, } 0 \leq |\alpha| \leq m,
\end{aligned}
\right.
\end{equation}
for all $\omega \in \bbS^{d-1}$, so that system \eqref{Fouriersyst} can be recast as
\begin{equation}
\label{Fourieroe}
\hU_t + i |\xi| A(\xi) \hU + B(\xi) \hU = 0,
\end{equation}
where
\begin{equation}\label{gen-flux-visc}
\begin{aligned}
A(\xi) &:= \sum_{\substack{1 \leq |\alpha| \leq m \\ |\alpha| \, \text{odd}}} |\xi|^{\vert \alpha \vert - 1} A^\alpha(\omega(\xi)), \\
B(\xi) &:= \sum_{\substack{0 \leq |\alpha| \leq m \\ |\alpha| \, \text{even}}} |\xi|^{\vert \alpha \vert} B^\alpha(\omega(\xi)).
\end{aligned}
\end{equation}
	
The solutions to system \eqref{Fouriersyst} and their evolution are determined by the family of eigenvalue problems,
\begin{equation}
\label{spectralprob}
\big( \lambda \bbI_n + i |\xi|A(\xi) + B(\xi) \big) \hU = 0,
\end{equation}
parametrized by the wave numbers $\xi\in \R^d$ in Fourier space and the frequency $\lambda \in \C$, the latter playing the role of the spectral parameter.
	
\begin{definition}
The system \eqref{linsyst} is \emph{strictly dissipative} if for all $\xi \in \R^d$, $\xi \neq 0$, the solutions to the spectral problem \eqref{spectralprob} satisfy
\[
\Re \lambda(\xi) < 0.
\]
\end{definition}
	
\begin{remark}
\label{Rmr-reg-typ}
Following Ueda \emph{et al.} \cite{UDK12, UDK18} it is said that the linear system \eqref{linsyst} is strictly dissipative of type $(p,q)$ (for non-negative integers $p$ and $q$) if the eigenvalues of the spectral problem \eqref{spectralprob} satisfy
\begin{equation}
\label{diss-(k,l)}
\Re \lambda(\xi) \leq - \, \frac{C |\xi|^{2p}}{(1 + |\xi|^2)^{q}}, 
\end{equation}
for some uniform constant $C > 0$ and all $\xi \in \R^d$, $\xi\neq 0$. The system is of standard type when $p = q$ \cite{UDK12}, and of regularity-loss type when $p < q$ \cite{UDK18}. Notice that strict dissipativity of type $(1,0)$ is precisely that of the heat kernel. Hence, the third case when $p > q$ is called dissipativity of regularity-gain type \cite{KSX22}. The type of dissipativity determines the decay rate of the solutions to the linearized system and the energy spaces where the solution belongs to (for examples in one dimension, see \cite{PlV22,PlV2}). 
\end{remark}
\begin{remark}
It is to be observed that strict dissipativity is equivalent to the stability of the essential spectrum of the linearized operator around the equilibrium state when computed with respect to the space $L^2(\R^d)$ of finite energy perturbations. Indeed, if the constant matrices $L^{\alpha}$ result from linearizing coefficients around a constant state $\bU$ then equation \eqref{linsyst} defines a linear operator, $\mathcal{L} = - \sum_{0 \leq |\alpha| \leq m} L^\alpha D^\alpha$, $\mathcal{L} : L^2 \to L^2$, with dense domain $D(\mathcal{L}) = H^m(\R^d)$ and whose essential spectrum is determined by the dispersion relation \eqref{spectralprob}.
\end{remark}

%
%
	
	Let us now recall that the constant coefficients linear system \eqref{linsyst} is said to be \emph{symmetrizable in the sense of Friedrichs} (see, e.g., Friedrichs \cite{Frd54}, Lax and Friedrichs \cite{FLa67} and Godunov \cite{Godu61a}) if there exists a real symmetric positive-definite constant matrix $S \in \R^{n \times n}$ such that the matrices $SL^\alpha$ are all real symmetric, for any multi-index $\alpha$. This notion, however, is of restricted applicability for higher order systems (see Remark \ref{remnoFsymm} below). Humpherys \cite{Hu05} then introduces the following definition, which generalizes the standard notion of symmetrizability in the sense of Friedrichs.
	
\begin{definition}
\label{defsymbolsymm}
System \eqref{linsyst} is said to be \emph{symbol symmetrizable} (or \emph{symmetrizable in the sense of Humpherys}) if there exists a smooth $n \times n$ symbol $S = S(\xi)$, $S \in C^\infty(\R^d\setminus \{ 0 \};\R^{n \times n})$, symmetric and positive definite for all $\xi\neq 0$, such that both $S(\xi) A(\xi)$ and $S(\xi) B(\xi)$ are symmetric for each $\xi \in \R^d$, and $S(\xi) B(\xi)$ is positive semi-definite, $S(\xi) B(\xi) \geq 0$, for all $\xi \in \R^d \setminus \{ 0 \}$.
\end{definition}

\begin{remark}
\label{remnoFsymm}
There exist examples of physical systems in one space dimension which are not symmetrizable in the sense of Friedrichs, but symbol symmetrizable in the sense of Humpherys (see, e.g., \cite{Hu05,PlV22,PlV2,PlZ24}). In this paper, in Section \ref{secappl}, we present their multidimensional counterparts: symbol symmetrizable systems which are not symmetrizable in the classical sense. This fact, and the importance of symmetrization of differential operators in relation to the existence of associated convex entropy functions, make difficult to underestimate the importance of Definition \ref{defsymbolsymm}.
\end{remark}


Once a certain system is symmetrized in the sense of Humpherys, one may ask whether there exists a compensating matrix symbol for it.

\begin{definition}
\label{Matrixsymbol}
Assume that system \eqref{linsyst} is symbol symmetrizable. Then a smooth, real matrix valued function,  $K\in C^{\infty} \left( \R^d \setminus \{ 0 \} ; \R^{n \times n} \right)$, is said to be a \emph{compensating matrix symbol} for the triplet $(S, A, B)$ provided that
\begin{itemize}
\item[(a)] $K(\xi)S(\xi)$ is skew-symmetric for all $\xi\in \R^d$, $\xi \neq 0$; and,
\item[(b)] $\big[K(\xi) S(\xi) A(\xi)\big]^{s}+ S(\xi)B(\xi) \geq \theta(\xi) \bbI_n > 0$ for all $\xi \in \R^d \setminus \{ 0 \}$, and some $\theta = \theta(\xi) > 0$.
\end{itemize}
\end{definition}
	
	
Finally, for any system of the form \eqref{Fouriersyst}, whether it is symbol symmetrizable or not, we have the following fundamental property, known as the \emph{genuine coupling condition}.
\begin{definition}
\label{gencoupling}
It is said that system \eqref{linsyst} is \emph{genuinely coupled} if for all fixed $\xi \in \R^d$ with $\xi \neq 0$, $(\mu \bbI_n + A(\xi)) \psi \neq 0$ for every $\mu \in \R$ and every $\psi \in \ker B(\xi)$, $\psi \neq 0$.
\end{definition}
	
\begin{remark}
Notice that the genuine coupling condition underlies a simple algebraic property for system \eqref{Fouriersyst}, inasmuch as it reduces to verifying that no eigenvector of the generalized transport symbol, $A(\xi)$, lies in the kernel of the generalized viscosity (or dissipation) symbol, $B(\xi)$.
\end{remark}
	
\section{The equivalence theorem}
\label{secequiv}
	
In this section we prove the equivalence between strict dissipativity, genuine coupling and the existence of a compensating matrix symbol in several space dimensions. 

\subsection{Preliminary results}

Before we proceed to enunciate and to prove the equivalence theorem, we are going to state some technical results which will be applied later on. Let us start with the following lemma.

\begin{lemma}
\label{expendable1}
Let $\mathcal{V}$ be the vector space of square symmetric matrices of order $n$ and let $\mathcal{V}_{+}\subset \mathcal{V}$ be the open set of symmetric positive definite matrices. Define the function $F:\mathcal{V} \rightarrow \mathcal{V}$ by $F(M)=M^{2}$ for any $M\in \mathcal{V}$. Then $F^{-1}$ exists and $F$ is a smooth diffeomorphism of $\mathcal{V}_{+}$ onto $F(\mathcal{V}_{+})$.
\end{lemma}
\begin{proof}
Let $M\in \mathcal{V}_{+}$ be arbitrary. It holds that 
\begin{align*}
DF(M)V=VM+MV, \quad\text{for any} \; V\in\mathcal{V}.
\end{align*}
Now, assume that $VM+MV = 0_n$, the null square matrix of order $n$. Let $\{e_{1},..,e_{n}\}$ be a basis of eigenvectors of $M$ with corresponding eigenvalues $\{\lambda_{1},..,\lambda_{n}\}$. Then,
\begin{align*}
MVe_{i}=-VMe_{i}=-\lambda_{i}Ve_{i}\quad\mbox{for all}\quad i=1,..,n.
\end{align*}
It follows that $Ve_{i}$ is an eigenvector of $M$ with negative eigenvalue, a contradiction. Therefore $DF(M)$ in injective. The result now follows by the inverse function theorem (see, for instance,  Theorem 2.5.2 in Abraham \emph{et al.} \cite{AMR88}).
\end{proof}
\begin{lemma}
\label{hypequiv}
Let us assume that \eqref{linsyst} is symbol symmetrizable and set $\widehat{V}:=S^{1/2}(\xi)\widehat{U}$. Then $\widehat{V}$ satisfies the following symmetric system
\begin{align}
\widehat{V}_{t}+i|\xi|A_{S}(\xi)\widehat{V}+B_{S}(\xi)\widehat{V}=0,\label{Fouriersyst2}
\end{align}
where $A_{S}, B_{S}\in C^{\infty}(\mathbb{R}^{d}\setminus\{0\};\mathbb{R}^{n\times n})$ are given by
\begin{align*}
A_{S}(\xi):=S^{1/2}(\xi)A(\xi)S^{-1/2}(\xi)\quad\mbox{and}\quad B_{S}(\xi):=S^{1/2}(\xi)B(\xi)S^{-1/2}(\xi).
\end{align*}
Moreover, the conditions of strict dissipativiy and genuinely coupling hold for \eqref{Fouriersyst} if and only if they hold for \eqref{Fouriersyst2}. Assume that $K_{S}\in C^{\infty}(\mathbb{R}^{d}:\mathbb{R}^{n\times n})$ is a compensating matrix symbol for the triplet $(I,A_{S},B_{S})$. Then,
\begin{align}
K(\xi):=S^{1/2}(\xi)K_{S}(\xi)S^{-1/2}(\xi)\label{compmatrix}
\end{align}
is a compensating matrix symbol for the triplet $(S,A,B)$.
\end{lemma}
\begin{proof}
Assume that $S=S(\xi)$ is a symbolic symmetrizer for  \eqref{linsyst}. Multiply \eqref{Fouriersyst} from the left by $S(\xi)$ and use the definition of $\widehat{V}$ to obtain
\begin{align*}
S(\xi)^{1/2}\widehat{V}_{t}+i|\xi|S(\xi)A(\xi)S^{-1/2}(\xi)\widehat{V}+S(\xi)B(\xi)S^{-1/2}(\xi)\widehat{V}=0,
\end{align*}
where the positive definiteness of $S(\xi)$ was used. After multiplying the last equation by $S^{-1/2}(\xi)$, \eqref{Fouriersyst2} is obtained. Since we can write
\begin{align*}
A_{S}(\xi)=S^{-1/2}(\xi)S(\xi)A(\xi)S^{-1/2}(\xi),\quad\text{and}\quad B_{S}(\xi)=S^{-1/2}(\xi)S(\xi)B(\xi)S^{-1/2}(\xi),
\end{align*}
it follows that both $A_{S}(\xi)$ and $B_{S}(\xi)$ are symmetric for every $\xi\in\mathbb{R}^{d}$. By Lemma \ref{expendable1} and the fact that the map 
\[
\left\{
\begin{aligned}
\mathcal{J}&:GL_{n}\rightarrow GL_{n}\\
&\quad\varphi~\mapsto~\varphi^{-1},
\end{aligned}
\right.
\]
is of class $C^{\infty}$ (see Lemma 2.5.5 in \cite{AMR88}), it follows that $A_{S}$ and $B_{S}$ are smooth functions of $\xi\in\mathbb{R}^{d}\setminus\{0\}$.

Now, assume that \eqref{linsyst} is genuinely coupled and the existence of $\xi_{0}\in\mathbb{R}^{d}\setminus\{0\}$,  $\mu_{0}\in\mathbb{R}$ and $\psi_{0}\in\mathbb{R}^{n}$ such that 
\begin{align*}
\left(\mu_{0} \bbI_n +A_{S}(\xi_{0})\right)\psi_{0}=0\quad\mbox{and}\quad \psi_{0}\in\ker B_{S}(\xi_{0}).
\end{align*}
Then, $S^{-1/2}(\xi_{0})\psi_{0}\in\ker B(\xi_{0})$ and is an eigenvector of $A(\xi_{0})$ with eigenvalue $-\mu_{0}$, a contradiction. The converse assertion follows by using the same argument.

Next consider a non-trivial solution $\widehat{V}$ to the problem
\begin{align}
\label{spectralprob2}
\left(\lambda \bbI_n +i|\xi|A_{S}(\xi)+B_{S}(\xi)\right)\widehat{V}=0,
\end{align}
and notice that $\widehat{V}$ is an eigenvector of $i|\xi|A_{S}(\xi)+B_{S}(\xi)$ with corresponding eigenvalue $-\lambda$. Then, $\widehat{U}=S^{-1/2}(\xi)\widehat{V}$ is a non-trivial solution to \eqref{spectralprob} with the same eigenvalue. Therefore, the strict dissipativity condition is equivalent for systems \eqref{spectralprob} and \eqref{spectralprob2}.

Finally, assume that $K_{S}\in C^{\infty}(\mathbb{R}^{d};\mathbb{R}^{n\times n})$ satisfies conditions (a) and (b) in Definition \ref{Matrixsymbol} for the symmetric symbols $A_{S}(\xi)$ and $B_{S}(\xi)$ in \eqref{Fouriersyst2} with $S(\xi) = \bbI_n$ for all $\xi\in\mathbb{R}^{d}$. In particular,
\begin{align*}
\left[K_{S}(\xi)A_{S}(\xi)\right]^{s}+B_{S}(\xi)\geq\theta(\xi) \bbI_n \quad\mbox{for all}~\xi\in\mathbb{R}^{d}\setminus\{0\},~\mbox{and some}~\theta=\theta(\xi)>0.
\end{align*} 
By \eqref{compmatrix}, we have that 
\begin{align*}
K(\xi)S(\xi)A(\xi)=S^{1/2}(\xi)K_{S}(\xi)A_{S}(\xi)S^{1/2}(\xi),
\end{align*}
and thus, 
\begin{align*}
\left[K(\xi)S(\xi)A(\xi)\right]^{s}=S^{1/2}(\xi)\left[K_{S}(\xi)A_{S}(\xi)\right]^{s}S^{1/2}(\xi).
\end{align*}
Therefore, $K(\xi)$ is a compensating matrix symbol for the triplet $(S,A,B)$.
\end{proof}
\subsection{The equivalence theorem}

The main result of the paper is the following equivalence theorem which relates the strict dissipativity, the genuine coupling and the existence of a compensating matrix symbol for linear systems of the form \eqref{linsyst} in several space dimensions. The equivalence theorem reads as follows.

\begin{theorem}[equivalence theorem]
\label{equivalencetheo}
Let us assume that system \eqref{linsyst} is symbol symmetrizable. The following conditions are equivalent:
\begin{itemize}
	\item [(i)] \eqref{linsyst} is stricly dissipative.
	\item [(ii)] \eqref{linsyst} is genuinely coupled.
	\item [(iii)] There exists a compensating matrix symbol for the triplet $(S,A,B)$.
\end{itemize}
\end{theorem}

\begin{remark}
Lemma \ref{hypequiv} tells us that in the case where system \eqref{Fouriersyst} (or, equivalently, system \eqref{linsyst}) is symbol symmetrizable, we can go back and forth by a smooth change of variable between the symmetric system 
\[
S(\xi)\hU_{t} + i \vert \xi \vert S(\xi)A(\xi)\hU + S(\xi)B(\xi)\hU = 0,
\]
and the one given by \eqref{Fouriersyst2}. In addition, verifying the genuine coupling condition or constructing a compensating matrix symbol for either of these systems is equivalent to doing so for the other one. Therefore we prove the equivalence Theorem \ref{equivalencetheo} for the symmetric system \eqref{Fouriersyst2} only, without loss of generality.
\end{remark}

Prior to proving the equivalence theorem we need some technical results. In what follows we enunciate some basic results from Linear Algebra, as well as some technical lemmata, which are used in the proof of the equivalence theorem. We gloss over most of the proofs and refer the reader to Humpherys \cite{Hu05} for more details (see also Ellis and Pinsky \cite{EllPin75a,EllPin75b}). Let $M_n$ denote the set of $n \times n$ matrices over $\C$. Let us take $A \in M_{n}$ and define the commutator operator $\mbox{Ad}_{A}$ on $M_{n}$ as
\[
\mbox{Ad}_{A}(X) = [A,X]:=AX-XA.
\]

\begin{lemma}\label{LemConsK}
Let $A\in M_{n}$ be semi-simple, that is, $A$ in diagonalizable. Suppose that the spectral resolution of $A$ is given by $A= \lambda_{1} P_{1} + \cdots + \lambda_{r} P_{r}$, where $\lambda_{1},\ldots, \lambda_{r}$ are the distinct eigenvalues of $A$ with corresponding eigenprojections $P_{1},\ldots,P_{r}$, respectively. Let us define the following linear operator on $M_{n}$:
\begin{equation}\label{PiA}
\Pi_{A}(X) := \sum_{j=1}^{r}P_{j}XP_{j}.
\end{equation}
Then the following hold:
\begin{itemize}
\item[(i)] $\Pi_{A}$ is the projection onto $\cN(\mbox{Ad}_{A})$ along $\cR(\mathrm{Ad}_{A})$.
\item[(ii)] For each $B \in M_{n}$ there exists $K\in M_{n}$ such that
\[
B = \Pi_{A}(B) + \mathrm{Ad}_{A}(K).
\]
Furthermore, the canonical solution $K$, that we call the compensating matrix, is given by the Drazin inverse or reduced resolvent of the commutator operator: 
\begin{equation}\label{ExK}
K = \sum_{i \neq j} \frac{P_{i}BP_{j}}{\lambda_{i}-\lambda_{j}}.
\end{equation}
\end{itemize}
\end{lemma}
\begin{proof}
See, e.g., Humpherys \cite{Hu05}, Lemma 4.3 (see also Ellis and Pinsky \cite{EllPin75b}, Propositions 4.3 and 4.4, and Lemma 2.2 in Shizuta and Kawashima \cite{ShKa85}).
\end{proof}
\begin{lemma}\label{LemProK}
Let us assume that $A,B \in M_{n}$ are Hermitian and $B\geq 0$. Then
\begin{itemize}
\item[(i)] $\Pi_{A}(B)$ is Hermitian.
\item[(ii)] $\Pi_{A}(B) \geq 0$.
\item[(iii)] The compensating matrix $K$ given in \eqref{ExK} is skew-Hermitian. 
\end{itemize}
\end{lemma}
\begin{proof}
See, e.g., Humpherys \cite{Hu05}, Lemma 4.4 or Ellis and Pinsky \cite{EllPin75b}, \S4.
\end{proof}

Now, let us take $\xi_{0} \in \R^{d}\setminus \{ 0\}$, and let $\lambda_{1}=\lambda_{1}(\xi_{0}),\ldots,\lambda_{r}=\lambda_{r}(\xi_{0})$ be the distinct eigenvalues of $A_{S}(\xi_{0})$ with multiplicity $m_{1},\ldots,m_{r}$, respectively. Let $\mathcal{N}_{0}(\xi_{0})$ be a neighborhood of $\xi_{0}$ such that for each $j=1,\ldots,r$ the eigenvalues associated to the eigenvalue $\lambda_{j}$, which we call the $\lambda_{j}-$group, remain in a neighbourhood of it, so that all of them can be enclosed with a positively oriented circle $\Gamma_{j}$ that excludes the eigenvalues of the $\lambda_{k}-$group for $k\neq j$.  Let us define
\begin{equation}\label{Pj}
P_{S}^{j}(\xi) := \frac{1}{2 \pi i} \int_{\Gamma_{j}} (z - A_{S}(\xi) )^{-1}dz,
\end{equation}
for all $\xi\in \mathcal{N}_{0}(\xi_0)$ and $j=1,\ldots,r$, which is the total projection onto the eigenspace generated by the eigenvalues belonging to the $\lambda_{j}-$group. 

We define $\widetilde{A}_{S}(\xi)$ by
\begin{equation}\label{hA}
\widetilde{A}_{S}(\xi) := \sum_{j=1}^{r} \lambda_{j} P_{S}^{j}(\xi),\quad \xi\in \mathcal{N}_{0}(\xi_{0}).
\end{equation}
Let us observe that in general $A_{S}(\xi)\neq \widetilde{A}_{S}(\xi)$ and $A_{S}(\xi_{0})=\widetilde{A}_{S}(\xi_0)$. Also, the matrix $\widetilde{A}_{S}(\xi)$ is symmetric and of constant multiplicity in $\xi \in\mathcal{N}_{0}(\xi_0)$. Then we have the following result.
\begin{lemma}
\label{GenCNewTr}
Suppose the genuinely coupling condition holds for system \eqref{Fouriersyst2} and let $\xi_{0}\in\mathbb{R}^{d}\setminus\{0\}$. Then there exists a neighborhood $\mathcal{N}_{1}(\xi_{0})\subset\mathcal{N}_{0}(\xi_{0})$ with the following properties: for each  $\psi\in\mathbb{R}^{n}\setminus\{0\}$ such that $\psi\in\ker B_{S}(\xi)$ for some $\xi\in\mathcal{N}_{1}(\xi_{0})$, then $\mu\psi+\widetilde{A}_{S}(\xi)\psi\neq 0$ for all $\mu\in\mathbb{R}$. 
\end{lemma}
\begin{proof}
We proceed by contradiction. If the conclusion is false then we can choose a sequence of neighborhoods $\{\cN_{j}(\xi_{0})\}_{j\in\mathbb{N}}$ which converges to $\xi_{0}$ and such that, there are $\xi_{j}\in\mathcal{N}_{j}(\xi_{0})$, $\psi_{j}\in\mathbb{S}^{n-1}$ and $\mu_{j}\in\mathbb{R}$ satisfying that, 
\begin{align}
B_{S}(\xi_{j})\psi_{j}=0\quad\mbox{and}\quad\mu_{j}\psi_{j}+\widetilde{A}_{S}(\xi_{j})\psi_{j}=0\quad\mbox{for all}\quad j\in\mathbb{N}.\label{impossible}
\end{align}
By \eqref{hA}, $-\mu_{j}$ belongs to the finite set $\{\lambda_{1},..,\lambda_{r}\}$. Hence, up to a subsequence, we can assume that $-\mu_{j}=\lambda_{k}$ for some fixed $k\in\{1,..,r\}$. Moreover, since $\{\psi_{j}\}\subset\mathbb{S}^{n-1}$ we can assure the existence of $\psi_{0}\in\mathbb{S}^{n-1}$ such that $\psi_{j}\rightarrow\psi_{0}$ when $j\rightarrow\infty$.
By Lemma \ref{hypequiv} and the smoothness of $P_{S}^{j}(\xi)$ on $\mathcal{N}_{0}(\xi_{0})$, it follows that 
\begin{align*}
B_{S}(\xi_{0})\psi_{0}=0,\quad-\lambda_{k}\psi_{0}+A_{S}(\xi_{0})\psi_{0}=0,
\end{align*}
where we used that $\widetilde{A}_{S}(\xi_{0})=A_{S}(\xi_{0})$. Since $|\psi_{0}|=1$, this contradicts the genuinely coupling of \eqref{Fouriersyst2}. This concludes the proof.
\end{proof}

\begin{lemma}\label{PosPi}
Let $\mathcal{N}_{1}(\xi_0)$ be the neighborhood of Lemma \ref{GenCNewTr}. Then the triplet $(I, \widetilde{A}_{S}(\xi), B_{S}(\xi))$ satisfies the genuine coupling condition for $\xi \in \mathcal{N}_{1}(\xi_0)$ if and only if 
\begin{equation}\label{hThe}
\widetilde{\theta}(\xi):= \inf_{\Vert x \Vert=1} \sum_{j=1}^{r} \langle P_{S}^{j}(\xi)x, B(\xi) P_{S}^{j}(\xi)x \rangle
\end{equation}
is positive for all $\xi \in\mathcal{N}_{1}(\xi_0)$.
\end{lemma}
\begin{proof}
The proof follows that of Lemma 3.2 in Humpherys \cite{Hu05}, word by word. 
\end{proof}

\subsection{Proof of Theorem \ref{equivalencetheo}}
We show the following implications
\begin{align*}
\mbox{(i)} \Rightarrow \mbox{(ii)} \Rightarrow \mbox{(iii)} \Rightarrow \mbox{(i)}.
\end{align*}
$\mathbf{(i)\Rightarrow(ii)}$. Let us assume that the system \eqref{linsyst} is strictly dissipative and that there exists $\xi \in \R^{d} \setminus \{0 \}$ such that for some $\mu \in \R$ and $\psi \in \ker (B_{S}\xi))$, there holds $A_{S}(\xi) \psi = -\mu \psi$. Then we have
\[
\big( -i \vert \xi \vert \mu + i \vert \xi \vert A_{S}(\xi) + B_{S}(\xi) \big)\psi = 0.
\]
The above formula tells us that $\psi\neq 0$ is a solution of the eigenvalue problem \eqref{spectralprob} with $\lambda(\xi) = -i \vert \xi \vert \mu$, from which we obtain that $\Re \lambda (\xi) = \Re (-i \vert \xi \vert \mu)  =0$. This contradicts, in turn, the strict dissipativity condition (i).
\vspace{0.3cm}\\
$\mathbf{(ii)\Rightarrow(iii)}$ Take $\xi_{0} \in \R^{d}\setminus \{ 0 \}$, and let $\widetilde{A}(\xi)$ given by \eqref{hA} for $\xi \in\mathcal{N}_{1}(\xi_{0})$, the neighbourhood as in Lemma \ref{GenCNewTr}. Lemmata \ref{LemConsK} and \ref{LemProK} imply the existence of a skew-symmetric Hermitian matrix $K_{\xi_0}(\xi)$  defined on $\mathcal{N}_{1}(\xi_0)$ satisfying
\[
B_{S}(\xi)= \Pi_{\widetilde{A}_{S}(\xi)}(B_{S}(\xi)) + \mbox{Ad}_{\widetilde{A}_{S}(\xi)}(K_{\xi_0}(\xi)), 
\]
with $\Pi_{\widetilde{A}_{S}(\xi)}(B_{S}(\xi))$ Hermitian and positive semi-definite. We now invoke Lemmata \ref{GenCNewTr} and \ref{PosPi} to conclude that $\Pi_{\widetilde{A}_{S}(\xi)}(B_{S}(\xi)) \geq \widetilde{\theta}(\xi) I_n$, with $\widetilde{\theta}(\xi)>0$ for $\xi \in \mathcal{N}_{1}(\xi_{0})$. Observe that for $\xi \in \mathcal{N}_{1}(\xi)$ we have
\[
\begin{aligned}
\big[ K_{\xi_{0}}(\xi) \tiA_{S}(\xi) \big]^{s} + B_{S}(\xi) & \geq \big[ K_{\xi_{0}}(\xi)\tiA_{S}(\xi) \big]^{s} + \frac{1}{2}B_{S}(\xi) \\ & = \frac{1}{2}\left( K_{\xi_{0}}(\xi)\tiA_{S}(\xi) + \left(K_{\xi_{0}}(\xi)\tiA_{S}(\xi) \right)^{\top} \right)+ \frac{1}{2}B_{S}(\xi) \\ &= \frac{1}{2}\left(K_{\xi_{0}}(\xi)\tiA_{S}(\xi) -\tiA_{S}(\xi)K_{\xi_{0}}(\xi) \right)+ \frac{1}{2}B_{S}(\xi) \\ &=  \frac{1}{2}\left(-\text{Ad}_{\tiA_{S}(\xi)}(K_{\xi_{0}}(\xi))+ B_{S}(\xi)  \right) \\&= \frac{1}{2}\Pi_{\tiA_{S}(\xi)}(B_{S}(\xi)).
\end{aligned}
\]
As each $P_{S}^{j}$ is a $C^{\infty}$ function in $\mathcal{N}_{1}(\xi_{0})$, the compensating matrix symbol given by \eqref{ExK} is a $C^{\infty}(\mathcal{N}_{1}(\xi_{0});\mathbb{R}^{n\times n})$ function.

Now, let us remember that $\widetilde{A}_{S}(\xi_{0})=A_{S}(\xi_{0})$, so that 
\[
[ K_{\xi_0}(\xi_0), A_{S}(\xi_0) ] + B_{S}(\xi_0) \geq \widetilde{\theta}(\xi_0)I_n,
\]  
with $\widetilde{\theta}(\xi_0)>0$. As $K_{\xi_0}(\xi)$, $A_{S}(\xi)$, and $B_{S}(\xi)$ are $C^{\infty}(\mathcal{N}_{1}(\xi_{0});\mathbb{R}^{n\times n})$ functions, we can take a neighborhood $\mathcal{N}_{2}(\xi_{0})\subset\mathcal{N}_{1}(\xi_{0})$ such that positive-definiteness of the matrix $[ K_{\xi_0}(\xi), A_{S}(\xi) ] + B_{S}(\xi)$ still holds for $\xi \in \mathcal{N}_{2}(\xi_{0})$. Moreover, if $\mu_{i}(\xi)$ for $1 \leq i \leq n$, denotes the eigenvalues of $[ K_{\xi_0}(\xi), A_{S}(\xi) ] + B_{S}(\xi)$, we can write
\begin{equation}\label{LocPosit}
	[K_{\xi_0}(\xi), A_{S}(\xi) ] + B_{S}(\xi) \geq \min_{1 \leq i \leq n} \mu_{i}(\xi) I_n =: \theta_{\xi_0}(\xi) I_n,
\end{equation}
with $\theta_{\xi_0}\in C(\mathcal{N}_{2}(\xi_{0}))$ (see Texier \cite{Tex18}, Proposition 1.1 and Remark 4.3) and $\theta_{\xi_0}(\xi) >0$ for $\xi \in \mathcal{N}_{2}(\xi_{0})$.

Therefore we have shown that for each $w \in \R^{d}\setminus \{ 0  \}$ there exists a neighborhood $\mathcal{N}(w)$ of it, such that a  real skew-Hermitian matrix valued  $C^{\infty}(\mathcal{N}(w))$ function $K_{w}(\xi)$ exists, which satisfies \eqref{LocPosit} for some continuous function $\theta_{w}(\xi) >0$ for all $\xi \in\mathcal{N}(w)$.
We consider the open cover $ \left\lbrace \mathcal{N}(w) \right\rbrace_{w \in \R^{d}\setminus \{ 0 \}}$ of $\R^{d} \setminus \{0 \}$, and take $\left\lbrace \phi_{w} \right\rbrace_{w \in \R^{d}\setminus \{ 0 \}}$ a locally finite smooth partition of unity on $\R^{d} \setminus \{ 0 \}$ subordinated to this cover. We define 
\begin{equation}\label{globalK}
	K_{S}(\xi) := \sum_{w \in \R^{d}\setminus \{ 0 \}} \phi_{w}(\xi)K_{w}(\xi),
\end{equation} 
and
\begin{equation}\label{global-thet}
	\theta(\xi) := \sum_{w \in \R^{d} \setminus \{ 0 \}} \phi_{w}(\xi)\theta_{w}(\xi),
\end{equation}
which are $C^{\infty}(\R^{d}\setminus \{ 0 \};\mathbb{R}^{n\times n})$ and $C(\R^{d} \setminus \{0 \})$ functions, respectively, satisfying
\begin{equation}\label{global-posit}
[K(\xi)A_{S}(\xi) ]^{s} + B_{S}(\xi) \geq \frac{1}{2}\big( [K(\xi),A_{S}(\xi)] + B_{S}(\xi) \big) \geq  \frac{\theta(\xi)}{2} I_n, 
\end{equation}
with $\theta(\xi) >0$ for all $\xi \in \R^{d} \setminus \{ 0 \}$.
\vspace{0.3cm}\\
$\mathbf{(iii)\Rightarrow(i)}$. 
Let $\lambda \in \C$ and $v\in \C^{n}$, $v \neq 0$, satisfying 
\begin{equation}
\label{spectralprob2}
\big(\lambda I_n + i |\xi| A_S(\xi) + B_S(\xi) \big) v = 0.
\end{equation}
Taking the inner product in $\C^{n}$ of \eqref{spectralprob2} with $v$ and taking taking the real part of the resulting expression, we obtain the following standard Friedrichs-type estimate: 
\begin{equation}\label{spec-est-1}
	(\Re \lambda) \vert v \vert^{2} + \langle v, B_{S}(\xi) v \rangle =0.
\end{equation}
As $B(\xi) \geq 0$, the last estimate implies
\begin{equation}\label{spec-est-2}
	(\Re \lambda) \vert v \vert^{2} + \frac{1}{\vert B_{S}(\xi) \vert}\vert B_{S}(\xi)v \vert^{2} \leq 0. 
\end{equation}
Observe that both \eqref{spec-est-1} and \eqref{spec-est-2} imply that $\Re \lambda \leq 0$. For the second estimate, multiply \eqref{spectralprob2} by $i \vert \xi \vert K_{S}(\xi)$ and take the inner product with to get
\[
\vert \xi \vert^{2} \langle v, K_{S}(\xi)A_{S}(\xi) v \rangle = \lambda \vert \xi \vert \langle v, iK_{S}(\xi) v \rangle + \vert \xi \vert \langle v, iK_{S}(\xi) B_{S}(\xi) v \rangle. 
\]
Next, using the fact that $iK_{S}(\xi)$ is Hermitian, after taking the real part we obtain
\[
\begin{aligned}
	\vert \xi \vert^{2} \langle v, [K_{S}(\xi), A_{S}(\xi)]v \rangle &= 2 (\Re \lambda) \vert \xi \vert  \langle v, iK_{S}(\xi)v \rangle + \vert \xi \vert \langle v, (iK_{S}(\xi)B_{S}(\xi) + iB_{S}(\xi)K_{S}(\xi))v \rangle \\ &\leq 2 \vert \Re \lambda \vert \vert \xi \vert \vert K_{S}(\xi) \vert \vert v \vert^{2} + \frac{\theta(\xi)}{2}\vert \xi \vert^{2} \vert v \vert^{2} + 2 \frac{\vert K_{S}(\xi) \vert^{2}}{\theta(\xi)} \vert B_{S}(\xi) v \vert^{2}.
\end{aligned}.
\]
Here we have used the inequality $\vert \langle a, b \rangle \vert \leq \tfrac{1}{2}(\epsilon^{2}\vert a \vert^{2} +\tfrac{1}{\epsilon^{2}} \vert b \vert^{2} )$, for any $\epsilon > 0$. Now, multiply \eqref{spec-est-1} by $\vert \xi\vert^{2}  $ and add it to the above inequality to get
\[
(\Re \lambda) \vert \xi \vert^{2} \vert v \vert^{2} + \frac{3 \theta(\xi)}{4} \vert \xi \vert^{2}\vert v \vert^{2} + (\Re \lambda) \vert \xi \vert \vert K_{S}(\xi) \vert \vert v \vert^2 \leq \frac{\vert K_{S}(\xi) \vert^{2}}{\theta(\xi)} \vert B_{S}(\xi)v \vert^{2},
\] 
where the property (b) of the compensating matrix symbol matrix has been used. Using \eqref{spec-est-2} to estimate the left-hand side, we obtain
\[
(\Re \lambda)\vert \xi \vert^{2} \vert v \vert^{2} + \frac{3 \theta(\xi)}{4}\vert \xi \vert^{2}\vert v \vert^{2} + (\Re \lambda) \vert \xi \vert \vert K_{S}(\xi) \vert \vert v \vert^2 +  (\Re \lambda) \frac{\vert K_{S}(\xi) \vert^{2}}{\theta(\xi)}\vert B_{S}(\xi) \vert \vert v \vert^{2} \leq 0,
\]
which yields
\begin{equation}\label{est-Re(lamb)}
	\Re \lambda \leq  \frac{-3\vert \xi \vert^{2} \theta(\xi)^{2}}{4 \vert \xi \vert \vert K_{S}(\xi) \vert \theta(\xi) + 4 \vert \xi \vert^{2}\theta(\xi) + 4 \vert K_{S}(\xi)\vert^{2} \vert B_{S}(\xi) \vert}.
\end{equation}
Therefore, as $\theta(\xi) >0$ for $\xi \neq 0$, we obtain $\Re \lambda < 0$ for $\xi\neq 0$. 
\qed

\begin{remark}\label{comp-mat-for-rmk}
Let us point out that, in order to drop out the constant multiplicity hypothesis on the generalized flux matrix $A_{S}(\xi)$ during the proof of the implication (ii) $\Rightarrow$ (iii), we have conveniently merged the proofs of Shizuta and Kawashima \cite{ShKa85} and Humpherys \cite{Hu05}. In the case when the matrix $A_{S}(\xi) $ is of constant multiplicity and its eigenvalues and eigenprojections are smooth functions of $\xi$, the proof of Humpherys still works. In such a case there is an explicit representation for the eigenvalues and for the eigenprojections and, consequently, one can try to find the compensating matrix symbol using the formula provided by Humpherys (see equation \eqref{ExK}). 

This method for computing the compensating matrix function, however, may not be the most efficient as it involves computing the eigenvalues and eigenprojections of the symbol $A(\xi)$, which is not a trivial task. Another method for computing the compensating matrix symbol is by inspection, which does not require $A_{S}(\xi)$ to be of constant multiplicity. In one space dimension, the property of $A_{S}(\xi)$ being symmetric in $\xi \in \R$ implies that its eigenvalues and eigenprojections are smooth (indeed, analytic) in $\xi$, which is not the case when the family $A_{S}(\xi)$ depends on two or more variables (see Theorem II.6.1 and Remark II.6.3 in Kato \cite{Kat80}, p. 120). If $A_{S}(\xi)$ is of constant multiplicity in $\xi \in \R^{d}$, $d\geq 2$, by using the smooth representation (see \cite{Tex18}, Proposition 2.1) of its eigenprojection with the Dunford-Taylor integral \eqref{Pj} and a partition of unity argument, we can conclude that its eigenprojections are smooth with respect to $\xi \in \R^{d}$, $\xi \neq 0$. However, this is not the same for its eigenvalues as we can only expect them to be locally Lipschitz continuous in $\xi$ (see Texier \cite{Tex18}, Remark 4.3).  
\end{remark}

\begin{remark}
\label{rembarB}
When the system \eqref{linsyst} does not have a relaxation term (that is, a zero order derivative term) we can write the generalized viscosity matrix as $B_{S}(\xi)= \vert \xi \vert^{2} \bB_{S}(\xi)$ for some matrix $\bB_{S}(\xi)$. In this case, computing a compensating matrix symbol $K_{S}(\xi)$ for the triplet $(I, A_{S}(\xi), B_{S}(\xi))$ is equivalent to do it for the triplet $(I, A_{S}(\xi),\bB_{S}(\xi))$. Indeed, if $\bK_{S}(\xi)$ is a compensating matrix for the triplet $(I, A_{S}(\xi), \bB_{S}(\xi))$ then there holds
\[
[ \bK_{S}(\xi)A_{S}(\xi) ]^{s} + \bB_{S}(\xi) \geq \bthe (\xi) I_n,
\]
for all $\xi \in \R^{d}\setminus \{ 0 \}$, with $\bthe(\xi) >0$. Then the matrix $K_{S}(\xi) := \vert \xi  \vert^{2} \bK_{S}(\xi)$ is a compensating matrix symbol for the triplet $(I, A_{S}(\xi), B_{S}(\xi))$. This can be seen by just multiplying the expression above by $\vert \xi \vert^{2}$ to get
\[
\big[ \vert \xi \vert^{2} \bK_{S}(\xi)A_{S}(\xi) \big]^{s} + \vert \xi \vert^{2}\bB_{S}(\xi) = [ K_{S}(\xi)A_{S}(\xi) ]^{s} + B_{S}(\xi)  \geq \vert \xi \vert^{2} \bthe(\xi) I_n =: \theta(\xi) I_n.
\]
\end{remark}

\section{Pointwise energy estimates in Fourier space}
\label{secenergy}

In this section we perform pointwise energy estimates in the Fourier space for the solutions of linear systems of the form \eqref{linsyst} which are symbol symmetrizable and which satisfy the hypotheses of the equivalence theorem. In contrast to the case of second order systems \cite{KaSh88,ShKa85}, in order to close these estimates in this more general setting, we need to assume some conditions on the generalized flux, the generalized viscous matrices and on the compensating matrix symbols. These conditions may seem somehow chosen \emph{ad hoc}; but they are, however, satisfied by the examples we consider next. Furthermore, they are satisfied by second order (viscous and relaxation) systems to which the classical Shizuta and Kawashima's equivalence theorem applies. 

Let us start by considering symbol symmetrizable systems written in the variable $V$ as in Lemmas \ref{hypequiv}, that is, in the form
\begin{equation}\label{sys-V}
\hV_{t}+\big( i \vert \xi \vert A_{S}(\xi) + B_{S}(\xi) \big)\hV = 0.
\end{equation}
Then we have the following result.
\begin{lemma}[pointwise estimate, general case]
\label{pnt-ee-full-lem}
Let us consider the symmetric linear system \eqref{sys-V}. Assume that there exists a compensating matrix symbol $K_{S}(\xi)$ for the triplet $(I,A_{S}(\xi),B_{S}(\xi))$, that is, for some positive smooth function $f(\xi)$, $\xi \in \R^{d}$, $\xi \neq 0$, there holds
\begin{equation}\label{pos-1}
[ K_{S}(\xi)A_{S}(\xi) ]^{s} + B_{S}(\xi) \geq f(\xi) I_n.
\end{equation}
Also, suppose that the matrices $K_{S}(\xi)$ and $B_{S}(\xi)$ satisfy
\begin{equation}
\label{hyp-enr-est}
\frac{\vert  \xi K_{S}(\xi) \vert}{1+\vert \xi \vert^{2}},\quad \frac{\vert K_{S}(\xi)B_{S}(\xi)^{1/2} \vert}{(1+\vert \xi \vert^{2})^{1/2} f(\xi)^{1/2}} \leq C,
\end{equation}
for all $\xi \in \R^{d}$, $\xi \neq 0$, and for some uniform positive constant $C$. Then, the solutions $\hV= \hV(\xi,t)$ to system \eqref{sys-V} satisfy the following estimate
\begin{equation}\label{dsp-str-i}
\vert \hV(\xi,t) \vert \leq C \exp \left( -\frac{k \vert \xi \vert^{2}f(\xi)}{1+\vert \xi \vert^{2}}t \right) \vert \hV(\xi,0)\vert, 
\end{equation} 
for all $\xi \in \R^{d}\setminus \{ 0 \}$, $t\geq 0$, and some uniform constants $C$, $k>0$. 
\end{lemma}

\begin{proof}
Let us take the inner product in $\C^{n}$ of \eqref{sys-V} with $\hV$, and then take the real part to obtain
\begin{equation}
\label{la8}
\tfrac{1}{2} \partial_t |\hV|^2 +  \< \hV, B_{S} \hV \> = 0.
\end{equation}
Next, multiplying \eqref{sys-V} by $-i\vert \xi \vert K_{S}(\xi)$ and then taking the inner product of the result with $\hV$, leads to 
\begin{equation}
\label{la9}
- \< \hV, i \vert \xi \vert K_{S} \hV_t \> +\vert \xi \vert^2 \< \hV, K_{S} A_{S} \hV \> = i \vert \xi \vert  \< \hV,   K_{S} B_{S} \hV \> = 0.
\end{equation}
Since $K_{S}(\xi)$ is skew-symmetric we have
\[
\Re \< \hV, i \vert \xi \vert K_{S} \hV_t \> = \tfrac{1}{2} \vert \xi \vert \partial_t \< \hV, i K_{S} \hV \>.
\]
Then taking the real part of \eqref{la9} we obtain
\[
- \tfrac{1}{2} \vert  \xi \vert \partial_t \< \hV, i K_{S} \hV \> + \vert \xi \vert^2 \< \hV, [K_{S} A_{S}]^s \hV \> = \Re \big( i \vert \xi \vert \< \hV, K_{S} B_{S} \hV \> \big).
\]
In view that $B_{S}(\xi) \geq 0$, and by the bounds imposed on the norm of the matrix product $K_{S}(\xi)B_{S}(\xi)$ described on the assumptions, we arrive at the estimate
\begin{equation}
\label{la10}
- \tfrac{1}{2} \vert \xi \vert \partial_t \< \hV, i K_{S} \hV \> + \vert \xi \vert^2 \< \hV, [K_{S} A_{S}]^{s} \hV \> \leq \vep  \vert \xi \vert^2 f(\xi) |\hV|^2 + C_\vep \big( \< \hV, B_{S} \hV \>+   \vert \xi \vert^2 \< \hV, B_{S} \hV \> \big),
\end{equation}
for any $\vep>0$ and where $C_{\vep}>0$ is an uniform constant depending on $\vep$ and on the bound for the norm of $K_{S}(\xi)B_{S}(\xi)^{1/2}/((1+\vert \xi \vert^{2})^{1/2} f(\xi)^{1/2})$. 

Now multiply \eqref{la8} by $1+\vert \xi \vert^{2}$ and  \eqref{la10} by $\delta$, and add them up in order to obtain
\[
\begin{aligned}
\tfrac{1}{2} \partial_t [&(1+\vert \xi \vert^{2}) |\hV|^2 - \delta \vert \xi \vert \< \hV, i K_{S} \hV \> ] + (1-\delta C_{\vep}) \< \hV, B_{S} \hV \> + \\ & +\vert \xi \vert^2 [ \delta \< \hV, [K_{S} A_{S}]^{s} \hV \> + (1-\delta C_\vep)  \< \hV, B_{S} \hV \> ] \\  &\leq \vep \delta \vert\xi \vert^2 f(\xi) |\hV|^2.
\end{aligned}
\]
Let us define the energy 
\[
\cE := |\hV|^2 - \frac{\delta \vert \xi \vert}{1+\vert \xi \vert^{2}} \< \hV, i K_{S} \hV \>.
\]
It is easy to verify that the quantity $\cE$ is real in view that $K_S(\xi)$ is skew-symmetric. Using that $\vert \vert \xi \vert K_{S}(\xi)/(1+\vert \xi \vert^{2}) \vert$ is uniformly bounded in $\xi \in \R^{d}\setminus \{ 0 \}$, we can find $\delta_{0}>0$ such that for $0 <\delta < \delta_{0}$ there holds
\[
\frac{1}{C_{1}}\vert \hV \vert^{2} \leq \cE \leq C_{1}\vert \hV \vert^{2},
\]	
for some uniform constant $C_{1}>0$. 

Let us choose $\vep=1/2$, so that $C_{\vep}$ is fixed. Next, take $0 < \delta < \delta_{0}$ sufficiently small such that $\delta= \min \{ \delta, 1-\delta C_{\vep} \}$. Therefore, we have
\[
\delta \< \hV, [K_{S} A_{S}]^{s} \hV \> + (1-\delta C_\vep)\<\hV, B_{S} \hV \> \geq \delta f(\xi) |\hV|^2.
\]
Thus we arrive at
\[
\tfrac{1}{2} \partial_t \cE + \tfrac{1}{2} \frac{\delta  \vert \xi \vert^2 f(\xi) }{1+\vert \xi \vert^{2}}  |\hV|^2 +\frac{(1-\delta C_{\vep})}{1+\vert \xi \vert^{2}} \< \hV, B_{S} \hV \> \leq 0,
\]
which yields
\[
\partial_t \cE + k \frac{\vert \xi \vert^2 f(\xi)}{1+\vert \xi \vert^{2}} \cE \leq 0,
\]
where $k=\delta /C_{1}$. This implies the result. 
\end{proof}

\begin{lemma}[pointwise estimate, relaxation-free case]
\label{pnt-ee-rel-fre-lem}
In the case that the linear system \eqref{linsyst} does not contain a relaxation term, so that the system \eqref{sys-V} can be written as
\begin{equation}\label{sys-V-no-rel}
\hV_{t}+\big( i \vert \xi \vert A_{S}(\xi) + \vert \xi \vert^{2} \bB_{S}(\xi) \big)\hV = 0
\end{equation}
(see Remark \ref{rembarB} above), let us assume that there exist a compensating matrix function $\bK_{S}(\xi)$ for the triplet $(I, A_{S}(\xi), \bB_{S}(\xi))$, that is, there holds
\begin{equation}\label{pos-2}
[ \bK_{S}(\xi), A_{S}(\xi) ]^{s} + \bB_{S}(\xi) \geq g (\xi) I_n,
\end{equation}
for all $\xi \in \R^{d}$, $\xi \neq 0$, and some smooth positive function $g(\xi)$. In addition, suppose that 
\begin{equation}
\label{hyp-enr-est-2}
\vert \xi\bK_{S}(\xi) \vert,\quad \frac{\vert \xi  \bK_{S}(\xi)\bB_{S}(\xi)^{1/2}\vert}{g(\xi)^{1/2}} \leq C,
\end{equation} 
for all $\xi \in \R^{d}$, $\xi \neq 0$, and for some uniform positive constant $C$. Then, the solutions $\hV(\xi,t)$ to the system \eqref{sys-V-no-rel} satisfy the following estimate
\begin{equation}\label{dsp-str-ii}
\vert \hV(\xi,t) \vert \leq C \exp ( -k \vert \xi \vert^{2}g(\xi)t ) \vert \hV(\xi,0)\vert,
\end{equation}
for all $\xi \in \R^{d}$, $\xi \neq 0$, and $t \geq 0$, and $C$ and $k$ some uniform constants. 
\end{lemma}
\begin{proof}
The proof is very similar to that of Lemma \ref{pnt-ee-full-lem}. Hence, we only sketch it, paying attention to the steps where hypothesis \eqref{hyp-enr-est-2} plays a role. The details are left to the reader. If we consider system \eqref{sys-V-no-rel}, equations \eqref{la8} and \eqref{la9} become
\begin{equation}
\label{la8-2}
\tfrac{1}{2} \partial_t |\hV|^2 + \vert \xi \vert^{2} \< \hV, \bB_{S} \hV \> = 0,
\end{equation}
and 
\begin{equation}
\label{la9-2}
- \< \hV, i \vert \xi \vert \bK_{S} \hV_t \> +\vert \xi \vert^2 \< \hV, \bK_{S} A_{S} \hV \> = i \vert \xi \vert^{3} \< \hV,   \bK_{S}\bB_{S} \hV \> = 0,
\end{equation}
respectively. Then, inequality \eqref{la10} now reads
\begin{equation}
\label{la10-2}
- \tfrac{1}{2} \vert \xi \vert \partial_t \< \hV, i \bK_{S} \hV \> + \vert \xi \vert^2 \< \hV, [\bK_{S} A_{S}]^{s} \hV \> \leq \vep  \vert \xi \vert^2 g(\xi) |\hV|^2 + C_\vep \vert \xi \vert^2 \< \hV, \tiB_{S} \hV \> ,
\end{equation}
for every $\vep$, and $C_{\vep}$ is a positive constant depending on $\vep$ and the uniform bound for the matrix $\vert \xi \vert \bK_{S}(\xi)\bB_{S}(\xi)^{1/2}/g(\xi)^{1/2}$. 

Next, we multiply \eqref{la10-2} by $\delta>0$ and add the result to \eqref{la8-2}, and then define the energy 
\[
\cE_{1} := |\hV|^2 - \delta \vert \xi \vert \< \hV, i \bK_{S} \hV \>, 
\]
which is indeed an energy in view that $\vert \xi \bK_{S}(\xi) \vert $ is uniformly bounded in $\xi \in \R^{d}\setminus \{ 0 \}$.  The proof is then completed in the same fashion as that of Lemma \ref{pnt-ee-full-lem}. 
\end{proof}
\begin{remark}
\label{remkawasi}
Next we show that second order systems with viscosity and relaxation considered under the classical framework of Shizuta and Kawashima \cite{ShKa85}, satisfy the conditions \eqref{hyp-enr-est}. For that, let us consider a symmetric second order system in the Fourier space of the form
\[
\hU_{t} + \left( i \vert \xi \vert A(\omega) + L + \vert \xi \vert^{2} B(\omega) \right) \hU=0,
\]
where $L\geq 0$ and $B(\omega) \geq 0$ for all $\omega \in \bbS^{d-1}$. Assume that there exists a compensating matrix  $K(\omega) \in C^{\infty}\big(\bbS^{d-1};\R^{n\times n} \big)$ for the system, that is, there holds
\begin{equation}\label{cls-K}
[K(\omega) A(\omega) ]^{s} + L + B(\omega) \geq \bsi I,\quad \forall \omega \in \bbS^{d-1},  
\end{equation}
for some uniform positive constant $\bsi$. 
Consider the open cover $(0,1+\epsilon)$, $(1,\infty)$ of the set $(0,\infty)$, for some $0 < \epsilon \ll 1$ fixed. Let $\phi_{1}$ and $\phi_{2}$ be a smooth partition of unity of the set $(0,\infty)$ subordinated to this cover, that is, $\phi_{1}$ and $\phi_{2}$ are smooth functions such that $0 \leq \phi_{1},\phi_{2} \leq 1$, $\text{supp}\,\phi_{1} \subset(0, 1+\epsilon)$, $\text{supp}\, \phi_{2} \subset (1, \infty)$ and $\phi_{1}(x) + \phi_{2}(x) = 1$ for all $x \in (0, \infty)$. We define 
\begin{equation}\label{new-K}
\tiK(\xi) := c \big( \vert \xi \vert^{2}\phi_{1}(\vert \xi \vert) + \phi_{2}(\vert \xi \vert) \big) K(\omega(\xi)), \quad \forall \xi \in \R^{d}\setminus \{ 0 \},
\end{equation} 
for some positive constant $0 < c \leq 1$ such that
\begin{equation}\label{c-def}
1 \geq  c \big( \vert \xi \vert^{2}\phi_{1}(\vert \xi \vert) + \phi_{2}(\vert \xi \vert) \big), \quad \text{for} \, \vert \xi \vert \geq 1.
\end{equation}
Then for $\vert \xi \vert \geq 1$ we have 
\[
\begin{aligned}
\big[\tiK(\xi)A(\omega) \big]^{s} +L + \vert \xi \vert^{2} B(\omega) &= c \big( \vert \xi \vert^{2}\phi_{1}(\vert \xi \vert) + \phi_{2}(\vert \xi \vert) \big) [ K(\omega)A(\omega)]^{s} + L + \vert \xi \vert^{2} B(\omega) \\ & \geq c \big( \vert \xi \vert^{2}\phi_{1}(\vert \xi \vert) + \phi_{2}(\vert \xi \vert) \big)[K(\omega)A(\omega)]^{s} + L + B(\omega) \\ & \geq c \big( \vert \xi \vert^{2}\phi_{1}(\vert \xi \vert) + \phi_{2}(\vert \xi \vert) \big) \left( [K(\omega)A(\omega)]^{s} + L + B(\omega) \right) \\ &\geq \bsi c \big( \vert \xi \vert^{2}\phi_{1}(\vert \xi \vert) + \phi_{2}(\vert \xi \vert) \big)I_n.
\end{aligned}
\]
Now, given that $\text{supp}\, \phi_{2} \subset (1, \infty)$, we have for $0 < \vert \xi \vert \leq 1$ that
\[
\begin{aligned}
\big[\tiK(\xi)A(\omega) \big]^{s} +L + \vert \xi \vert^{2} B(\omega) &= c \vert \xi \vert^{2}\phi_{1}(\vert \xi \vert)[ K(\omega)A(\omega)]^{s} + L + \vert \xi \vert^{2} B(\omega) \\ & \geq c \vert \xi \vert^{2}\phi_{1}(\vert \xi \vert)[ K(\omega)A(\omega)]^{s} + \vert \xi \vert^{2} L + \vert \xi \vert^{2} B(\omega) \\ &\geq c \vert \xi \vert^{2} \phi_{1}(\vert \xi \vert) \left( [K(\omega)A(\omega)]^{s} + L + B(\omega) \right) \\
&\geq \bsi c |\xi|^2 \phi_{1}(\vert \xi \vert)I_n,
\end{aligned}   
\]
where we have used that $c \leq 1$. 

Thus we have shown that the compensating matrix (symbol) defined by \eqref{new-K} satisfies
\[
\big[ \tiK(\xi)A(\omega) \big]^{s} + L + \vert \xi \vert^{2}B(\omega) \geq f(\xi)I_n,
\]
for $f(\xi)$ a continuous, positive and uniformly bounded function on $\R^{d-1}$ given by
\begin{equation}\label{pos-tiK}
f(\xi) = \begin{cases} \bsi c  \big( \vert \xi \vert^{2}\phi_{1}(\vert \xi \vert) + \phi_{2}(\vert \xi \vert) \big), & \text{for}\quad \vert \xi \vert \geq 1, \\
\bsi c \vert \xi \vert^{2} \phi_{1}(\vert \xi \vert), & \text{for}\quad 0< \vert \xi \vert < 1.
\end{cases}
\end{equation}
Observe that the function $\vert \xi \vert^{2} \phi_{1}(\vert \xi \vert) + \phi_{2}(\vert \xi \vert)$ is uniformly bounded in $\xi\in \R^{d}$, $\xi \neq 0$, so that 
\[
\frac{\vert \xi \vert \tiK(\xi)}{1+\vert \xi \vert^{2}} \leq C,
\]
for some uniform positive constant $C$.
In this setting the generalized viscous matrix is $\tiB(\xi) := L + \vert \xi \vert^{2}B(\omega)$, from which one easily obtains
\[
\vert \tiB(\xi)^{1/2} \vert = O((1+ \vert \xi \vert^{2})^{1/2}).
\]
Thus it is easy to see that for $\vert \xi \vert \geq 1$ there holds
\[
\frac{\vert \tiK(\xi)\tiB(\xi)^{1/2}\vert}{(1+\vert \xi \vert^{2})f(\xi)^{1/2}} \leq \frac{C \vert K(\omega) \vert (1+\vert \xi \vert^{2})^{1/2}}{(1+\vert \xi \vert^{2})^{1/2}} = O(1),
\]
whereas for $0 < \vert \xi \vert < 1$ we have
\[
\frac{\vert \tiK(\xi)\tiB(\xi)\vert}{(1+\vert \xi \vert^{2})^{1/2}f(\xi)^{1/2}} = \frac{c \vert \xi \vert^{2}\phi_{1}(\vert \xi \vert)K(\omega) \tiB(\xi)^{1/2}\vert}{(1+\vert \xi \vert^{2})^{1/2}f(\xi)^{1/2}} \leq C\frac{\vert \xi \vert^{2}\vert K(\omega)\vert (1+\vert \xi \vert^{2})^{1/2}}{(1+\vert \xi \vert^{2})^{1/2}\vert \xi \vert} = O(\vert \xi \vert). 
\]
Thus the estimate
\[
\frac{\vert \tiK(\xi) \bB(\xi)^{1/2}\vert}{(1+\vert \xi \vert)^{1/2}f(\xi)^{1/2}} \leq C
\]
holds uniformly in $\xi \in \R^{d}\setminus \{ 0 \}$.  This proves the assertion.
\end{remark}
%
%
	
\section{Applications}
\label{secappl}

In this Section, we examine (under the framework and hypotheses of Theorem \ref{equivalencetheo}) several multi-dimensional systems arising in Physics. After each system is linearized around a constant state, we investigate its symmetrizability (in the classical or in the symbol sense). For the resulting symmetric system the genuinely coupling condition (see Definition \ref{gencoupling}) is verified and a compensating matrix symbol is provided. For the latter, we can proceed either by using the formula provided by Humpherys (see equation \eqref{eq:humpcompensation}), or by inspection. 

Regarding this last point, a few observations are in order.  As we have already pointed out in Remark \ref{comp-mat-for-rmk}, in order to use the formula for the compensating matrix symbol we need the generalized flux matrix $A_{S}(\xi)$ to be of constant multiplicity and with explicit representations of its eigenvalues and eigenprojections. These last representations, however, may not be available.  In all of the examples treated here, $A_{S}(\xi)$ is of constant multiplicity and, thus (see M\'etivier and Zumbrun \cite[Definition 2.2, Remark 2.3]{MeZ05}), for each $\bar{\xi} \in \R^{d}\setminus \{ 0 \}$ we can find an open neighborhood  $\Omega$ of it and a smooth family of linearly independent functions $\{e_{j}(\xi) \}_{j=1}^{m}$ defined on $\Omega$ and such  that,
\begin{align*}
A_{S}(\xi)e_{j}(\xi)=\mu(\xi)e_{j}(\xi)\quad\mbox{for all}\quad\xi\in\Omega,
\end{align*}
where $\mu=\mu(\xi)$ is a smooth eigenvalue of $A_{S}(\xi)$ on $\Omega$ with algebraic multiplicity equal to $m$. Thus, by a partition of unity argument we obtain that the eigenvalues of $A_{S}(\xi)$ are smooth on $\xi \in \R^{d}\setminus \{ 0 \}$. The same can be done for the eigenprojections (see Remark \ref{comp-mat-for-rmk}). However, an \emph{explicit global expression} for them may not be possible as the example in Section \ref{isoQHD} shows (such a global representation would contradict the Hairy Ball Theorem; see Remark \ref{remhairyball} for details).

\subsection{Isothermal compressible viscous fluids of Korteweg type}
\label{sec-NSK}

As our first example, we consider the following system describing the dynamics of a compressible isothermal fluid exhibiting viscosity and capillarity in several space dimensions,
\begin{equation}
\label{Iso-KModel}
\begin{aligned}
\partial_{t}\rho + \nabla \cdot (\rho \bfu) &= 0, \\
\partial_{t}(\rho \bfu) + \nabla \cdot (\rho \bfu \otimes \bfu) + \nabla p &= \nabla \cdot \big(\mathbf{S+K} \big), 
\end{aligned}
\end{equation}	
where $x \in \R^d$, $d \geq 1$, $t>0$; $\rho$ and $\bfu \in \R^d$ represent the mass density and the velocity field, respectively. According to custom, $p$ is the thermodynamical  pressure, a smooth positive function of the mass density, $p = p(\rho)$. The viscous stress tensor $\mathbf{S}$ and the Korteweg tensor $\mathbf{K}$ are given by
\begin{equation}
\label{S,K}
\begin{aligned}
\mathbf{S} &=  \lambda (\nabla \cdot \bfu) I_d + 2 \nu \mathbf{D}(\bfu) , \\
\mathbf{K} &= \big( \rho \nabla \cdot (k \nabla \rho) + \tfrac{1}{2}k\vert \nabla \rho \vert^{2} - \tfrac{1}{2}\rho k_{\rho}\vert \nabla \rho \vert^{2} \big) I_d -k ( \nabla \rho \otimes \nabla \rho ). 
\end{aligned}
\end{equation} 
Here $\mathbf{D}(\bfu)= \tfrac{1}{2}\big(\nabla \bfu+(\nabla \bfu)^{\top} \big)$ is the strain tensor and $I_d$ is the identity $d \times d$ matrix. The bulk and shear viscosity coefficients ($\lambda$ and $\nu$, respectively) are smooth positive functions of $\rho >0$, satisfying $2\nu + \lambda >0$. The capillarity coefficient $k$ (which appears in $\mathbf{K}$) is a smooth positive function of $\rho$ as well. 

System \eqref{Iso-KModel}, also known as the \emph{Navier-Stokes-Korteweg (NSK) system} \cite{CLS19,TaGa16,PlV22,KSX22,GeLeF16}, is the isothermal version of the model derived by Dunn and Serrin \cite{DS85} to describe fluid capillarity effects in diffuse interfaces for liquid vapor flows (and originally suggested in an early contribution by Korteweg \cite{Kortw1901}; see also Section \ref{sec-NSFK} below). Here we are assuming the general case in which the nonlinear viscosity and capillarity coefficients are positive functions of $\rho > 0$. 
Hence, we consider the mass density $\rho > 0$ as the sole independent thermodynamic variable. It belongs to the set
\[
\mathcal{D}:=\{ \rho \in \R:\: \rho>0\}.
\]
In what follows we assume that the pressure satisfies 
\begin{equation}
p>0,\quad p_{\rho}>0.
\end{equation}

We rewrite system \eqref{ModelSev} in quasilinear form. Without loss of generality and for the sake of simplicity, we consider the case of two space dimensions ($d = 2$). The results hold for (and can be easily extrapolated to) any dimension $d \geq 2$ at the expense of extra bookkeeping. To that end, we proceed in the same fashion as for the Navier-Stokes system: use the continuity equation to simplify the momentum balance equation. 
Thus, after some straightforward computations, the quasi-linear form of system \eqref{Iso-KModel} can be written as 
\begin{equation}
\label{quas-mat-form-NSK}
\begin{aligned}
A^{0}(U)\partial_{t}U + \sum_{i=1}^2 A^{i}(U)\partial_{x_i}U &= \sum_{i,j=1}^2 B^{ij}(U)\partial_{x_i}\partial_{x_j}U + \sum_{i,j,k=1}^2 C^{ijk}(U)\partial_{x_i}\partial_{x_j}\partial_{x_k}U + \\ &+  \mathcal{N}(U,DU,D^{2}U), 
\end{aligned}
\end{equation}
where 
\[
U = \begin{pmatrix}
\rho \\ \bfu
\end{pmatrix} \in \R^3, \quad
A^{0}(U)= \begin{pmatrix}
1 & 0_{1 \times 2} \\ 0_{2 \times 1} & \rho \: I_2
\end{pmatrix} \in \R^{3 \times 3}
\]
(here all vectors are column vectors, $\xi = (\xi_1, \xi_2)^\top$ and $\boldsymbol{u} = (u_1, u_2)^\top$). The remaining matrix coefficients can be given in terms of their symbols: 
\[
\sum_{i=1}^{2}A^{i}(U)\xi_{i}= \begin{pmatrix}
\bfu \cdot \xi & \rho \:\xi^\top \\ p_{\rho}\:\xi & \rho\:(\bfu \cdot\xi) \: I_2
\end{pmatrix} \in \R^{3 \times 3},
\]
\[
\sum_{i,j=1}^{2}B^{ij}\xi_{i}\xi_{j}=\begin{pmatrix}
0 & 0_{1 \times 2} \\ 0_{2 \times 1} & \nu \vert \xi \vert^{2}I_2 + (\nu+\lambda)\xi \otimes\xi
\end{pmatrix} \in \R^{3 \times 3},
\]
\[
\sum_{i,j,k=1}^{2}C^{ijk}\xi_{i}\xi_{j}\xi_{k}= \begin{pmatrix}
0 & 0_{1 \times 2} \\ k\:\rho \vert \xi \vert^{2}\xi & 0_2 
\end{pmatrix} \in \R^{3 \times 3},
\]
for all $\xi = (\xi_1,\xi_2)^\top \in \R^2$. The non-linear terms are contained in $\mathcal{N}(U,DU,D^{2}U)$. We are not going to write their explicit form as we are going to consider system \eqref{quas-mat-form-NSK} linearized around a constant equilibrium state. 
Let $\bU=(\brho,\bu_{1}, \bu_{2})^\top$ be a constant equilibrium state. Linearizing around $\bU$ the system takes the form
\begin{equation}
\label{NSK-2D}
\begin{aligned}
A^{0}\partial_{t} U+A^{1}\partial_{x} U+A^{2}\partial_{y} U&= B^{11}\partial_{x}^{2} U+B^{12}\partial_{y}\partial_{x} U+B^{21}\partial_{x}\partial_{y}U+B^{22}\partial_{y}^{2}U \\ &+C^{111}\partial_{x}^{3} U + C^{221}\partial_{x}\partial_{y}^{2} U + C^{112}\partial_{y}\partial_{x}^{2} U+ C^{222}\partial_{y}^{3} U,
\end{aligned}
\end{equation}
where $U=(\rho, u_{1}, u_{2})^{\top}$. The matrix coefficients are evaluated at the constant state $\bU$, that is, $A^{0}=A^{0}(\bU)$, $A^{1}=A^{1}(\bU)$, etc. The form of the matrix coefficients is given next:  
\begin{equation}
\label{f-A0}
A^{0}=\begin{pmatrix}
1 & 0 & 0 \\ 0 & \brho & 0 \\ 0 & 0 & \brho
\end{pmatrix}, 
\end{equation}
\[
\begin{aligned}
A^{1} &= \begin{pmatrix}
\bu_{1} & \brho & 0 \\ \bp_{\rho} & \brho\: \bu_{1} & 0 \\ 0 & 0 & \brho \: \bu_{1}
\end{pmatrix}, & A^{2} &= \begin{pmatrix}
\bu_{2} & 0 & \brho \\ 0 & \brho\: \bu_{2} & 0 \\ \bp_{\rho} & 0 & \brho\: \bu_{2}
\end{pmatrix},  \\
B^{11} &= \begin{pmatrix}
0 & 0 & 0 \\ 0 & 2\bnu + \bla & 0 \\ 0 & 0 & \bnu
\end{pmatrix}, & B^{12}&= \begin{pmatrix}
0 & 0 & 0 \\ 0 & 0 & \bnu \\ 0 & \bla & 0 
\end{pmatrix},\\
B^{21} &= \begin{pmatrix}
0 & 0 & 0 \\ 0 & 0 & \bla \\ 0 & \bnu & 0
\end{pmatrix}, 
& B^{22} &= \begin{pmatrix}
0 & 0 & 0 \\ 0 & \bnu & 0 \\ 0 & 0 & 2 \bnu + \bla
\end{pmatrix}, \\
C^{111} = C^{221} &=\begin{pmatrix}
0 & 0 & 0 \\ \bk \: \brho & 0 & 0 \\ 0 & 0 & 0 
\end{pmatrix}, & C^{112} =C^{222} &=\begin{pmatrix}
0 & 0 & 0 \\ 0 & 0 & 0 \\ \bk \: \brho & 0 & 0
\end{pmatrix},
\end{aligned}
\]
%
%
%
where all the overlined variables (for example, $\bp_{\rho}=p_{\rho}(\brho)$, $\bnu=\nu(\brho)$ and $\bla = \lambda(\brho)$) denote the physical quantities evaluated at the constant state $\bU$. We keep this notation for the rest of the paper.

By rewriting \eqref{NSK-2D} in the form \eqref{linsyst} we obtain
\begin{equation}
\label{NSK-2D-comp-form}
A^{0}U_{t}+ \sum_{0\leq \vert \alpha \vert \leq 3}L^{\alpha}D^{\alpha}U=0,
\end{equation}
where $L^{(0,0)}=0$, $L^{(1,0)}=A^{1}$, $L^{(0,1)}=A^{2}$, $L^{(2,0)}=-B^{11}$, $L^{(1,1)}=-(B^{12}+B^{21})$, $L^{(0,2)}=-B^{22}$, $L^{(3,0)}=-C^{111}$, $L^{(2,1)}=-C^{112}$, $L^{(1,2)}=-C^{221}$, $L^{(0,3)}=-C^{222}$. Therefore after taking the Fourier transform to \eqref{NSK-2D-comp-form} and splitting the symbol into odd- and even-order derivatives we are led to the system
\begin{equation}\label{NSK-2D-Four}
A^{0}\hU_{t} + \big( i\vert \xi \vert A(\xi) + B(\xi) \big)\hU=0,
\end{equation}
where the generalized transport and generalized viscosity symbols, $A(\xi)$ and $B(\xi)$, respectively are, according to \eqref{gen-flux-visc}, given by
\[
A(\xi) = \sum_{\substack{1 \leq |\alpha| \leq m \\ |\alpha| \, \text{odd}}} |\xi|^{ \vert \alpha \vert - 1} A^ \alpha (\omega(\xi)) = \begin{pmatrix}
\omega_{1}\bu_{1} + \omega_{2}\bu_{2} & \omega_{1}\brho & \omega_{2}\brho \\ \omega_{1}\beta(\xi) & \omega_{1}\bu_{1} + \omega_{2}\bu_{2} & 0 \\ \omega_{2}\beta(\xi) & 0 & \omega_{1}\bu_{1} + \omega_{2}\bu_{2}
\end{pmatrix},
\]
with
\[
\beta(\xi):= \bp_{\rho}+\bk\:\brho \vert \xi \vert^{2} > 0, \qquad \forall \xi \in \R^2,
\]
and by
\[
B(\xi) = \sum_{\substack{0 \leq |\alpha| \leq m \\ |\alpha| \, \text{even}}} |\xi|^{\vert \alpha \vert} B^\alpha(\omega(\xi)) = \begin{pmatrix}
0 & 0 & 0 \\ 0 & (2 \bnu + \bla)\omega_{1}^{2} + \bnu \omega_{2}^{2} & (\bla + \bnu)\omega_{1}\omega_{2} \\ 0 & (\bla +\bnu)\omega_{1}\omega_{2} & \bnu \omega_{1}^{2} + (2\bnu + \bla)\omega_{2}^{2}
\end{pmatrix}
\]
(let us recall that $\omega = \omega(\xi) = \xi/|\xi| \in \bbS^1$).
Thanks to the conditions on the viscosity coefficients $\bnu$ and $\bla$, one can easily check that $B(\xi) \geq 0 $ for all $\xi \in \R^{2}\setminus \{0 \}$.

\begin{proposition}
\label{propNSK}
The linearized NSK system \eqref{NSK-2D-comp-form} satisfies the hypotheses of the equivalence Theorem \ref{equivalencetheo} (in particular, it is genuinely coupled).
\end{proposition}

\begin{remark}
\label{remNSKnotF}
Before we proceed, let us make a comment on the symmetrizability in the sense of Friedrichs of system \eqref{NSK-2D}. Let us assume that there exists such a classical symmetrizer having the form
\[
S=\begin{pmatrix}
s_{11} & s_{12} & s_{13}  \\ s_{12} & s_{22} & s_{23} \\ s_{13} & s_{23} & s_{33}
\end{pmatrix}.
\]
Then, by requiring the matrices $SC^{111}$ and $SC^{222}$ to be symmetric one gets $s_{22}=s_{23}=s_{33}=0$. Thus the second and third columns of $S$ are linearly dependent, so that $S$ can not be positive definite. Therefore no such symmetrizer exists, implying that \emph{system \eqref{NSK-2D} is not symmetrizable in the sense of Friedrichs}. 
\end{remark}
\begin{proof}[Proof of Proposition \ref{propNSK}]
If we look at the system in the Fourier space, namely, \eqref{NSK-2D}, we note that it is not in symmetric form as $A(\xi)$ is not symmetric. However, by inspection it is easy to verify that 
\begin{equation}\label{form-S(xi)-NSK}
S(\xi):= \begin{pmatrix}
\beta(\xi)/\brho & 0 & 0 \\ 0 & 1 & 0 \\ 0 & 0 & 1 
\end{pmatrix}
\end{equation}
is a symbol symmetrizer for system \eqref{NSK-2D-Four}. Let us recast system \eqref{NSK-2D-Four} in the variable $\hV:=(S(\xi)A^{0})^{1/2}\hU$. This yields,
\begin{equation}
\label{NSK-2D-V}
\hV_{t}+ \big( i\vert \xi \vert A_{S}(\xi) + \vert \xi \vert^{2} \bB_{S}(\xi) \big)\hV =0,
\end{equation}
where 
\begin{equation}\label{NSK-2D-tiA}
\begin{aligned}
A_{S}(\xi) &:= \big(S(\xi)A^{0} \big)^{-1/2}S(\xi)A(\xi)\big(S(\xi)A^{0} \big)^{-1/2} \\ &= \begin{pmatrix}
\omega_{1}\bu_{1}+ \omega_{2}\bu_{2} & \omega_{1}\beta^{1/2} & \omega_{2}\beta^{1/2} \\ \omega_{1}\beta^{1/2} & \omega_{1}\bu_{1} + \omega_{2}\bu_{2} & 0 \\ \omega_{2}\beta^{1/2} & 0 & \omega_{1}\bu_{1} + \omega_{2}\bu_{2}
\end{pmatrix},
\end{aligned}
\end{equation}
and
\begin{equation}\label{NSK-2D-tiB}
\begin{aligned}
\vert \xi \vert^{2} \bB_{S}(\xi) &:= \big(S(\xi)A^{0} \big)^{-1/2}S(\xi)B(\xi)\big(S(\xi)A^{0} \big)^{-1/2} \\ &= 
\frac{\vert \xi \vert^{2}}{\brho}\begin{pmatrix}
0 & 0 & 0 \\ 0 & (2 \bnu + \bla)\omega_{1}^{2} + \bnu \omega_{2}^{2} & (\bla + \bnu)\omega_{1}\omega_{2} \\ 0 & (\bla +\bnu)\omega_{1}\omega_{2} & \bnu \omega_{1}^{2} + (2\bnu + \bla)\omega_{2}^{2}
\end{pmatrix}.
\end{aligned}
\end{equation}
Clearly $\bB_{S}(\xi)\geq 0$ and $\mbox{ker}\big( \bB(\xi) \big)=\mbox{span}\left\lbrace (1,0,0) \right\rbrace$ for all $\xi \in \R^{2}\setminus \{0 \}$. With this fact and the property that $\beta(\xi) >0$ for all $\xi \in \R^{2}\setminus \{0 \}$, it easily follows that $A_{S}(\xi)$ and $\bB_{S}(\xi)$ satisfy the genuine coupling condition. 
Therefore, by the equivalence Theorem \ref{equivalencetheo}, a compensating matrix symbol exists for the triplet $\big( I, A_{S}(\xi), \bB_{S}(\xi) \big)$. Let us try to find one of such matrices by means of formula \eqref{ExK}.  

Computing the eigenvalues of the generalized flux matrix $A_{S}(\xi)$ yields,
\[
\mu_{1}(\xi) = \omega_{1}\bu_{1} + \omega_{2}\bu_{2}-\beta(\xi)^{1/2} < \mu_{2}(\xi) = \omega_{1}\bu_{1} + \omega_{2}\bu_{2} < \mu_{3}(\xi)= \omega_{1}\bu_{1} + \omega_{2}\bu_{2} + \beta(\xi)^{1/2},
\]
so that $A_{S}(\xi)$ is of constant multiplicity in $\xi \in \R^{2}\setminus \{0\}$. The respective eigenvectors of the eigenvalues of $A_{S}(\xi)$ are 
\[
v_{1}(\xi)= (-1, \omega_{1}, \omega_{2})^{\top},\quad v_{2}(\xi)=(0,-\omega_{2},\omega_{1})^{\top},\quad v_{3}(\xi)= (1, \omega_{1}, \omega_{2})^{\top}, 
\]
which are orthogonal. Therefore, the eigenprojections onto the eigenspaces associated to $\mu_{j}(\xi)$, $j =1,2,3$, are given by
\[
P_{1}(\xi) = \frac{1}{|v_1|^2}\begin{pmatrix}
-v_{1}^\top \\  \omega_{1} v_{1}^\top \\ \omega_{2}v_{1}^\top
\end{pmatrix} = \frac{1}{2}\begin{pmatrix} 1 & - \omega_1 & - \omega_2 \\ - \omega_1 & \omega_1^2 & \omega_1 \omega_2 \\ - \omega_2 & \omega_1 \omega_2 & \omega_2^2\end{pmatrix},
\]
\[
P_2(\xi) = \frac{1}{|v_2|^2}\begin{pmatrix}
0_{1 \times 3} \\ -\omega_{2}v_{2}^\top \\ \omega_{1}v_{2}^\top
\end{pmatrix} = \begin{pmatrix} 0 & 0 & 0 \\ 0 & \omega_2^2 & - \omega_1 \omega_2 \\ 0 & - \omega_1 \omega_2 & \omega_2^2 \end{pmatrix},
\]
\[
P_3(\xi) = \frac{1}{|v_3|^2}\begin{pmatrix}
v_{3}^\top \\  \omega_{1} v_{3}^\top \\ \omega_{2}v_{3}^\top
\end{pmatrix} = \frac{1}{2}\begin{pmatrix} 1 & \omega_1 & \omega_2 \\ \omega_1 & \omega_1^2 & \omega_1 \omega_2 \\ \omega_2 & \omega_1 \omega_2 & \omega_2^2\end{pmatrix},
\]
respectively, where we have used that $\vert v_{1} \vert^{2}=\vert v_{3} \vert^{2}=2$ and $\vert v_{2}\vert^{2}=1$.
%
Notice that the eigenprojections are smooth functions of $\xi\in \R^{2}\setminus \{0\}$. This observation and the fact that $A_{S}(\xi)$ is of constant multiplicity enable us to use formula \eqref{ExK} to construct the compensating matrix function, that is, 
\[
\begin{aligned}
\bK_{S}(\xi)&= \sum_{i\neq j}\frac{P_{i}(\xi)\bB_{S}(\xi)P_{j}(\xi)}{\mu_{i}(\xi)-\mu_{j}(\xi)} \\ &= \frac{1}{4 \beta(\xi)^{1/2}\brho}\begin{pmatrix}
0 & (2 \bnu + \bla)\omega_{1} & (2 \bnu + \bla)\omega_{2} \\ -(2 \bnu + \bla)\omega_{1} & 0 & 0 \\ -(2 \bnu + \bla)\omega_{2} & 0 & 0 
\end{pmatrix}.
\end{aligned}
\]
Let us recall that, by definition, $\beta(\xi)=O(\vert \xi \vert^{2})$ as $\vert \xi \vert \to \infty$, so that $\bK_{S}(\xi)$ as well as $\vert \xi \vert \bK_{S}(\xi)$ are uniformly bounded in $\xi \in \R^{2}$, $\xi \neq 0$. This compensating matrix function satisfies 
\[
\begin{aligned}
\big[\bK_{S}(\xi)A_{S}(\xi) \big]^{s} &+ \bB_{S}(\xi) \geq \frac{1}{2}\left( \big[\bK_{S}(\xi),A_{S}(\xi) \big] + \bB_{S}(\xi) \right) \\ &= \frac{1}{4 \brho} \begin{pmatrix}
(2 \bnu + \bla)(\omega_{1}^{2}+\omega_{2}^{2}) & 0 & 0 \\ 0 & (2 \bnu + \bla)\omega_{1}^{2} + 2\bnu \omega_{2}^{2} & \bla \omega_{1}\omega_{2} \\ 0 & \bla \omega_{1}\omega_{2} & 2\bnu \omega_{1}^{2}+ (2 \bnu + \bla)\omega_{2}^{2} 
\end{pmatrix} \\ &\geq \bthe I_3,
\end{aligned}
\]
for some uniform constant $\bthe>0$. The proposition is now proved.
\end{proof}
We apply Lemma \ref{pnt-ee-rel-fre-lem} to obtain the pointwise energy estimate for the two-dimensional linear NSK system \eqref{NSK-2D}.

\begin{lemma}[basic energy estimate for the NSK system]
\label{lembee-NSK-2D}
The solutions $\hV = \hV(\xi,t)$ to system \eqref{NSK-2D-V} satisfy the estimate
\begin{equation}
\label{bestV}
|\hV(\xi,t)| \leq C \exp (-c \vert \xi \vert^2 t) |\hV(\xi,0)|,
\end{equation}
for all $\xi \in \R^{2}$, $\xi \neq 0$, $t \geq 0$ and some uniform constants $C,c > 0$.
\end{lemma}
\begin{proof}
The compensating matrix symbol $\bK_{S}(\xi)$ computed above satisfies condition \eqref{pos-2} in Lemma \ref{pnt-ee-rel-fre-lem} with $g(\xi) \equiv \bthe > 0$, constant. Let us observe that in \eqref{NSK-2D-V}, $\bB_{S}$ is a matrix depending on $\omega\in \bbS^{1}$, thus its norm is bounded. Thus using that $\vert \xi \bK_{S}(\xi) \vert$ is uniformly bounded in $\xi \in \R^{2}$, $\xi \neq 0$, one can easily verify that the conditions \eqref{hyp-enr-est-2} of Lemmas \ref{pnt-ee-rel-fre-lem} are satisfied, namely, 
\[
\vert \xi \bK_{S}(\xi) \vert,\, \frac{\vert \xi  \bK_{S}(\xi) \bB_{S}(\xi)^{1/2} \vert}{g(\xi)^{1/2}} \leq C,
\]
for some positive constant $C$ uniform in $\xi \in \R^{2} \setminus \{ 0 \}$. Hence the result follows. 
\end{proof}

An immediate result is the following Corollary.
\begin{corollary}
The Navier-Stokes-Korteweg system \eqref{Iso-KModel} is strictly dissipative of regularity gain type.
\end{corollary}
\begin{proof}
Follows from the decay rate in \eqref{bestV} and Remark \ref{Rmr-reg-typ}.
\end{proof}

Another consequence of the basic energy estimate \eqref{bestV} and the definition of $\hV$ is the following appropriate decay for the original variables, yielding the right regularity for the density and for the velocity field.
\begin{lemma}
\label{lemprelimdecay}
The solutions $\hU(\xi,t) = \big(\hU_1(\xi,t), \hU_2(\xi,t), \hU_{3}(\xi, t) \big)^\top$ to system \eqref{NSK-2D-Four} satisfy the estimate
\begin{equation}
\label{bestU}
\begin{aligned}
& \left( (1+ \vert \xi \vert^2) |\hU_1(\xi,t)|^2 + |\hU_2(\xi,t)|^2 +| \hU_{3}(\xi, t) |^{2} \right)^{1/2}  \leq \\ & \qquad \qquad \leq C \exp ( -c \, \xi^2 t ) \left( (1+\vert \xi\vert^2) |\hU_1(\xi,0)|^2 + |\hU_2(\xi,0)|^2 +| \hU_{3}(\xi,0 ) |^{2} \right)^{1/2},
\end{aligned}
\end{equation}
for all $\xi \in \R^{2} \setminus \{ 0 \}$, $t \geq 0$ and some uniform constants $C$, $c > 0$.
\end{lemma}
\begin{proof}
Remember that $\hV(\xi,t)=\big( S(\xi)A^{0}\big)^{1/2}\hU$. Next, by the expressions for $A^{0}$ and $S(\xi)$ given, respectively, by \eqref{f-A0} and \eqref{form-S(xi)-NSK}, it is easy to see that $\big( S(\xi) A^{0} \big)^{1/2}$ is a diagonal matrix whose first element is the only one depending on $\xi$ and it is of order $\beta^{1/2}$. As $\beta( \xi ) = \bp_{\rho} +  \bk \: \brho \: \vert \xi \vert^{2} > 0$, one can find $C_{1}$, $c_{1}>0$ such that $c_{1}(1+\vert \xi \vert^{2})\leq \beta(\xi) \leq C_{1} (1+\vert \xi \vert^{2}) $. The result follows from these observations.
\end{proof}

The pointwise energy estimate \eqref{bestU} in the Fourier space implies the decay structure of the solutions of the linear NSK system \eqref{NSK-2D}. This is the content of the next result. 
\begin{lemma}
\label{lin-dec-U}
Let us assume that $U=(U_{1},U_{2},U_{3})^{\top}$ is a solution of the linear NSK system \eqref{NSK-2D} with initial condition $U(x,0)\in \big( H^{s+1}(\R^{2})\times \big( H^{s}(\R^{2}) \big)^{2} \big)\cap \big( L^{1}(\R^{2}) \big)^{3}$ for $s\geq 0$. Then for each $0\leq \ell \leq s $ the estimate
\begin{equation}
\label{ln-dec-U-est}
\begin{aligned}
\big( \Vert \partial_{x}^{\ell} U_{1}(t) \Vert_{1}^{2} &+ \Vert \partial_{x}^{\ell} U_{2}(t) \Vert_{0}^{2} + \Vert \partial_{x}^{\ell} U_{3}(t) \Vert_{0}^{2}\big)^{1/2} \leq \\ &\leq  Ce^{-c_{1}t} \big( \Vert \partial_{x}^{\ell}U_{1}(0) \Vert_{1}^{2}+ \Vert \partial_{x}^{\ell}U_{2}(0) \Vert_{0}^{2} + \Vert \partial_{x}^{\ell}U_{3}(0) \Vert_{0}^{2} \big)^{1/2}+ \\ & \quad + C(1+t )^{-(1/2+\ell/2)} \Vert U(0)\Vert_{L^{1}},
\end{aligned}
\end{equation}
holds for all $t\geq 0$ and some uniform positive constants $C$, $c_{1}$. 
\end{lemma}
\begin{proof}
Fix $l$ satisfying $0 \leq \ell \leq s$, multiply \eqref{bestU} by $\vert \xi \vert^{2 \ell}$ and integrate in $\xi \in \R^{2}$ to obtain
\begin{equation}
\label{est-J1J2}
\int_{\R^{2}} \vert \xi \vert^{2\ell} \big[ (1+\vert \xi \vert^{2})\hU_{1}(\xi,t)\vert^{2} + \vert \hU_{2}(\xi,t) \vert^{2} + \vert \hU_{3}(\xi,t) \vert^{2} \big]\: d\xi \leq C\big( J_{1}(t) +J_{2}(t) \big),
\end{equation}
where 
\[
\begin{aligned}
J_{1}(t):&=\int_{\vert \xi \vert\leq 1}\vert \xi \vert^{2\ell} \big[ (1+\vert \xi \vert^{2})\hU_{1}(\xi,0)\vert^{2} + \vert \hU_{2}(\xi,0) \vert^{2} + \vert \hU_{3}(\xi,0) \vert^{2}\big]\exp(-2c \vert \xi \vert^{2}t) \: d\xi, \\
J_{2}(t):&=\int_{\vert \xi \vert \geq 1}\vert \xi \vert^{2\ell} \big[ (1+\vert \xi \vert^{2})\hU_{1}(\xi,0)\vert^{2} + \vert \hU_{2}(\xi,0) \vert^{2} + \vert \hU_{3}(\xi,0) \vert^{2} \big]\exp(-2c\vert \xi \vert^{2}t) \: d\xi.
\end{aligned}
\]
Let us estimate $J_{1}(t)$. Since $\vert \xi \vert \leq 1$, we have
\begin{equation}\label{J1}
\begin{aligned}
J_{1}(t)\leq 2 \int_{\vert \xi \vert \leq 1} \vert \xi \vert^{2 \ell}  \vert \hU(\xi,0)\vert^{2} e^{-c_{1}\vert \xi \vert^{2}t}\: d\xi &\leq 2 \sup_{\vert \xi \vert \leq 1}\vert \hU(\xi,0) \vert^{2} \int_{\vert \xi \vert \leq 1} \vert \xi \vert^{2 \ell} e^{-c_{1}\vert \xi \vert^{2}t}\: d\xi \\ &\leq C (1+t)^{-(1+\ell)}\Vert \hW(0) \Vert_{L^{1}}^{2}.
\end{aligned}
\end{equation}
Here we have used that for $\xi \in \R^{n}$ there holds 
\[
\int_{\vert \xi \vert\leq 1}\vert \xi \vert^{2\ell}e^{-c_{0}\vert \xi \vert^{2}t}\: d\xi \leq C(1+t)^{-(n/2+\ell)}
\] 
for some uniform constant $C>0$.

To estimate $J_{2}(t)$, observe that for $\vert \xi \vert\geq 1$ there holds $\exp(-2c\vert \xi \vert^{2}t) \leq \exp(-c_{0}t)$; therefore by Plancherel's theorem we have
\begin{equation}
\label{J2}
\begin{aligned}
J_{2}(t) &\leq e^{-c_{0}t}\int_{\vert \xi \vert \geq 1} \vert \xi \vert^{2\ell} \big[ (1+\vert \xi \vert^{2})\hU_{1}(\xi,0)\vert^{2} + \vert \hU_{2}(\xi,0) \vert^{2}+ \vert \hU_{3}(\xi,0) \vert^{2}\big] \: d\xi \\ &= e^{-c_{0}t}\int_{\vert \xi \vert \geq 1} \big[\big( \vert \xi \vert^{2\ell}  +\vert \xi \vert^{2(\ell+1)} \big)\hU_{1}(\xi,0)\vert^{2} + \vert \xi \vert^{2\ell} \vert \hU_{2}(\xi,0) \vert^{2} + \vert \xi \vert^{2\ell} \vert \hU_{3}(\xi,0) \vert^{2}\big] \: d\xi \\ &\leq e^{-c_{0}t}\int_{\R^{2}} \big[ \big( \vert \xi \vert^{2\ell}  +\vert \xi \vert^{2(\ell+1)} \big)\hU_{1}(\xi,0)\vert^{2} + \vert \xi \vert^{2\ell} \vert \hU_{2}(\xi,0) \vert^{2} + \vert \xi \vert^{2\ell} \vert \hU_{3}(\xi,0) \vert^{2}\big] \: d\xi\\ &= e^{-c_{0}t}\left(  \Vert \partial_{x}^{\ell}U_{1}(0) \Vert_{1}^{2}+ \Vert \partial_{x}^{\ell}U_{2}(0) \Vert_{0}^{2} + \Vert \partial_{x}^{\ell}U_{3}(0) \Vert_{0}^{2}\right).
\end{aligned}
\end{equation}
The result follows by combining \eqref{est-J1J2}, \eqref{J1} and \eqref{J2}.
\end{proof}

\subsection{Heat conducting compressible fluids with viscosity and capillarity}
\label{sec-NSFK}

In order to describe capillarity effects in diffuse interfaces for liquid vapor flows, Korteweg \cite{Kortw1901} (based on an idea by van der Waals \cite{vdW1894}) proposed a constitutive equation for the Cauchy stress that takes in account density gradients. Korteweg's formulation was, however, incompatible with the usual continuum theory of thermodynamics. Later, Dunn and Serrin \cite{DS85} circumvented this problem with the introduction of the interstitial work flux into the energy equation, which accounts for an additional supply of mechanical energy. The resulting model, also known as the \emph{Navier-Stokes-Fourier-Korteweg (NSFK) system}, has the following form,
%
%
\begin{equation}
\label{ModelSev}
\begin{aligned}
\partial_{t}\rho + \nabla \cdot (\rho \bfu) &= 0, \\
\partial_{t}(\rho \bfu) + \nabla \cdot (\rho \bfu \otimes \bfu) + \nabla p &= \nabla \cdot \big(\mathbf{S+K} \big), \\
\partial_{t}\big( \rho \varepsilon + \tfrac{1}{2} \rho \vert \bfu \vert^{2} \big) + \nabla \cdot \big( \rho \bfu \big( \varepsilon + \tfrac{1}{2}\vert \bfu \vert^{2} \big) \big) + \nabla \cdot (p \bfu) &= \nabla \cdot \big( \alpha \nabla \theta + \big( \mathbf{S+K} \big)  \bfu + \bfw \big),
\end{aligned}
\end{equation}	
for $x \in \R^d$, $t > 0$. As before, $\rho$ and $\bfu \in \R^d$ represent the mass density and the velocity field, respectively, $\theta > 0$ is the absolute temperature, $p$ is the thermodynamical  pressure and $\varepsilon$ denotes the internal energy density (per unit mass) of the fluid. The viscous stress $\mathbf{S}$ and the Korteweg tensor $\mathbf{K}$ have the same expressions as the ones given in \eqref{S,K}. In addition, now the viscosity coefficients $\lambda$ and $\nu$ and the capillarity coefficient $k$ are smooth positive functions of $\rho$ and $\theta$. The thermal conductivity coefficient $\alpha > 0$ also depends on $(\rho,\theta)$. The interstitial work flux $\bfw$ is given by
\begin{equation}
\label{w}
\bfw = -k \rho (\nabla \cdot \bfu)\nabla \rho.
\end{equation}

\begin{remark} 
\label{remnoneq}
It is important to observe that the NSFK model \eqref{ModelSev} is based on an extended version of \textit{non-equilibrium} thermodynamics that assumes that the energy of the fluid not only depends on the standard thermodynamic variables but also on density gradients. In system \eqref{ModelSev} this fact is reflected in the form of the non-standard internal energy, given by
\begin{equation}
\label{form-E}
\varepsilon(\rho, \theta, \nabla \rho) = e(\rho, \theta) + \frac{1}{2}(k(\rho,\theta)-\theta k_{\theta}(\rho, \theta) )\vert \nabla \rho \vert^{2}/\rho, 
\end{equation}  
where the function $e=e(\rho,\theta)$ is a standard internal energy from equilibrium thermodynamics and $k=k(\rho, \theta)$ is the capillarity coefficient. We call the term in $\varepsilon$ involving density gradients the non-standard part of the internal energy, so that $\varepsilon$ is decomposed into a standard and a non-standard part. Similar decompositions hold for the specific entropy (per unit mass) $s$ and for the Helmholtz free energy $\Psi$ of the fluid, and we denote their standard parts by $\eta$ and $\psi$, respectively. For more information see Dunn and Serrin \cite{DS85} (see also a quick review in Plaza and Valdovinos \cite{PlV2}, \S 2.1). The major consequence of this thermodynamic description of the fluid is that the conserved quantities depend on the state variables $\rho$, $\theta$ and $\bfu$, \emph{but also on the density gradient, $\nabla \rho$}. This is important at the nonlinear level to close energy estimates (see, e.g., \cite{PlV2}). In this paper, however, we are interested in the \emph{linearization} around constant states for which the resulting constant coefficients do not involve density gradients.
\end{remark}

We consider the mass density $\rho$ and the absolute temperature $\theta$ as the independent thermodynamic variables. They belong to the set
\[
\mathcal{D}:=\{ (\rho,\theta)\in \R^{2}:\: \rho>0,\: \theta>0  \}.
\]
In what follows we assume that the standard part of the free energy, $\psi\in C^{\infty}(\mathcal{D})$, 
is such that the pressure, $p=\rho^{2}\psi_{\rho}(\rho, \theta)>0$, $p\in C^{\infty}(\cD)$, the standard energy , $e=\psi-\theta \psi_{\theta}\in C^{\infty}(\cD)$, and the standard entropy, $\eta=-\psi_{\theta}\in C^{\infty}(\cD)$, satisfy the classical conditions of equilibrium thermodynamics 
\begin{equation}\label{therm-cond}
p>0,\quad p_{\rho}>0,\quad p_{\theta}>0,\quad e_{\theta}>0,
\end{equation}
as well as the relations
\begin{equation}\label{therm-rel}
e_{\rho}=\frac{1}{\rho^{2}}(p-\theta p_{\theta}),\quad \eta_{\theta}=\frac{e_{\theta}}{\theta},\quad \eta_{\rho}=-\frac{p_{\theta}}{\rho^{2}}.
\end{equation}

Next, we recast system \eqref{ModelSev} in quasilinear form. For concreteness and without loss of generality let us make the computations for $d = 3$. We proceed in the standard way: use the quasilinear form of the isothermal model to simplify the energy equation in \eqref{ModelSev}, take the inner product of the momentum equation with the velocity field $\bfu$, and use the result together with the mass balance equation to reduce the energy equation. One has to be careful and to notice that, by the form of the internal energy $\varepsilon$ given by \eqref{form-E}, the time derivative and the gradient of its non-standard part, namely of $\tfrac{1}{2}(k-\theta k_{\theta})\vert \nabla \rho \vert^{2}/\rho$, will result into nonlinear terms. 
Thus, after some lengthy but straightforward calculations, the quasi-linear form of \eqref{ModelSev} can be written as 
\begin{equation}
\label{quas-mat-form-NSFK}
\begin{aligned}
A^{0}(U)\partial_{t}U + \sum_{i=1}^3 A^{i}(U)\partial_{x_i}U &= \sum_{i,j=1}^3 B^{ij}(U)\partial_{x_i}\partial_{x_j}U + \sum_{i,j,k=1}^3C^{ijk}(U)\partial_{x_i}\partial_{x_j}\partial_{x_k}U + \\ &+  \mathcal{N}(U,DU,D^{2}U, \partial_{t} \nabla \rho ), 
\end{aligned}
\end{equation}
where $U=(\rho, \bfu, \theta)^\top \in \R^5$ and 
\[
A^{0}(U)= \begin{pmatrix}
1 & 0 & 0 \\ 0 & \rho \: I_3 & 0 \\ 0 & 0 & \rho\:e_{\theta}
\end{pmatrix} \in \R^{5 \times 5}.
\]
The remaining matrix coefficients are given in terms of their symbols: 
\begin{equation}\label{sym-Ai}
\sum_{i=1}^{3}A^{i}(U)\xi_{i}= \begin{pmatrix}
\bfu \cdot \xi & \rho \:\xi^\top & 0 \\ p_{\rho}\:\xi & \rho\:(\bfu \cdot\xi) \: I_3 & p_{\theta}\:\xi\\ 0 & \theta \: p_{\theta}\:\xi^\top & \rho\: e_{\theta}\: (\bfu \cdot \xi)  
\end{pmatrix} \in \R^{5 \times 5},
\end{equation}
\begin{equation}\label{sym-Bij}
\sum_{i,j=1}^{3}B^{ij}\xi_{i}\xi_{j}=\begin{pmatrix}
0 & 0_{1 \times 3} & 0 \\ 0_{3 \times 1} & \nu \vert \xi \vert^{2}I_3 + (\nu+\lambda)\xi \otimes \xi & 0_{3 \times 1} \\ 0 & 0_{1 \times 3} & \alpha \:\vert \xi \vert^{2}
\end{pmatrix} \in \R^{5 \times 5},
\end{equation}
\begin{equation}\label{sym-Cijk}
\sum_{i,j,k=1}^{3}C^{ijk}\xi_{i}\xi_{j}\xi_{k}= \begin{pmatrix}
0 & 0_{1 \times 3} & 0 \\ k\:\rho \vert \xi \vert^{2}\xi & 0_3 & 0_{3 \times 1} \\ 0 & 0_{1 \times 3} & 0 
\end{pmatrix} \in \R^{5 \times 5}
\end{equation}
(recall that all vectors are column vectors, $\xi = (\xi_1, \xi_2,\xi_3)^\top \in \R^3$). The non-linear terms are contained in $\mathcal{N}(U,DU,D^{2}U, \partial_{t}\nabla \rho )$. 

Then, the linearized version of system \eqref{quas-mat-form-NSFK} around a constant equilibrium state $\bU=(\brho,\overline{\bfu},\bthe)^\top \in \R^5$ reads as
\begin{equation}
\label{lin-NSFK}
A^{0}\partial_{t}U + \sum_{j=1}^{3}A^{j}\partial_{x_j}U= \sum_{j,k=1}^{3}B^{jk}\partial_{x_j}\partial_{x_k} U + \sum_{j,k,l=1}^{3}C^{jkl}\partial_{x_{j}}\partial_{x_{k}}\partial_{x_{l}} U
\end{equation}
where the vector of state variables is $U=(\rho, u_{1}, u_{2}, u_{3}, \theta)^\top \in \R^5$ and the matrix coefficients are those in \eqref{quas-mat-form-NSFK} evaluated at $\bU$, that is, $A^{0}=A^{0}(\bU)$, $A^{i}=A^{i}(\bU)$, $i=1,2,3$, and so on. Thus, for $A^{i}$, $B^{ij}$, and $C^{ijk}$, $i,j,k=1,2,3$, their symbols are obtained by evaluating, respectively, the symbols in \eqref{sym-Ai}, \eqref{sym-Bij}, and \eqref{sym-Cijk} at $\bU$. 

By applying the Fourier transform to \eqref{NSK-2D}, and then splitting the symbol into odd and even orders of derivation, we are led to the system
\begin{equation}
\label{lin-NSFK-Four}
A^{0}\hU_{t}+ \big( i \vert \xi \vert A(\xi) + B(\xi) \big)\hU=0,
\end{equation}
where the generalized transport matrix symbol is given by
\[
\begin{aligned}
A(\xi) = 
\begin{pmatrix}
\overline{\bfu}\cdot\omega & \brho\:\omega^\top  & 0 \\ \beta(\xi)\:\omega & \brho\:(\overline{\bfu}\cdot\omega) I_3 & \bp_{\theta}\:\omega \\ 0 & \bthe\:\bp_{\theta}\:\omega^\top & \brho\:\be_{\theta}\:(\overline{\bfu}\cdot\omega)
\end{pmatrix} =: \sum_{j=1}^{3} \omega_{j} \bA^{j} ,
\end{aligned}
\]
with $\beta(\xi):=\bp_{\rho}+\bk\:\brho\:\vert \xi \vert^{2}$, and the generalized viscous matrix symbol is given by 
\[
\begin{aligned}
B(\xi) =  \sum_{j,k=1}^{3}\vert \xi \vert^{2} B^{jk}\omega_{j}\omega_{k}   = \vert \xi \vert^{2}
  \begin{pmatrix}
0 & 0_{1 \times 3} & 0 \\ 0_{3 \times 1} & \bnu\:\vert \omega \vert^{2}I_3 + (\bnu+\bla)\omega\otimes\omega & 0_{3 \times 1} \\ 0 & 0_{1 \times 3} & \balp\:\vert \omega \vert^{2}
\end{pmatrix},
\end{aligned}
\]
which is (symmetric) positive semi-definite for all $\xi\in \R^{3}\setminus \{0 \}$ because of the conditions on the viscosity coefficients $\nu$ and $\lambda$ and the thermal conductivity $\alpha$. 


Let us examine the symmetrizability of system \eqref{lin-NSFK}. Like in the case of the two dimensional NSK system \eqref{NSK-2D}, it can be shown that system \eqref{lin-NSFK} is not symmetrizable in the sense of Friedrichs (details are left to the reader). Now if we look at system \eqref{lin-NSFK-Four} in the Fourier space, we see that it is not symmetric, given the form of $A(\xi)$ . However, it happens that 
\begin{equation}\label{form-S(xi)-NSFK}
S(\xi)= \begin{pmatrix}
\beta(\xi)/\brho & 0_{1 \times 3} & 0 \\ 0_{3 \times 1} & I_3 & 0_{3 \times 1} \\ 0 & 0_{1 \times 3} & 1/\bthe
\end{pmatrix} \in \R^{5 \times 5},
\end{equation}
is a symbol symmetrizer of system \eqref{lin-NSFK} as the reader can easily verify. Next, we consider a symmetric version of \eqref{lin-NSFK-Four}, that is, 
\begin{equation}\label{NSFK-V}
\hV_{t}+ \big(i\vert \xi \vert A_{S}(\xi) + \vert \xi \vert^{2} \bB_{S}(\xi) \big)\hV=0,
\end{equation}
where we have set $\hV=\big( S(\xi)A^{0}\big)^{1/2}\hU$, and the matrices $A_{S}(\xi)$ and $\bB_{S}(\xi)$ are given by:
\begin{equation}\label{NSFK-tiA}
\begin{aligned}
A_{S}(\xi):&= \big(S(\xi)A^{0}\big)^{-1/2}S(\xi)A(\xi)\big(S(\xi)A^{0}\big)^{-1/2} \\ &=  \big(S(\xi)A^{0}\big)^{-1/2}S(\xi)\Big( \sum_{j=1}^{3}\omega_{j}\bA^{j}(\xi) \Big) \big(S(\xi)A^{0}\big)^{-1/2} \\ & =: \sum_{j=1}^{3}\omega_{j} A^{j}_{S}(\xi) \\ &= \sum_{j=1}^{3} \omega_{j} \begin{pmatrix}
\overline{\bfu}\cdot e_{j} & \beta(\xi)\: e_{j}^\top & 0 \\ \beta\: e_{j} & (\overline{\bfu}\cdot e_{j}) I_3 & \gamma\: e_{j} \\ 0 & \gamma\: e_{j}^\top & \overline{\bfu}\cdot e_{j}
\end{pmatrix} 
\end{aligned}
\end{equation}
where 
\begin{equation}
\label{defgamma}
\gamma:=\frac{\bthe^{1/2}\bp_{\theta}}{\be_{\theta}^{1/2}\brho},
\end{equation}
$\{ e_j \}_{j=1}^3$ is the canonical basis of $\R^3$, and
\begin{equation}\label{NSFK-tiB}
\begin{aligned}
\vert \xi \vert^{2}\bB_{S}(\xi):&= \big(S(\xi)A^{0}\big)^{-1/2}S(\xi)B(\xi)\big(S(\xi)A^{0}\big)^{-1/2} \\ &= \big(S(\xi)A^{0}\big)^{-1/2}S(\xi)\Big(\sum_{j,k=1}^{3}\vert \xi \vert^{2} B^{jk}\omega_{j}\omega_{k}\Big) \big(S(\xi)A^{0}\big)^{-1/2} \\ & =: \vert \xi \vert^{2} \sum_{j,k=1}^{3}\bB^{jk}_{S}\omega_{j}\omega_{k}
\\ &= \frac{\vert \xi \vert^{2}}{\brho}\begin{pmatrix}
0 & 0_{1 \times 3} & 0 \\ 0_{3 \times 1} & \bnu\vert \omega \vert^{2}I_3 + (\bnu+\bla)\omega\otimes\omega & 0_{3 \times 1} \\ 0 & 0_{1 \times 3} & \balp\vert \omega \vert^{2}/\be_{\theta}
\end{pmatrix}.
\end{aligned}
\end{equation}
Direct computations show that 
\begin{equation}\label{unif-sem-pos}
\bB_S(\xi) \geq \sigma \begin{pmatrix}
0 & 0_{1 \times 3} & 0 \\ 0_{3 \times 1} & I_3 & 0_{3 \times 1} \\ 0 & 0_{1 \times 3} & 1
\end{pmatrix},
\end{equation}
for some uniform constant $\sigma>0$ independent of $\omega=\xi/\vert \xi \vert$ for all $\xi\in \R^{3}$, $\xi \neq 0$. Also, in view that $\mbox{ker}\big(\bB_{S}(\xi)\big)=\mbox{span}\left\lbrace (1,0,0,0,0) \right\rbrace$ for all $\xi \in \R^{3}\setminus \{0 \}$, one can easily verify that $A_{S}(\xi)$ and $\bB_{S}(\xi)$ satisfy the genuine coupling condition. To sum up, we have proved the following result.
\begin{proposition}
\label{propNSFK}
The linearized Navier-Stokes-Fourier-Korteweg system \eqref{lin-NSFK} satisfies the hypotheses of the equivalence Theorem \ref{equivalencetheo} (in particular, it is genuinely coupled).
\end{proposition}

Consequently, the equivalence theorem implies the existence of a compensating matrix symbol for the triplet $\big(I, A_{S}(\xi), \bB_{S}(\xi) \big)$. Let us try to find one of such matrices by the method of inspection. Indeed, it is enough to find skew-symmetric matrices $\bK_{S}^{j}(\xi)$, $j=1,2,3$, such that
\[
\sum_{j,k=1}^{3}\left( \big[\bK_{S}^{j}(\xi)A_{S}^{k}(\xi)\big]^{s}+ \bB_{S}^{jk} \right)\omega_{j}\omega_{k} = \sum_{j,k=1}^{3}\big[\bK_{S}^{j}(\xi)A_{S}^{k}(\xi)\big]^{s}\omega_{j}\omega_{k} + \bB_{S}(\xi) >0,
\]
$\forall \xi \in \R^{3}$, $\xi \neq 0$. Then $\bK_{S}(\xi)=\sum_{j=1}^{3}\omega_{j}\bK_{S}^{j}(\xi)$ turns out to be a compensating matrix symbol. 
Let us consider 
\[
\bK_{S}^{1}(\xi)=\frac{\delta_{1}}{\beta(\xi)^{1/2}}\begin{pmatrix}
0 & a & 0 & 0 & 0 \\ -a & 0 & 0 & 0 & b \\ 0 & 0 & 0 & 0 & 0 \\ 0 & 0 & 0 & 0 & 0 \\ 0 & -b & 0 & 0 & 0 
\end{pmatrix},\quad \bK_{S}^{2}(\xi)=\frac{\delta_{2}}{\beta^{1/2}}\begin{pmatrix}
0 & 0 & c & 0 & 0 \\ 0 & 0 & 0 & 0 & 0 \\ -c & 0 & 0 & 0 & d \\ 0 & 0 & 0 & 0 & 0 \\ 0 & 0 & -d & 0 & 0
\end{pmatrix},
\]
\[
\bK_{S}^{3}(\xi)=\frac{\delta_{3}}{\beta(\xi)^{1/2}}\begin{pmatrix}
0 & 0 & 0 & e & 0 \\ 0 & 0 & 0 & 0 & 0 \\ 0 & 0 & 0 & 0 & 0 \\ -e & 0 & 0 & 0 & f \\ 0 & 0 & 0 & -f & 0
\end{pmatrix}.
\]
for some unknowns $a,b,c,d,e,f$ and some small positive constants, $1 \gg \delta_j > 0$, to be chosen later. We compute $\bK_{S}^{1}(\xi)A_{S}^{1}(\xi)$, $\bK_{S}^{1}(\xi)A_{S}^{2}(\xi)$, $\bK_{S}^{1}(\xi)A_{S}^{3}(\xi)$, yielding 
\[
\bK_{S}^{1}(\xi)A_{S}^{1}(\xi)=\frac{\delta_{1}}{\beta(\xi)^{1/2}}\begin{pmatrix}
a\beta(\xi)^{1/2} & a\bu_{1} & 0 & 0 & a\gamma \\ -a\bu_{1} & b\gamma-a\beta(\xi)^{1/2} & 0 & 0 & b\bu_{1} \\ 0 & 0 & 0 & 0 & 0 \\ 0 & 0 & 0 & 0 & 0 \\ -b\beta(\xi)^{1/2} & -b\bu_{1} & 0 & 0 & -b\gamma
\end{pmatrix},
\]
\[
\bK_{S}^{1}(\xi)A_{S}^{2}(\xi)=\frac{\delta_{1}}{\beta(\xi)^{1/2}}\begin{pmatrix}
0 & -a\bu_{2} & 0 & 0 & 0 \\ -a\bu_{2} & 0 & b\gamma-a\beta(\xi)^{1/2} & 0 & b\bu_{2} \\ 0 & 0 & 0 & 0 & 0 \\ 0 & 0 & 0 & 0 & 0 \\ 0 & -b\bu_{2} & 0 & 0 & 0
\end{pmatrix},
\]
\[
\bK_{S}^{1}(\xi)A_{S}^{3}(\xi)= \frac{\delta_{1}}{\beta(\xi)^{1/2}}\begin{pmatrix}
0 & a\bu_{3} & 0 & 0 & 0 \\ -a\bu_{3} & 0 & 0 & b\gamma-a\beta(\xi)^{1/2} & b\bu_{3} \\ 0 & 0 & 0 & 0 & 0 \\ 0 & 0 & 0 & 0  & 0 \\ 0 & -b\bu_{3} & 0 & 0 & 0 
\end{pmatrix},
\]
and similar representations are obtained for $\bK_{S}^{j}(\xi)A_{S}^{k}(\xi)$ for $j=2,3$ and $k=1,2,3$ (which we omit for shortness). Next, by taking $a=c=e=1$, $b=d=f=\gamma/\beta(\xi)^{1/2}$, and $\delta_{i}=\delta/\brho$, $i=1,2,3$, with $0 <\delta \ll 1$, we are led to
\[
\begin{aligned}
\sum_{j,k=1}^{3}\big[\bK_{S}^{j}(\xi)A_{S}^{k}(\xi) \big]^{s}\omega_{j}\omega_{k}&= \frac{\delta}{\brho} \begin{pmatrix}
\vert \omega \vert^{2} & 0_{1 \times 3} & 0 \\ 0_{3 \times 1} & (b\gamma-\beta(\xi)^{1/2}) (\omega\otimes\omega)/\beta(\xi)^{1/2} & 0_{3 \times 1} \\ 0 & 0_{1 \times 3} & -b\gamma\vert \omega \vert^{2}/\beta(\xi)^{1/2}
\end{pmatrix} \\ &= \frac{\delta}{\brho} \begin{pmatrix}
1 & 0_{1 \times 3} & 0 \\ 0_{3 \times 1} & (b\gamma-\beta(\xi)^{1/2}) (\omega\otimes\omega)/\beta(\xi)^{1/2}  & 0_{3 \times 1} \\ 0 & 0_{1 \times 3} & -b\gamma/\beta(\xi)^{1/2}
\end{pmatrix},
\end{aligned}  
\]
inasmuch as $|\omega| = 1$. Since $b\:\gamma/\beta(\xi)^{1/2}=\gamma^{2}/\beta(\xi)$, and $0 \leq 1/\beta(\xi) \leq c$ for some uniform constant $c$ independent of $\xi\in \R^{3}$, the last computation and the property \eqref{unif-sem-pos} of $\bB_{S}(\xi)$ show that there holds
\[	
\sum_{j,k=1}^{3}\left( \big[\bK_{S}^{j}(\xi)A_{S}^{k}(\xi)\big]^{s}+ \bB_{S}^{jk} \right)\omega_{j}\omega_{k} \geq \bsi I_5,
\]
for all $\xi \in \R^{3}$, $\xi \neq 0$, for some uniform $\bsi >0$ as long as we take $\delta$ sufficiently small.
Finally, by the choice of the constants $a,\ldots,f$, and that $\beta(\xi)=O(\vert \xi \vert^{2})$ as $\vert \xi \vert \to \infty$, we conclude $\vert \bK_{S}(\xi) \vert$, $\vert \xi \bK_{S}(\xi) \vert \leq C$ for some uniform positive constant $C$ for all $\xi \in \R^{3}\setminus \{ 0 \}$. 

Thus, along the same lines of the proof of Lemma \ref{lembee-NSK-2D} we can show the next result for the solutions of system \eqref{NSFK-V}. Details are omitted.

\begin{lemma}[basic energy estimate for the NSFK system] Let $\hV=\hV(\xi,t)$ be the solution to the linear system \eqref{NSFK-V}, then there holds 
\begin{equation}
\label{ptw-ee-NSFK}
\vert \hV(\xi,t) \vert \leq C \exp (-c\vert \xi \vert^{2}t) \vert \hV(\xi,0) \vert, 
\end{equation}
for all $\xi \in \R^{3}\setminus \{0 \}$, for some uniform positive constants $C$, $c$. 
\end{lemma}

\begin{corollary}
The Navier-Stokes-Fourier-Korteweg system \eqref{ModelSev} is strictly dissipative of regularity gain type.
\end{corollary}
\begin{proof}
Follows from the decay rate in \eqref{ptw-ee-NSFK} and Remark \ref{Rmr-reg-typ}.
\end{proof}

As we did for the NSK model in the previous section, we need to write the estimate \eqref{ptw-ee-NSFK} in terms of the original variable $\hU$. In view that in this case the matrices $A^{0}$ and $S(\xi)$ are very similar to those in Section \ref{sec-NSK}, we readily obtain the next result. 

\begin{lemma}
The solutions $\hU(\xi,t) = \big(\hU_1, \hU_2, \hU_{3}, \hU_{4}, \hU_{5} \big)(\xi,t)^\top$ to system \eqref{lin-NSFK-Four} satisfy the estimate
\begin{equation}
\label{pwt-ee-U}
\begin{aligned}
& \left( (1+ \vert \xi \vert^2) |\hU_1(\xi,t)|^2 + |\hU_2(\xi,t)|^2 +| \hU_{3}(\xi, t) |^{2} +| \hU_{4}(\xi, t) |^{2} +| \hU_{5}(\xi, t) |^{2} \right)^{1/2}  \leq \\ &\leq C e^{-c \vert \xi \vert^2 t } \left( (1+\vert \xi\vert^2) |\hU_1(\xi,0)|^2 + |\hU_2(\xi,0)|^2 +| \hU_{3}(\xi,0 ) |^{2}+| \hU_{4}(\xi,0 ) |^{2} +| \hU_{5}(\xi,0 ) |^{2} \right)^{1/2},
\end{aligned}
\end{equation}
for all $\xi \in \R^{3} \setminus \{ 0 \}$, $t \geq 0$ and some uniform constants $C$, $c > 0$.
\end{lemma}

We finish the section by stating the result describing the decay structure of the solutions to system \eqref{lin-NSFK}. Its proof is very similar to that of Lemma \ref{lin-dec-U} and we omit it.
\begin{lemma}
\label{lin-dec-NSFK-lm}
Let $U (x,t)=(U_{1}, U_{2}, U_{3}, U_{4}, U_{5})(x,t)^{\top}$ be the solution of the linear system \eqref{lin-NSFK} with initial condition $U(x,0)\in \big( H^{s+1}(\R^{3})\times \big( H^{s}(\R^{3}) \big)^{4} \big)\cap \big( L^{1}(\R^{3}) \big)^{5}$ for $s\geq 0$. Then for each $0\leq \ell \leq s $, the estimate
\begin{equation}
\label{ln-dec-U-NSFK}
\begin{aligned}
&\big(\Vert \partial_{x}^{\ell} U_{1}(t) \Vert_{1}^{2}+ \Vert \partial_{x}^{\ell} U_{2}(t) \Vert_{0}^{2} + \Vert \partial_{x}^{\ell} U_{3}(t) \Vert_{0}^{2} + \Vert \partial_{x}^{\ell} U_{4}(t) \Vert_{0}^{2} +\Vert \partial_{x}^{\ell} U_{5}(t) \Vert_{0}^{2} \big)^{1/2} \leq \\ & \leq  Ce^{-c_{1}t} \big( \Vert \partial_{x}^{\ell}U_{1}(0) \Vert_{1}^{2}+ \Vert \partial_{x}^{\ell}U_{2}(0) \Vert_{0}^{2} + \Vert \partial_{x}^{\ell}U_{3}(0) \Vert_{0}^{2} + \Vert \partial_{x}^{\ell}U_{4}(0) \Vert_{0}^{2} + \Vert \partial_{x}^{\ell}U_{5}(0) \Vert_{0}^{2} \big)^{1/2} \\ & \quad + C(1+t )^{-(3/4+\ell)} \Vert U(0)\Vert_{L^{1}},
\end{aligned}
\end{equation}
holds for all $t\geq 0$ and some uniform positive constants $C$, $c_{1}$.
\end{lemma}

\subsection{Inviscid heat-conducting compressible fluids exhibiting capillarity}
\label{secEFK1d}

A particular case of the NSFK system \eqref{ModelSev} occurs when one considers an \emph{inviscid} heat-exchanging fluid exhibiting some internal capillarity. The model, known as the \emph{Euler-Fourier-Korteweg (EFK) system}, results from taking the viscosity coefficients equal to zero ($\nu = \lambda \equiv 0$) in \eqref{ModelSev} and has the form,
\begin{equation}
\label{EFK-system}
\begin{aligned}
\partial_{t}\rho + \nabla \cdot (\rho \bfu) &= 0, \\
\partial_{t}(\rho \bfu) + \nabla \cdot (\rho \bfu \otimes \bfu) + \nabla p &= \nabla \cdot \mathbf{K}, \\
\partial_{t}\big( \rho \varepsilon + \tfrac{1}{2} \rho \vert \bfu \vert^{2} \big) + \nabla \cdot \big( \rho \bfu \big( \varepsilon + \tfrac{1}{2}\vert \bfu \vert^{2} \big) \big) + \nabla \cdot (p \bfu) &= \nabla \cdot \big( \alpha \nabla \theta + \mathbf{K} \bfu + \bfw \big),
\end{aligned}
\end{equation}
for $x \in \R^d$, $d \geq 1$, $t > 0$. The Korteweg tensor $\mathbf{K}$ is given in \eqref{S,K} and the capillarity and thermal conductivity coefficients ($k$ and $\alpha$, respectively) are positive functions of $(\rho,\theta)$. Notice the presence of second order dissipative terms thanks to the Fourier heat flux term.

Because of reasons that will become clear later (see Proposition \ref{propEFKmdnorgc} below), let us specialize system \eqref{EFK-system} to one space dimension, $d = 1$. The resulting system reads,
\begin{equation}
\label{EFK1d}
\begin{aligned}
        \rho_t + (\rho u)_x &= 0,   \\
        (\rho u)_t + \big(\rho u^{2}+ p \big)_x &= K_x, \\
        \big(\rho  \varepsilon + \tfrac{1}{2}\rho u^{2} \big)_t + \big( \rho u \big( \varepsilon + \tfrac{1}{2}u^2 \big) + pu \big)_x &= \big( \alpha \theta_{x} + uK + w \big)_x,      
\end{aligned} 
\end{equation}
where now $x \in \R$ and $t > 0$ denote space and time coordinates, respectively. The unknown functions are the mass density $\rho > 0$, the velocity $u \in \R$ and the absolute temperature $\theta > 0$. As before, $p$ denotes the thermodynamic pressure function, and $\varepsilon$ is the internal energy density (per unit mass) of the fluid. The Korteweg stress tensor, $K$, and the interstitial work flux, $w$, are given by
\begin{equation}
\label{defKw}
\begin{aligned}
K &= k \rho \rho_{xx} + \rho k_{x}\rho_{x} - \tfrac{1}{2}k_{\rho}\rho \rho_{x}^2 - \tfrac{1}{2} k\rho_{x}^2, \quad \text{and,}\\
w &= -k \rho \rho_{x} u_{x},
\end{aligned}
\end{equation}
respectively. Recall that here the thermal conductivity, $\alpha$, and the capillarity coefficient, $k$, are strictly positive smooth functions of the state variables $\rho$ and $\theta$.

The EFK system \eqref{EFK1d} is the inviscid, one-dimensional version of the model \eqref{ModelSev}. Proceeding as before, if one linearizes around a constant state $\bU=(\brho, \bu, \bthe)^\top \in \R^3$ then one obtains a constant coefficient system of the form, 
\begin{equation}
\label{EFK-1D}
A^{0}U_{t}+A^{1}U_{x}= BU_{xx}+ C U_{xxx},
\end{equation}
where $U=(\rho, u, \theta)^{\top} \in \R^3$ and 
\[
A^{0}=\begin{pmatrix}
1 & 0 & 0  \\ 0 & \brho & 0 \\ 0 & 0 & \brho\: \be_{\theta} \end{pmatrix},\quad A^{1}= \begin{pmatrix}
\bu & \brho & 0 \\ \bp_{\rho} & \brho\:\bu & \bp_{\theta} \\ 0 & \bthe\: \bp_{\theta} & \brho\:\bu\:\be_{\theta}
\end{pmatrix},
\]
\[
B=\begin{pmatrix}
 0 & 0 & 0 \\   0 & 0 & 0 \\  0 & 0 & \balp
\end{pmatrix},\quad C=\begin{pmatrix}
0 & 0 & 0 \\ \bk\:\brho & 0 & 0  \\ 0 & 0 & 0 
\end{pmatrix}.
\]

We apply the Fourier transform to \eqref{EFK-1D}, and after splitting the symbol into odd and even order derivatives, we obtain
\begin{equation}\label{EFK-1D-Four}
\hU_{t}+\big( i  \xi  A(\xi) + B(\xi) \big)\hU=0,
\end{equation}
where 
\[
A(\xi) = A^{1}+\xi^{2}C = \begin{pmatrix}
\bu & \brho & 0 \\ \beta(\xi) & \brho\:\bu & \bp_{\theta} \\ 0 & \bthe\:\bp_{\theta} & \brho\:\bu\:\be_{\theta}
\end{pmatrix}
\]
with $\beta(\xi):=\bp_{\rho}+\bk\:\brho \vert \xi \vert^{2}$ and
\[
B(\xi)= \xi^{2}B := \xi^{2} \begin{pmatrix}
0 & 0 & 0 \\ 0 & 0 & 0 \\ 0 & 0 & \balp
\end{pmatrix},
\]
which is positive semi-definite. As one can easily check, $B(\xi)$ is symmetric, but $A(\xi)$ is not, thus \eqref{EFK-1D-Four} is not in symmetric form. However, it turns out that
\[
S(\xi)= \begin{pmatrix}
\beta/\brho & 0 & 0  \\ 0 & 1 & 0 \\  0 & 0 & 1/\bthe
\end{pmatrix},
\]
is a symbol symmetrizer of \eqref{EFK-1D-Four} as the reader may verify. 

\begin{remark}
By analogous arguments to those employed for the two dimensional NSK system \eqref{NSK-2D} and the three dimensional NSFK system \eqref{lin-NSFK}, it is easy to verify that the one-dimensional EFK system \ref{EFK1d} \emph{is not symmetrizable in the sense of Friedrichs}. We leave the details to the dedicated reader.
\end{remark}

By writing system \eqref{EFK-1D-Four} in the variable $\hV:=\big(S(\xi)A^{0}\big)^{1/2}\hU$ we arrive at the system
\begin{equation}
\label{EFK-1D-Four-i}
\hV_{t}+ \big( i \vert \xi \vert A_{S}(\xi)+ \xi^{2} \bB_{S}(\xi) \big)\hV=0,
\end{equation}
where
\[
A_{S}(\xi) := \big(S(\xi)A^{0} \big)^{-1/2}S(\xi)A(\xi)\big(S(\xi)A^{0} \big)^{-1/2} = \begin{pmatrix}
\bu & \beta(\xi)^{1/2}  & 0 \\ \beta(\xi)^{1/2} & \bu & \gamma \\ 0 & \gamma & \bu 
\end{pmatrix},
\]
$\gamma > 0$ is the same constant defined in \eqref{defgamma}, and
\[
\bB_{S}(\xi) := \big(S(\xi)A^{0} \big)^{-1/2}S(\xi)B\big(S(\xi)A^{0} \big)^{-1/2} =  \begin{pmatrix}
0 & 0 & 0 \\ 0 & 0 & 0 \\ 0 & 0 & \balp/\brho\: \be_{\theta} 
\end{pmatrix}.
\]
The eigenvalues of $A_{S}(\xi)$ can be easily computed:
\[
\mu_{1}(\xi) = \bu-\sqrt{\beta(\xi)+\gamma^{2}} < \mu_{2}(\xi)= \bu < \mu_{3}(\xi)= \bu + \sqrt{\beta(\xi)+ \gamma^{2}}, 
\]
with respective eigenvectors
\[
\begin{aligned}
v_{1} &= \big(\beta(\xi)^{1/2},-\sqrt{\beta(\xi)+\gamma^{2}}, \gamma \big)^{\top},\\
v_{2} &= \big(-\gamma, 0, \beta(\xi)^{1/2} \big)^{\top},\\
v_{3} &= \big( \beta(\xi)^{1/2}, \sqrt{\beta(\xi)+\gamma^{2}}, \gamma \big)^{\top}.
\end{aligned}
\]
From these expressions, it is straightforward to verify that no eigenvector of $A_{S}(\xi)$ belongs to $\ker \bB_{S}(\xi)$ as $\gamma >0$ and $\beta(\xi) > 0$ for all $\xi \in \R$. Moreover, the projections onto the eigenspaces generated by the eigenvalues are determined explicitly by
\[
\begin{aligned}
P_1(\xi) &= \frac{1}{|v_1|^2} \begin{pmatrix} \beta(\xi)^{1/2} v_1^\top \\ -\sqrt{\beta(\xi)+\gamma^{2}} \; v_1^\top \\ \gamma v_1^\top \end{pmatrix} \\
&= \frac{1}{2(\beta(\xi)+\gamma^{2})}\begin{pmatrix} \beta(\xi) & - \beta(\xi)^{1/2} \sqrt{\beta(\xi)+\gamma^{2}}  & \gamma\beta(\xi)^{1/2}  \\ - \beta(\xi)^{1/2}  \sqrt{\beta(\xi)+\gamma^{2}}  & \beta(\xi)+\gamma^{2} & -\gamma \sqrt{\beta(\xi)+\gamma^{2}} \\ \gamma\beta(\xi)^{1/2}  & - \gamma \sqrt{\beta(\xi)+\gamma^{2}}  & \gamma^2 \end{pmatrix},\\
P_2(\xi) &= \frac{1}{|v_2|^2} \begin{pmatrix} -\gamma v_2^\top \\ 0_{1 \times 3} \\ \beta(\xi)^{1/2} v_2^\top \end{pmatrix} = \frac{1}{\beta(\xi)+\gamma^{2}}\begin{pmatrix} \gamma^2 & 0 & - \gamma \beta(\xi)^{1/2} \\ 0 & 0 & 0 \\ - \gamma \beta(\xi)^{1/2} & 0 & \beta(\xi)\end{pmatrix},\\
P_3(\xi) &= \frac{1}{|v_3|^2} \begin{pmatrix} \beta(\xi)^{1/2} v_3^\top \\ \sqrt{\beta(\xi)+\gamma^{2}} \; v_3^\top \\ \gamma v_3^\top \end{pmatrix} \\
&= \frac{1}{2(\beta(\xi)+\gamma^{2})}\begin{pmatrix} \beta(\xi) & \beta(\xi)^{1/2} \sqrt{\beta(\xi)+\gamma^{2}}  & \gamma\beta(\xi)^{1/2}  \\ \beta(\xi)^{1/2}  \sqrt{\beta(\xi)+\gamma^{2}}  & \beta(\xi)+\gamma^{2} & \gamma \sqrt{\beta(\xi)+\gamma^{2}} \\ \gamma\beta(\xi)^{1/2}  & \gamma \sqrt{\beta(\xi)+\gamma^{2}}  & \gamma^2 \end{pmatrix}.
\end{aligned}
\]
%
%
%
These projections are clearly smooth functions on $\xi \in \R\setminus \{0\}$. Therefore, since $A_{S}(\xi)$ is of constant multiplicity, we can use formula \eqref{ExK} to compute the compensating matrix symbol as
\[
\begin{aligned}
\bK_{S}(\xi) &= \sum_{i\neq j}\frac{P_{i}(\xi)\bB_{S}(\xi)P_{j}(\xi)}{\mu_{i}(\xi)-\mu_{j}(\xi)} \\ &= \frac{1}{4(\beta + \gamma^{2})^{2}\brho\:\be_{\theta}} \begin{pmatrix}
0 & 3 \balp \beta(\xi)^{1/2}\gamma^{2} & 0 \\ -3 \balp \beta(\xi)^{1/2}\gamma^{2} & 0 & \balp \gamma(4\beta(\xi)+\gamma^2) \\ 0 & -\balp \gamma(4\beta(\xi)+\gamma^2) & 0 
\end{pmatrix}.
\end{aligned}
\]
Using the form of the coefficients in $\bK_{S}(\xi)$, and that $\beta(\xi)=O(\vert \xi \vert^{2})$ as $\vert \xi \vert \to \infty$, one can conclude that $\vert \bK_{S}(\xi) \vert$, $\vert \xi \bK_{S}(\xi) \vert$, and $\vert (1+\xi^{2})^{1/2}\xi \bK_{S}(\xi)\vert$ are uniformly bounded in $\xi\in \R$, $\xi \neq 0$. Furthermore, this compensating function satisfies
\[
\big[ \bK_{S}(\xi),A_{S}(\xi) \big]+ \bB_{S}(\xi) = \frac{1}{2(\beta(\xi) + \gamma^{2})^{2}\brho\:\be_{\theta}}\begin{pmatrix}
a_{11} & 0 & a_{13} \\ 0 & a_{22} & 0 \\ a_{31} & 0 & a_{33}
\end{pmatrix}, 
\]
where the coefficients are given by
\[
a_{11}=3 \balp \gamma^{2} \beta(\xi) ,\quad a_{13}=a_{31}=\balp \gamma \beta(\xi)^{1/2}(\gamma^{2}-2\beta),
\]
\[
a_{22}=\frac{2\balp\gamma^{2}(\beta(\xi)+\gamma^{2})^{2}\brho\:\be_{\theta}}{2\beta(\xi) \brho\:\be_{\theta}+2\gamma^{2}\brho\:\be_{\theta}},\quad a_{33}=\balp(2\beta(\xi)^{2}+\gamma^{4}).
\]
Using the form of $\beta(\xi)$ and that of the coefficients $a_{11}$, $a_{13}$, $a_{22}$, and $a_{33}$ one can easily see that 
\begin{equation}\label{pos-K-EFK}
\big[ \bK_{S}(\xi)A_{S}(\xi) \big]^{s}+ \bB(\xi) \geq \frac{1}{2}\left( \big[ \bK_{S}(\xi),A_{S}(\xi) \big]+ \bB(\xi) \right) \geq \frac{\bsi}{1+\xi^{2}}I_3,
\end{equation}
for some uniform constant $\bsi >0$. Hence, we have proved the following result.
\begin{proposition}
\label{propEFK1d}
The linearized one-dimensional Euler-Fourier-Korteweg system \eqref{EFK-1D} is genuinely coupled and satisfies the hypotheses of the equivalence Theorem \ref{equivalencetheo}.
\end{proposition}

The compensating matrix symbol that we just computed and its properties allows us to obtain the dissipative structure of the system \eqref{EFK-1D} as it is stated next.

\begin{lemma}[basic energy estimate for the one-dimensional EFK system]
\label{lem-pntw-ee-EFK}
The so\-lu\-tions $\hV(\xi,t)$ to system \eqref{EFK-1D-Four-i} satisfy the estimate 
\begin{equation}\label{pntw-ee-EFK}
\vert \hV(\xi, t) \vert \leq C \exp \left( -\frac{c \xi^{2}}{1+\xi^{2}}t \right) \vert \hV(\xi, 0) \vert,
\end{equation}
for all $\xi \in \R$ and $t\geq 0$, with $C$, $c$ positive uniform constants. 
\end{lemma}
\begin{proof}
We are going to apply Lemma \ref{pnt-ee-rel-fre-lem} once more. To that end, observe that condition \eqref{pos-2} (see \eqref{pos-K-EFK}) is satisfied with $g(\xi) = \bsi /(1+\xi^{2})$, $\bsi$ a constant. To verify the condition \eqref{hyp-enr-est-2} we notice that $\bB_{S}(\xi)$ in \eqref{EFK-1D-Four-i} is a constant matrix. Thus, the bounds
\[
\vert  \xi  \bK_{S}(\xi) \vert,\, \frac{\vert \xi \bK_{S}(\xi) \bB(\xi)^{1/2}\vert}{g(\xi)^{1/2}} \leq C,
\]
are a direct consequence of the bound conditions on $\bK_{S}(\xi)$ as
\[
\frac{\vert \xi \vert \bK_{S}(\xi) \bB_{S}(\xi)^{1/2}}{g(\xi)^{1/2}}= \frac{\vert \xi \vert \bK_{S}(\xi) \bB_{S}}{\bsi (1+\xi^{2})^{-1/2}}= \frac{(1+\xi^{2})^{1/2}\vert \xi \vert \bK_{S}(\xi)\bB_{S}}{\bsi}.
\]
Hence we obtain estimate \eqref{pntw-ee-EFK} taking $g(\xi) = \bsi/(1+\xi^{2})$ in \eqref{dsp-str-ii} of Lemma \ref{pnt-ee-rel-fre-lem}. 
\end{proof}

\begin{corollary}
The one-dimensional Euler-Fourier-Korteweg system \eqref{EFK1d} is strictly dissipative of the standard type.
\end{corollary}
\begin{proof}
Follows from the decay rate in \eqref{ptw-ee-NSFK} and Remark \ref{Rmr-reg-typ} with $p = q =1$.
\end{proof}

The pointwise estimate \eqref{pntw-ee-EFK} in Fourier space is, once again, the key ingredient to obtain decay estimates for the solutions to the evolution problem in physical space. Similarly to the proofs of Lemmata \ref{lemprelimdecay} and \ref{lin-dec-U}, one can obtain the following result which we enunciate without proof (we leave the details to the dedicated reader; see also the proof of Theorem 5.2 in \cite{AnMP20}).
\begin{lemma}
\label{lemdecayEFK1d}
Let us assume that $U=(U_{1},U_{2},U_{3})^{\top}$ is a solution of the linear EFK system \eqref{EFK-1D} with initial condition $U(x,0)\in \big( H^{s+1}(\R)\times \big( H^{s}(\R) \big)^{2} \big)\cap \big( L^{1}(\R) \big)^{3}$ for $s\geq 0$. Then for each $0\leq \ell \leq s $ the estimate
\begin{equation}
\label{ln-dec-U-EFK1d}
\begin{aligned}
\big( \Vert \partial_{x}^{\ell} U_{1}(t) \Vert_{1}^{2} + \Vert \partial_{x}^{\ell} U_{2}(t) \Vert_{0}^{2} + \Vert \partial_{x}^{\ell} U_{3}(t) \Vert_{0}^{2}\big)^{1/2} &\leq  Ce^{-c_{1}t} \big( \Vert \partial_{x}^{\ell}U_{1}(0) \Vert_{1}^{2}+ \Vert \partial_{x}^{\ell}U_{2}(0) \Vert_{0}^{2} + \Vert \partial_{x}^{\ell}U_{3}(0) \Vert_{0}^{2} \big)^{1/2}+ \\ & \quad + C(1+t )^{-(1/4+\ell)} \Vert U(0)\Vert_{L^{1}},
\end{aligned}
\end{equation}
holds for all $t\geq 0$ and some uniform positive constants $C$, $c_{1}$.
\end{lemma}

\subsection{Dispersive Navier-Stokes-Fourier equations}
\label{secDNS}

Let us consider the following \emph{dispersive Navier-Stokes-Fourier (DNSF)} system of equations proposed by Levermore and Sun \cite{LeSu11},
\begin{equation}
\label{nl-Lev-mod}
\begin{aligned}
\partial_{t}\rho + \nabla \cdot (\rho \bfu)&=0, \\ 
\partial_{t}(\rho \bfu) + \nabla \cdot (\rho \bfu \otimes \bfu) + \nabla p&= \nabla \cdot \big( \bSigma + \widetilde{\bSigma} \big), \\
\partial_{t}(\rho e)+ \nabla \cdot (\rho e \bfu  + p \bfu) &= \nabla \cdot \big( \bSigma \bfu + \bfq \big) + \nabla \cdot \big( \widetilde{\bSigma}\bfu + \widetilde{\bfq} \big),
\end{aligned}
\end{equation}
for $x\in \R^{d}$, $d\geq 2$, $t \geq 0$ and where, as before, $\rho$ denotes the mass density, $\bfu \in \R^d$ is the velocity field, $\theta$ denotes the absolute temperature and $p$ is the pressure. The total energy density $e$ is of the form 
\[
e = \frac{1}{2}\rho \vert \bfu \vert^{2} + \frac{d}{2}\rho \theta,
\]
and the classical Navier-Stokes stress tensor $\bSigma$, the pressure and heat flux vector $\bfq$ are given by
\begin{equation}\label{Sigma-q}
\bSigma = \mu(\theta) \mathbf{D}(\bfu), \quad p = \rho \theta, \quad \bfq = \alpha(\theta) \nabla \theta,
\end{equation}
respectively, with $\mathbf{D}(\bfu)=\nabla \bfu + (\nabla \bfu)^{\top}-\tfrac{2}{d}(\nabla \cdot \bfu) I_d$ being the strain and $I_d$ the identity tensor of order $d$. Here the shear viscosity coefficient $\mu$ and the heat conductivity $\alpha$ are assumed to be smooth positive functions of the temperature $\theta$. The dispersive correction to the stress, namely $\widetilde{\bSigma}$, and the correction to the heat flux, $\widetilde{\bfq}$, are given by 
\begin{equation}
\label{disp-corr}
\begin{aligned}
\tilde{\Sigma}&= \tau_{1}(\rho, \theta) \big( \nabla^{2} \theta - \tfrac{1}{d} \Delta \theta I_d \big) + \tau_{2}(\rho, \theta) \big( \nabla \theta \otimes \nabla \theta - \tfrac{1}{d}\vert \nabla \theta \vert^{2} I_d \big)+ \\ &\quad + \tau_{3}(\rho, \theta) \big( \nabla \rho \otimes \nabla \theta + \nabla \theta \otimes \nabla \rho - \tfrac{2}{d}(\nabla \rho \cdot \nabla \theta) I_d \big), \\ 
\widetilde{\bfq} &= \tau_{4}(\rho, \theta) \big( \Delta \bfu + \tfrac{d-2}{d}\nabla \nabla\cdot \bfu \big)+ \tau_{5}(\rho, \theta)\nabla \bfu \nabla \theta + \tau_{6}(\rho, \theta) \nabla \bfu \nabla \rho + \\ & \quad + \tau_{7}(\rho, \theta) \big( \nabla \bfu -(\nabla \bfu )^{\top}\big) \nabla \theta,  
\end{aligned} 
\end{equation}
with $\nabla^{2}\theta$ denoting the Hessian matrix of $\theta$, and $\tau_{i}(\rho, \theta)$, $i=1,2,\ldots,7$, are additional transport coefficients. 

The motivation for the introduction of this dispersive DNSF system \eqref{nl-Lev-mod} stems from the well known fact that the classical fluid dynamics equations are not accurate in certain regimes. For instance, in the phenomenon known as the \emph{ghost effect} for rarefied gases (see, e.g., Sone \cite{Sone02}), the flow moves along the direction from the low temperature domain to the high temperature domain, which is quite different from the usual heat flow. The dispersive corrections \eqref{disp-corr} are not derivable from the Navier-Stokes system of gas dynamics but are derivable from kinetic theory. The Boltzmann equation from kinetic theory provides more information on the microscopic level so that some interesting phenomena, such as the ghost effect, cannot be modeled by the classical Euler and Navier–Stokes-Fourier equations. Therefore, some higher-order corrections need to be added to the governing equations. These corrections can be formally derived from the Boltzmann equation in small mean-free-path regimes. Levermore and collaborators \cite{Lev03rep,LeSu11,LeSuT12} have proposed a balance argument to derive a family of gas dynamical systems in a systematic fashion, including the classical fluid equations and equations beyond Navier-Stokes. One feature of the models derived in this way is that they obey an entropy structure (see eq. (1.5) in \cite{LeSu11}) which shows their consistency with the second law of thermodynamics. This entropy structure also guarantees that the systems are at least linearly stable. Among them, system \eqref{nl-Lev-mod} is the one with lowest order that provides a first correction to the compressible Navier-Stokes-Fourier model.

In system \eqref{nl-Lev-mod} the form of the transport coefficients $\mu(\theta)$, $\alpha(\theta)$, $\tau_{i}(\rho, \theta)$, $i=1,\ldots,7$, depend on details of the underlying kinetic equations. In particular, the transport coefficients $\tau_{i}(\rho, \theta)$, $i=1,\ldots,6$, satisfy the relations 
\begin{equation}\label{taus-rel}
\tau_{4}=\frac{\theta}{2}\tau_{1},\quad \frac{\tau_{2}}{\theta}+ \frac{2 \tau_{5}}{\theta^{2}}= \partial_{\theta}\big( \frac{\tau_{4}}{\theta^{2}}\big),\quad \theta \tau_{3} + \tau_{6} = 2 \partial_{\rho}\tau_{4}. 
\end{equation} 

We are going to take the mass density $\rho$ and the absolute temperature $\theta$ as the independent thermodynamic variables which take values in the set
\[
\mathcal{D}= \left\lbrace (\rho, \theta)\in \R^{2}:\: \rho >0,\: \theta>0 \right\rbrace.
\]
Let us observe that from the first relation in \eqref{taus-rel}, we obtain that $\tau_{1}$ and $\tau_{4}$ have the same sign in $\mathcal{D}$. We assume that the transport coefficient being smooth functions on $\mathcal{D}$ satisfy
\begin{equation}\label{therd-cond}
\mu( \theta) >0,\quad \alpha(\theta) >0,\quad \tau_{1}(\rho, \theta),\quad \tau_{4}(\rho, \theta)>0.
\end{equation}

Next we write the linearized version of system \eqref{nl-Lev-mod} around a constant equilibrium state $\bU=(\rho, \overline{\bfu}, \bthe) \in \R^5$, $(\brho, \bthe)\in \mathcal{D}$, if we specialize the system to dimension $d=3$. We could proceed as we did in Section \ref{sec-NSFK} for the NSFK system. Instead of that, we use the quasilinear form of system \eqref{nl-Lev-mod} already obtained in \cite{LeSu11} to deduce its linear form. Specifically, by looking at equations (2.2), (2.2) in \cite{LeSu11} one can verify, after some calculations, that the linearized version of system \eqref{nl-Lev-mod} around the equilibrium state $\bU$ has the form 
\begin{equation}
\label{Lev-mod}
\partial_{t}U + \sum_{j=1}^{3}A^{j}\partial_{x_{j}}U= \sum_{j,k=1}^{3}B^{jk}\partial_{x_{j}}\partial_{x_{k}} U + \sum_{j,k,l=1}^{3}C^{jkl}\partial_{x_{j}} \partial_{x_{k}}\partial_{x_{l}} U,
\end{equation}
where $U=(\rho,u_{1},u_{2},u_{3},\theta)^\top$, and the matrix coefficients are given in terms of their symbols:
\[
\sum_{j=1}^{3}A^{j}\xi_{j}=\begin{pmatrix}
\overline{\bfu} \cdot \xi & \brho\xi^{\top} & 0 \\ \frac{\bthe}{\brho}\xi & (\overline{\bfu}\cdot \xi) I_3 & \xi \\ 0 & \frac{2}{3}\bthe\xi^{\top} & \overline{\bfu}\cdot \xi
\end{pmatrix},
\]
\[
\sum_{j,k=1}^{3}B^{jk}\omega_{j}\omega_{k}= \begin{pmatrix}
0 & 0_{1 \times 3} & 0 \\ 0_{3 \times 1} & \frac{\bmu}{\brho}\vert \xi \vert^{2}I_3 + \frac{1}{3}\frac{\bmu}{\brho}\xi\otimes\xi & 0_{3 \times 1} \\
0 & 0_{1 \times 3} & \frac{2}{3}\frac{\balp}{\brho}\vert \xi \vert^{2}
\end{pmatrix},
\]
\[
\sum_{j,k,l=1}^{3}C^{jkl}\xi_{j}\xi_{k}\xi_{l}= \begin{pmatrix}
0 & 0_{1 \times 3} & 0 \\ 0_{3 \times 1} & 0_3 & \frac{2}{3\brho}\btau_{1}\vert \xi \vert^{2}\xi \\ 0 & \frac{8}{9\brho}\btau_{4}\vert \xi \vert^{2}\xi^{\top} & 0 
\end{pmatrix}.
\]
Here the overlined variables represent the variables evaluated at the constant state $\bU$ (for example, $\btau_{1}=\tau_{1}(\brho,\bthe)$, $\btau_{4}=\tau_{4}(\brho,\bthe)$,  $\bmu=\mu(\bthe)>0$, etc.)
Apply the Fourier transform to \eqref{Lev-mod}, and split the symbol into odd and even orders of derivation to obtain
\begin{equation}
\label{Lev-mod-Four}
\hU_{t}+\big( i \vert \xi \vert A(\xi) + B(\xi) \big) \hU=0,
\end{equation}
where the generalized flux and the generalized viscosity matrices are, respectively, given by
\begin{equation}\label{A-Lev-mod}
A(\xi) = \sum_{j=1}^{3}\omega_{j}\bA^{j}(\xi)
:= \sum_{j=1}^{3}\omega_{j}\begin{pmatrix} \overline{\bfu} \cdot e_{j} & \brho e_{j}^{\top} & 0 \\ \frac{\bthe}{\brho}e_{j} & (\overline{\bfu}\cdot e_{j}) I_3 & \beta_{1}(\xi) e_{j} \\ 0 & \beta_{2}(\xi) e_{j} & \overline{\bfu}\cdot e_{j}
\end{pmatrix},
\end{equation}
where we have set $\beta_{1}(\xi):=1+\tfrac{2}{3\brho}\btau_{1}\vert \xi \vert^{2}$, $\beta_{2}(\xi):=\tfrac{2}{3}\bthe+\tfrac{8}{9\brho}\btau_{4}\vert \xi \vert^{2}$, $\{ e_j \}_{j=1}^3$ is the canonical basis in $\R^3$, and by	 
\begin{equation}\label{B-Lev-mod}
B(\xi)= \sum_{j,k=1}^{3}\vert \xi \vert^{2} B^{jk}\omega_{j}\omega_{k}= \vert \xi \vert^{2} \begin{pmatrix}
0 & 0_{1 \times 3} & 0 \\ 0_{3 \times 1} & \frac{\bmu}{\brho}\vert \omega \vert^{2}I_3 + \frac{1}{3}\frac{\bmu}{\brho}\omega\otimes\omega & 0_{3 \times 1} \\
0 & 0_{1 \times 3} & \frac{2}{3}\frac{\balp}{\brho}\vert \omega \vert^{2}
\end{pmatrix},
\end{equation}
which is symmetric and positive semi-definite. 
Using the first relation in \eqref{taus-rel}, we obtain
\[
\frac{2}{3}\bthe \left( \frac{2}{3\brho}\btau_{1} \right)= \frac{4\bthe}{9\brho}\tau_{1}=\frac{4}{9\brho}(2\tau_{4})= \frac{8}{9\brho}\tau_{4}.
\]
Then one can easily see that 
\begin{equation}\label{Fr-symm}
S=S(\bU)=\begin{pmatrix}
\frac{2}{3}\frac{\bthe^{2}}{\brho^{2}} & 0_{1 \times 3} & 0 \\ 0_{3 \times 1} & \frac{2}{3}\bthe I_3 & 0_{3 \times 1} \\ 0 & 0_{1 \times 3} & 1
\end{pmatrix}
\end{equation}	
symmetrizes the system \eqref{Lev-mod} in the sense of Friedrichs. As Friedrich symmetrizability implies symbol symmetrizability, the matrix $S$ in \eqref{Fr-symm} is a symbol symmetrizer of \eqref{Lev-mod-Four}. Observe that in this case the symbol symmetrizer does not depend on the Fourier wave number $\xi \in \R^{3}$.

Let us consider the system \eqref{Lev-mod-Four} in the variable $\hV:=S^{1/2}\hU$, so that we are led to 
\begin{equation}\label{Lev-Four-hV}
\hV_{t}+\big( i\vert \xi \vert A_{S}(\xi) + \vert \xi \vert^{2}\bB_{S}(\xi) \big)\hV=0.
\end{equation}	
Here the matrix coefficients have the form
\begin{equation}\label{tiA-Lev-mod}
\begin{aligned}
A_{S}(\xi) &= S^{-1/2}SA(\xi)S^{-1/2} \\ &=S^{-1/2}S\Big( \sum_{j=1}^{3}\omega_{j}\bA^{j}(\xi) \Big) S^{-1/2} =: \sum_{j=1}^{3}\omega_{j}A_{S}^{j}(\xi) \\ &= \sum_{j=1}^{3}\omega_{j} \begin{pmatrix}
\overline{\bfu} \cdot e_{j} & \bthe^{1/2} e_{j}^{\top} & 0 \\ \bthe^{1/2} e_{j} & (\overline{\bfu} \cdot e_{j}) I_3 & \beta(\xi) e_{j} \\ 0 & \beta(\xi) e_{j}^{\top} & \overline{\bfu}\cdot e_{j}
\end{pmatrix},
\end{aligned}
\end{equation}
with $\beta(\xi) :=\big( \tfrac{3}{2\bthe} \big)^{1/2}\beta_{2}(\xi)$, and 
\begin{equation}\label{tiB-Lev-mod}
\begin{aligned}
\bB_{S}(\xi)&= S^{-1/2}S\Big( \sum_{j,k=1}^{3} B^{jk}\omega_{j}\omega_{k}\Big) S^{-1/2} =:\sum_{j,k=1}^{3}\bB_{S}^{jk}\omega_{j}\omega_{k} \\ &= \begin{pmatrix}
0 & 0_{1 \times 3} & 0 \\ 0_{3 \times 1} & \frac{\bmu}{\brho}\vert \omega \vert^{2}I_3 + \frac{1}{3}\frac{\bmu}{\brho}\omega\otimes\omega & 0_{3 \times 1} \\
0 & 0_{1 \times 3} & \frac{2}{3}\frac{\balp}{\brho}\vert \omega \vert^{2}
\end{pmatrix}.
\end{aligned}
\end{equation}

Clearly $\ker \bB_{S}(\xi) =\text{span}\{(1,0,0,0,0)^\top \} \subset \R^5$. Then we have that for $a\neq 0$, 
\[
A_{S}(\xi)\begin{pmatrix}
a \\ 0 \\ 0 \\ 0 \\ 0
\end{pmatrix}= \begin{pmatrix}
a (\bfu\cdot\omega) \\ \bthe^{1/2}a\omega \\ 0
\end{pmatrix} \neq \zeta \begin{pmatrix}
a \\ 0 \\ 0 \\ 0 \\ 0
\end{pmatrix},\quad \forall \zeta \in \R,
\]
as $\bthe>0$. Then the triplet $\big(I_5, A_{S}(\xi), \bB_{S}(\xi) \big)$ satisfies the genuine coupling condition. Thus, we have proved the following result.
\begin{proposition}
\label{propDNSF}
The linearized dispersive Navier-Stokes-Fourier system \eqref{Lev-mod} is genuinely coupled and, therefore, it satisfies the hypotheses of the equivalence Theorem \ref{equivalencetheo}.
\end{proposition}

By the equivalence theorem there exists a compensating matrix symbol for \eqref{Lev-mod}. Let us try to find one of such matrices. We proceed by inspection, as in the case of the NSFK system. We start by proposing matrices $\bK_{S}^{j}(\xi)$, $j=1,2,3$, of the form
\[
\bK_{S}^{1}(\xi)=\delta_{1}\begin{pmatrix}
0 & a & 0 & 0 & 0 \\ -a & 0 & 0 & 0 & b \\ 0 & 0 & 0 & 0 & 0 \\ 0 & 0 & 0 & 0 & 0 \\ 0 & -b & 0 & 0 & 0 
\end{pmatrix},\quad \bK_{S}^{2}(\xi)=\delta_{2}\begin{pmatrix}
0 & 0 & c & 0 & 0 \\ 0 & 0 & 0 & 0 & 0 \\ -c & 0 & 0 & 0 & d \\ 0 & 0 & 0 & 0 & 0 \\ 0 & 0 & -d & 0 & 0
\end{pmatrix},
\]
\[
\bK_{S}^{3}(\xi)=\delta_{3}\begin{pmatrix}
0 & 0 & 0 & e & 0 \\ 0 & 0 & 0 & 0 & 0 \\ 0 & 0 & 0 & 0 & 0 \\ -e & 0 & 0 & 0 & f \\ 0 & 0 & 0 & -f & 0
\end{pmatrix},
\]
for some coefficients $a,\ldots,f$, and $\delta_{i}$, $i=1,2,3$, to be determined later. 
Computations of $\bK_{S}^{1}(\xi)A_{S}^{1}(\xi)$, $\bK_{S}^{1}(\xi)A_{S}{2}(\xi)$, and $\bK_{S}^{1}(\xi)A_{S}^{3}(\xi)$ yield
\[
\bK_{S}^{1}(\xi)A_{S}^{1}(\xi)=\delta_{1} \begin{pmatrix}
a \bthe^{1/2} & a\bu_{1} & 0 & 0 & a\beta(\xi) \\ -a\bu_{1} & b\beta(\xi)-a\bthe^{1/2} & 0 & 0 & b\bu_{1} \\ 0 & 0 & 0 & 0 & 0 \\ 0 & 0 & 0 & 0 & 0 \\ -b\bthe^{1/2} & -b\bu_{1} & 0 & 0  & -b \beta(\xi)
\end{pmatrix}, 
\]
\[
\bK_{S}^{1}(\xi)A_{S}^{2}(\xi)=\delta_{1} \begin{pmatrix}
0 & a\bu_{2} & 0 & 0 & 0 \\ -a\bu_{2} & 0 & b\beta(\xi)-a\bthe^{1/2} & 0 & b\bu_{2} \\ 0 & 0 & 0 & 0 & 0 \\ 0 & 0 & 0 & 0 & 0 \\ 0 & -b\bu_{2} & 0 & 0 & 0 
\end{pmatrix},
\]
\[
\bK_{S}^{1}(\xi)A_{S}^{3}(\xi)=\delta_{1}\begin{pmatrix}
0 & a\bu_{3} & 0 & 0 & 0 \\ -a\bu_{3} & 0 & 0 & b\beta(\xi)-a\bthe^{1/2} & b\bu_{3} \\ 0 & 0 & 0 & 0 & 0 \\ 0 & 0 & 0 & 0 & 0 \\ 0 & -b\bu_{3} & 0 & 0 & 0 
\end{pmatrix}. 
\]
Similar expressions are obtained for $\bK_{S}^{j}(\xi)A_{S}^{k}(\xi)$, $j=2,3$, $k=1,2,3$. Next, we take $a=c=e=1/\beta(\xi)^{2}$, $b=d=f=1/\beta(\xi)^{3/2}$, and $\delta_{i}=\delta$, $i=1,2,3$, with $0 < \delta \ll 1$. 
Thus, we obtain
\[
\begin{aligned}
\sum_{j,k=1}^{3} \big[ \bK_{S}^{j}(\xi)A_{S}^{k}(\xi) \big]^{s}\omega_{j}\omega_{k} &= \delta \begin{pmatrix}
\frac{\bthe^{1/2}}{\beta^{2}}\vert \omega \vert^{2} & 0_{1 \times 3} & \frac{1}{2}\left( \frac{1}{\beta} - \frac{\bthe^{1/2}}{\beta^{3/2}} \right) \vert \omega \vert^{2} \\ 0_{3 \times 1} & \left( \frac{1}{\beta^{1/2}} -\frac{\bthe^{1/2}}{\beta^{2}} \right) \omega\otimes\omega & 0_{3 \times 1} \\ \frac{1}{2}\left( \frac{1}{\beta} - \frac{\bthe^{1/2}}{\beta^{3/2}} \right) \vert \omega \vert^{2} & 0_{1 \times 3} & -\frac{1}{\beta^{1/2}}\vert \omega \vert^{2}
\end{pmatrix} \\ &= \begin{pmatrix}
\delta\frac{\bthe^{1/2}}{\beta} & 0_{1 \times 3} &  \frac{\delta}{2 \beta}\left( 1 - \frac{\bthe^{1/2}}{\beta^{1/2}} \right) \\ 0_{3 \times 1} & \frac{\delta}{\beta^{1/2}}\left( 1 - \frac{\bthe^{1/2}}{\beta^{3/2}} \right) \omega\otimes\omega & 0_{3 \times 1} \\  \frac{\delta}{2\beta}\left( 1 - \frac{\bthe^{1/2}}{\beta^{1/2}} \right) & 0_{1 \times 3} & -\frac{\delta}{\beta^{1/2}} 
\end{pmatrix},
\end{aligned}
\]
which implies 
\begin{equation}\label{pos-Lev}
\begin{aligned}
\sum_{j,k=1}^{3}  \big( \big[ \bK_{S}^{j}(\xi) A_{S}^{k}(\xi) \big]^{s} &+ \bB_{S}^{jk}(\xi) \big)\omega_{j}\omega_{k} = \sum_{j,k=1}^{3}\big[ \bK_{S}^{j}(\xi)A_{S}^{k}(\xi) \big]^{s}\omega_{j}\omega_{k} + \bB_{S}(\xi) \\ &=\begin{pmatrix}
\delta \frac{\bthe^{1/2}}{\beta^2} & 0_{1 \times 3} & \frac{\delta}{2\beta}\left( 1 - \frac{\bthe^{1/2}}{\beta^{1/2}} \right)  \\ 0_{3 \times 1} & \frac{\bmu}{\brho} I_3 + \left( \frac{1}{3}\frac{\bmu}{\brho} + \frac{\delta}{\beta^{1/2}}\left( 1 - \frac{\bthe^{1/2}}{\beta^{3/2}} \right) \right) \omega\otimes\omega & 0_{3 \times 1} \\ \frac{\delta}{2\beta}\left( 1 - \frac{\bthe^{1/2}}{\beta^{1/2}} \right)  & 0_{1 \times 3} & \frac{2}{3}\frac{\balp}{\brho} - \frac{\delta}{\beta^{1/2}}
\end{pmatrix}
\end{aligned}
\end{equation}
Let us verify that the last matrix is positive definite. For the sake of simplicity, we define
\[
M(\xi):= \frac{\bmu}{\brho} I_3 + \left( \frac{1}{3}\frac{\bmu}{\brho} + \frac{\delta}{\beta^{1/2}}\left( 1 - \frac{\bthe^{1/2}}{\beta^{3/2}} \right) \right) \omega\otimes\omega,\quad \xi \in \R^{3}\setminus \{ 0 \}.
\]
As $(1-\bthe^{1/2}/\beta^{3/2})/\beta^{1/2} \to 0$, then there exist $0 < \delta_{1} \ll 1$ such that for $\delta \leq \delta_{1}$ there holds $\mbox{det}M(\xi) \geq c >0$ $\forall \xi \in \R^{3}$, $\xi \neq 0$, and for some uniform constant $c$ independent of $\xi$. With this choice for $\delta$ we have that the principal $4 \times 4$ minor of the matrix \eqref{pos-Lev} has positive determinant, namely 
\[
\mbox{det} \begin{pmatrix}
\delta \frac{\bthe^{1/2}}{\beta^2} & 0_{1 \times 3} \\ 0 & \frac{\bmu}{\brho} I_3 + \left( \frac{1}{3}\frac{\bmu}{\brho} + \frac{\delta}{\beta^{1/2}}\left( 1 - \frac{\bthe^{1/2}}{\beta^{3/2}} \right) \right) \omega\otimes\omega
\end{pmatrix} = \frac{\delta \bthe^{1/2}}{\beta^{2}}\text{det}M(\xi) >0.
\]
Thus the positive definiteness of \eqref{pos-Lev} will follow if we can show that its determinant is uniformly bounded below. Performing the computations we obtain
\[
\begin{aligned}
\det  &\begin{pmatrix}
\delta \frac{\bthe^{1/2}}{\beta^2} & 0_{1 \times 3} & \frac{\delta}{2\beta}\left( 1 - \frac{\bthe^{1/2}}{\beta^{1/2}} \right)  \\ 0_{3 \times 1} & \frac{\bmu}{\brho} I_3 + \left( \frac{1}{3}\frac{\bmu}{\brho} + \frac{\delta}{\beta^{1/2}}\left( 1 - \frac{\bthe^{1/2}}{\beta^{3/2}} \right) \right) \omega\otimes\omega & 0_{3 \times 1} \\ \frac{\delta}{2\beta}\left( 1 - \frac{\bthe^{1/2}}{\beta^{1/2}} \right)  & 0_{1 \times 3} & \frac{2}{3}\frac{\balp}{\brho} - \frac{\delta}{\beta^{1/2}}
\end{pmatrix} \\ &=  \frac{\delta \bthe^{1/2}}{\beta^{2}}\det M(\xi)\left( \frac{3}{2}\frac{\balp}{\brho} - \frac{\delta}{\beta^{1/2}} \right) + \frac{\delta}{2 \beta}\left( 1- \frac{\bthe^{1/2}}{\beta^{1/2}} \right) \left( -\frac{\delta}{2 \beta}\left( 1- \frac{\bthe^{1/2}}{\beta^{1/2}} \right)\det M(\xi) \right)  \\ &= \frac{\delta \bthe^{1/2}}{\beta^{2}}\det M(\xi)\left( \frac{3}{2}\frac{\balp}{\brho}- \frac{\delta}{\beta^{1/2}} \right)  -\frac{\delta^{2}}{4 \beta^{2}}\left( 1- \frac{\bthe^{1/2}}{\beta^{1/2}} \right)^{2}\det M(\xi) \\ &= \det M(\xi) \left( \frac{3}{2}\frac{\balp}{\brho}\frac{\delta \bthe^{1/2}}{\beta^{2}}-\frac{\delta^{2}\bthe^{1/2}}{\beta^{5/2}}- \frac{\delta^{2}}{4 \beta^{2}} \left( 1- \frac{\bthe^{1/2}}{\beta^{1/2}} \right)^{2}  \right) \\ &= \det M(\xi) \left( \frac{3}{2}\frac{\balp}{\brho}\frac{\delta \bthe^{1/2}}{\beta^{2}}-\frac{\delta^{2}\bthe^{1/2}}{2\beta^{5/2}}- \frac{\delta^{2}}{4 \beta^{2}} - \frac{\delta^{2} \bthe}{4 \beta^{2}}  \right)  \\ & \geq \det M(\xi) \left( \frac{ \delta \bthe^{1/2}}{\beta^{2}}- \frac{\delta^{2}}{\beta^{2}}\left( \frac{1}{4}+ \frac{\bthe^{1/2}}{2}+ \frac{\bthe}{4} \right) \right) \geq   \frac{C}{\beta^{2}},
\end{aligned}
\]
for $C$ a positive uniform constant independent of $\xi$, as long as we take $0 < \delta \leq \delta_{2}$, for some $0< \delta_{2} \ll 1$, sufficiently small.

From the above computations and the fact that $1/\beta^{2} \geq C/(1+\vert \xi \vert^{2})^{2}$, we arrive at
\[ 
\sum_{j,k=1}^{3}  \big( \big[ \bK_{S}^{j}(\xi) A_{S}^{k}(\xi) \big]^{s} + \bB_{S}^{jk}(\xi) \big)\omega_{j}\omega_{k} \geq \frac{C}{(1+\vert \xi \vert^{2})^{2}}I_5,
\]
for some uniform constant $C$, as long as we take $0 < \delta \leq \min \{ \delta_{1}, \delta_{2} \} $. Finally, we let $\bK_{S}(\xi):=\sum_{j=1}^{3}\bK_{S}^{j}(\xi)$. Thus, by the choice of the coefficients $a, \ldots, f$, we have that $\bK(\xi) = O(1/\beta^{3/2})=O(1/\vert \xi \vert^{3})$, to that the norms of $\vert \xi \vert \bK_{S}(\xi)$ and $(1+\vert \xi \vert^{2})\vert \xi \vert \bK_{S}(\xi)$ are uniformly bounded in $\xi \in \R^{3}\setminus \{ 0 \}$. The compensating matrix symbol just computed and its properties imply the next pointwise energy estimate of the solutions to system \eqref{Lev-Four-hV}. 

\begin{lemma}[basic energy estimate for the DNSF system] The solution $\hV= \hV(\xi, t)$ to the system \eqref{Lev-Four-hV} satisfies the estimates
\begin{equation}
\label{pnt-ee-Lev}
\vert \hV(\xi, t) \vert \leq C \exp \left( - k \frac{ \vert \xi \vert^{2}}{(1 + \vert \xi \vert^{2})^{2}}t \right) \vert \hV(\xi, 0 ) \vert, 
\end{equation}
for all $\xi \in \R^{3} \setminus \{ 0 \}$, $t\geq 0$, with positive uniform constant $C$ and $k$.
\end{lemma} 
\begin{proof}
The proof is the same as that of Lemma \ref{lem-pntw-ee-EFK}, but with $g(\xi) = C/(1+ \vert \xi \vert^{2})^{2}$. 
\end{proof}

\begin{corollary}
\label{corregloss}
The dispersive Navier-Stokes-Fourier system \eqref{nl-Lev-mod} is strictly dissipative of regularity loss type.
\end{corollary}
\begin{proof}
Follows from the decay rate in \eqref{pnt-ee-Lev} and Remark \ref{Rmr-reg-typ} with $p =1 < q = 2$.
\end{proof}

By looking at the form of the (Friedrichs) symmetrizer $S$ given in \eqref{Fr-symm} (being a diagonal matrix), we can easily show that estimate \eqref{pnt-ee-Lev} also holds for the original variable $\hU$. This estimate implies the linear decay of solutions as it is stated next. 

\begin{lemma}
Let $U=U(x,t)$ be a solution of the linear system \eqref{Lev-mod} with initial data $U_{0}\in H^{s}(\R^{3}) \cap L^{1}(\R^{3})$, for $s\geq 0$. Then the solution satisfies the estimate
\[
\Vert \partial_{x}^{\ell}U(t) \Vert_{0} \leq C (1+ t)^{-k/2} \Vert \partial_{x}^{\ell+k}U_{0} \Vert_{0} + C(1+ t )^{-(3/4 + \ell/2)}\Vert U_{0} \Vert_{L^{1}},
\]
for all $0 \leq \ell, k \leq s$, $\ell + k \leq s$, and $t\geq 0$, and some positive constant $C$. 
\end{lemma}
\begin{proof}
The proof is very similar to that of Lemma \ref{lin-dec-U}, and the only difference comes when we estimate the high-frequency part $J_{2}$ in \eqref{est-J1J2}. We start by fixing $\ell$ satisfying $0 \leq \ell \leq s$, next multiply estimate \eqref{pnt-ee-Lev} for $\hU$ by $\vert \xi \vert^{2\ell}$ and integrate in $\xi \in \R^{3}$ to get
\begin{equation}\label{est-I1I2}
\begin{aligned}
\int_{\R^{3}} \vert \xi \vert^{2\ell} \vert \hU(\xi,t)\vert^{2}\, d\xi &\leq C \int_{\vert \xi \vert \leq 1}\exp{\left( -2k \frac{\vert \xi \vert^{2}}{(1+\vert \xi \vert^{2})^{2}}t \right)}\vert \hU(\xi, 0) \vert^{2}\, d\xi +  \\ &+ C \int_{\vert \xi \vert \geq 1}\exp{\left( -2k \frac{\vert \xi \vert^{2}}{(1+\vert \xi \vert^{2})^{2}}t \right)}\vert \hU(\xi, 0) \vert^{2}\, d\xi \\&=: I_{1}(t) + I_{2}(t).
\end{aligned}
\end{equation}

To estimate $I_{1}(t)$, let us observe that for $\vert \xi \vert \leq 1$ we have $\vert \xi \vert^{2}/(1+ \vert \xi \vert^{2})^{2} \geq c \vert \xi \vert^{2} $ for some positive uniform constant $c$. Thus, we obtain
\[
I_{1}(t) \leq C \int_{\vert \xi \vert \leq 1}\exp{ (-2c_{0} \vert \xi \vert^{2}t )}\vert \hU(\xi, 0) \vert^{2}\, d\xi,
\]
so that we can proceed in the same fashion as we did to estimate $J_{1}(t)$ in the proof of Lemma \ref{lin-dec-U}, which leads to
\begin{equation}\label{est-I1}
I_{1}(t) \leq C(1+ t )^{-(3/2+ \ell)}\Vert U_{0} \Vert_{L^{1}}^{2}.
\end{equation}

For the estimation of $I_{2}(t)$, use that for $\vert \xi \vert \geq 1$ there holds $\vert \xi \vert^{2}/(1+\vert \xi \vert^{2})^{2} \geq c/\vert \xi \vert^{2}$ for some positive (uniform) constant $c$. Thus, we get
\begin{equation}\label{est-I2}
\begin{aligned}
I_{2}(t) \leq C \sup_{t\geq 0,\, \vert \xi \vert \geq 0} \frac{e^{-c_{0}t/\vert \xi \vert^{2}}}{\vert \xi \vert^{2k}} \int_{\vert \xi \vert\geq 1} \vert \xi \vert^{2(\ell + k)}\vert \hU(\xi, 0) \vert^{2}\, d\xi \leq C(1+t)^{-k}\Vert \partial_{x}^{\ell + k} U_{0} \Vert_{0}^{2}. 
\end{aligned}
\end{equation}
Thus, the proof is completed by combining estimates \eqref{est-I1I2} through \eqref{est-I2}. 
\end{proof}

\subsection*{Asymptotic behaviour of the eigenvalues for the one-dimensional DNSF system}

According to Corollary \ref{corregloss} the linearized system \eqref{Lev-mod} is strictly dissipative of type $(1,2)$ and it exhibits a regularity-loss type structure. Next, following \cite{KSX22}, we are going to study the asymptotic expansion of the eigenvalues $\lambda$ of
\begin{equation}\label{eig-prob-Lev-i}
\big( \lambda + i \vert \xi \vert A(\xi) + B(\xi) \big) \phi=0,
\end{equation}
as $\vert  \xi \vert \to \infty$, and verify that this characterization of its dissipative structure is actually optimal. For that purpose we apply perturbation theory of one-parameter family of matrices \cite{Kat82}. For simplicity, we restrict ourselves to the one dimensional version of \eqref{eig-prob-Lev-i}, namely
\[
\big( \lambda + i \xi A^{1} - (i \xi )^{2}B^{1} - (i\xi)^{3}C^{1} \big)\phi=0,
\]
where $A^{1}$, $B^{1}$, and $C^{1}$ are the one-dimensional versions of the matrix coefficients appearing in \eqref{Lev-mod}. This eigenvalue problem is equivalent to the one we obtain after we multiply on the left by the (Friedrichs) symmetrizer $S(\bU)$ given in \eqref{Fr-symm},
\begin{equation}
\label{1d-eig-prob-LevMod}
\left( \lambda(i\xi)A^{0} + (i\xi) A -(i \xi )^{2} B- (i\xi)^{3}C \right)\phi=0,
\end{equation}
where
\[
A^{0}= S(\bU) =\begin{pmatrix}
\frac{2}{3}\frac{\bthe^2}{\brho^{2}} & 0 & 0 \\ 0 & \frac{2}{3}\bthe & 0 \\ 0 & 0 & 1
\end{pmatrix}, \quad A= A^{0}A^{1}= \begin{pmatrix}
0 & \frac{2}{3}\frac{\bthe^{2}}{\brho} & 0 \\ \frac{2}{3}\frac{\bthe^{2}}{\brho} & 0 & \frac{2}{3}\bthe \\ 0 & \frac{2}{3}\bthe & 0 
\end{pmatrix},
\]
\[
B= A^{0}B^{1}= \begin{pmatrix}
0 & 0 & 0 \\ 0 & \frac{8}{9}\frac{\bmu}{\brho}\bthe & 0 \\ 0 & 0 & \frac{2}{3}\frac{\balp}{\brho}
\end{pmatrix}, \quad C = A^{0}C^{1}= \begin{pmatrix}
0 & 0 & 0 \\ 0 & 0 & \frac{8}{9 \brho}\btau \\ 0 & \frac{8}{9 \brho}\btau & 0  
\end{pmatrix}.
\]
Here we have assumed, for simplicity, that $\bu=\bu_{1}=0$. 

By perturbation theory of one-parameter family of matrices \cite{Kat82} (see also \cite{LiuZ97}, Section 6),  we can write for the eigenvalue $\lambda=\lambda(i \xi)$ and the corresponding eigenvectors $\phi=\phi(i \xi)$ of \eqref{1d-eig-prob-LevMod} the asymptotic expansion as $\vert \xi \vert \to \infty$,
\begin{equation}\label{asym-exp-eigs}
\lambda(i \xi) = \sum_{n=1}^{3}(i\xi)^{n}\lambda^{(n)} + \sum_{n=0}^{\infty}(i \xi)^{-n}\lambda^{(-n)}, \quad \phi(i \xi) = \sum_{n=0}^{\infty}(i \xi )^{-n}\phi^{(-n)}.
\end{equation}
We substitute the expressions in \eqref{asym-exp-eigs} into \eqref{1d-eig-prob-LevMod}, and collect the terms having the same power in $(i\xi )^{n}$. By considering $n=3, 2, 1, 0, -1, -2$ we obtain the relations
\begin{equation}\label{eig-rel-i}
\big( \lambda^{(3)}A^{0}-C  \big)\phi^{(0)} =0, 
\end{equation} 
\begin{equation}\label{eig-rel-ii}
\big( \lambda^{(3)}A^{0}-C \big) \phi^{(-1)} + \big( \lambda^{(2)}A^{0}-B \big)\phi^{(0)}=0,
\end{equation}
\begin{equation}\label{eig-rel-iii}
\big( \lambda^{(3)}A^{0}-C \big) \phi^{(-2)} + \big( \lambda^{(2)}A^{0}-B \big)\phi^{(-1)}+ \big( \lambda^{(1)}A^{0}+A \big)\phi^{(0)}=0,
\end{equation}
\begin{equation}
\label{eig-rel-iv}
\big( \lambda^{(3)}A^{0}-C \big) \phi^{(-3)} + \big( \lambda^{(2)}A^{0}-B \big)\phi^{(-2)}+ \big( \lambda^{(1)}A^{0}+A \big) \phi^{(-1)} + \lambda^{(0)}A^{0}\phi^{(0)}=0,
\end{equation}
\begin{equation}
\label{eig-rel-v}
\big( \lambda^{(3)}A^{0}-C \big) \phi^{(-4)} + \big( \lambda^{(2)}A^{0}-B \big)  \phi^{(-3)}+  \big( \lambda^{(1)}A^{0}+A \big)\phi^{(-2)} + \lambda^{(0)}A^{0}\phi^{(-1)}+\lambda^{(-1)}A^{0}\phi^{(0)}=0,
\end{equation}
\begin{equation}
\label{eig-rel-vi}
\begin{aligned}
\big( \lambda^{(3)}A^{0}-C \big) \phi^{(-5)} + \big( \lambda^{(2)}A^{0}-B \big)\phi^{(-4)} &+ \big( \lambda^{(1)}A^{0}+A \big)\phi^{(-3)} + \lambda^{(0)}A^{0}\phi^{(-2)}+\\ &+\lambda^{(-1)}A^{0}\phi^{(-1)}+\lambda^{(-2)}A^{0}\phi^{(0)}=0
\end{aligned}
\end{equation}
Let us start by computing $\lambda^{(3)}$ and $\phi^{(0)}$ from \eqref{eig-rel-i}. Observe that \eqref{eig-rel-i} means that $\lambda^{(3)}$ is an eigenvalue of $(A^{0})^{-1}C$ with eigenvector $\phi^{(0)}$. Thus we compute $\lambda^{(3)}$ by computing the roots of 
\[
\begin{aligned}
\text{det}\big( \lambda^{(3)}A^{0}- C \big) = \text{det} \begin{pmatrix}
\frac{2}{3}\frac{\bthe^{2}}{\brho^{2}}\lambda^{(3)} & 0 & 0 \\ 0 & \frac{2}{3}\bthe \lambda^{(3)} & -\frac{8}{9\brho}\btau_{4} \\ 0 & -\frac{8}{9\brho}\btau_{4} & \lambda^{(3)} 
\end{pmatrix} =\frac{2}{3}\frac{\bthe^{2}}{\brho^{2}}\lambda^{(3)} \left( \frac{2}{3}\bthe \big( \lambda^{(3)} \big)^{2} -  \left( \frac{8}{9 \brho}\btau_{4} \right)^{2} \right).
\end{aligned}
\]
Thus
\[
\lambda^{(3)}_{1}=0,\quad \lambda^{(3)}_{2,3}= \pm \frac{8}{9 \brho}\btau_{4}\sqrt{\frac{3}{2\bthe}},
\]
with respective eigenvectors
\[
\phi^{(0)}_{1}=(1,0,0)^{\top},\quad \phi^{(0)}_{2,3}= \left( 0, \pm \sqrt{\frac{3}{2\bthe}},1 \right)^{\top}.
\]
Let us denote $\phi^{(-1)}_{i} := (\rho_{1}, u_{1}, \theta_{1})^\top$, $i=1,2,3$. Then, using $\lambda^{(3)}_{1}$ and $\phi^{(0)}_{1}$ in \eqref{eig-rel-ii}, we are led to
\[
-C\phi^{(-1)}_{1}+ \big( \lambda^{(2)}_{1}A^{0}-B \big) \phi^{(0)}_{1}=0
\]
or 
\[
\big( \lambda^{(2)}_{1}-B \big)(1,0,0)^{\top} = C(\rho_{1}, u_{1}, \theta_{1})^{\top}.
\]
Component by component, this equation reads
\[
\frac{2}{3}\frac{\bthe^{2}}{\brho^{2}}\lambda^{(2)}_{1}=0,\quad 0=\frac{8}{9\brho}\btau_{4}\theta_{1}, \quad 0 = \frac{8}{9\brho}\btau u_{1},
\]
from where we obtain
\begin{equation}\label{lam-2-1}
\lambda^{(2)}_{1}=0, \quad \phi^{(-1)}_{1}=(\tilde{\rho},0, 0)^{\top},\,\,  \tilde{\rho} \,\, \text{being arbitrary}.
\end{equation}
Let us compute $\lambda^{(2)}_{2}$ and $\phi^{(-1)}_{2}$. We let $\phi^{(-1)}_{2}=(\rho_{1},u_{1},\theta_{1})$, then \eqref{eig-rel-ii} now reads
\[
\big( \lambda^{(3)}_{2} A^{0}- C \big)(\rho_{1}, u_{1}, \theta_{1})^{\top} + \big( \lambda^{(2)}_{2} A^{0}-B \big) \left( 0, \sqrt{\frac{3}{2\bthe}},1 \right)^{\top}=0,
\]
or in coordinates
\[
\begin{aligned}
\frac{2}{3}\frac{\bthe^{2}}{\brho^{2}}\lambda^{(3)}_{2} \rho_{1} &=0, \\
\frac{2}{3}\bthe \lambda^{(3)}_{2} u_{1}- \frac{8}{9\brho}\btau_{4} \theta_{1} + \left( \lambda^{(2)}_{2} - \frac{4}{3}\frac{\bmu}{\brho} \right) \frac{2}{3}\bthe \sqrt{\frac{3}{2\bthe}} &= 0, \\
 -\frac{8}{9\brho}\btau_{4}u_{1} + \lambda^{(3)}_{2} \theta_{1} + \left( \lambda^{(2)}_{2} - \frac{2}{3}\frac{\balp}{\brho} \right) &=0.
\end{aligned}
\]
Given that $\lambda^{(3)}_{2}=\frac{8}{9 \brho}\btau_{4}\sqrt{\frac{3}{2\bthe}}$, if we multiply the second equation by $\sqrt{\frac{3}{2\bthe}}$ and add it to the third one, we obtain
\[
2 \lambda^{(2)}_{2} - \frac{4}{3}\frac{\bmu}{\brho}- \frac{2}{3}\frac{\balp}{\brho}=0,
\]
so that
\[
\lambda^{(2)}_{2} = \frac{2}{3}\frac{\bmu}{\brho} + \frac{1}{3}\frac{\balp}{\brho}>0.
\]
The same procedure works for $\lambda^{(2)}_{3}$ to get
\begin{equation}\label{lam-2-2,3}
\lambda^{(2)}_{3}= \lambda^{(2)}_{2}= \frac{2}{3}\frac{\bmu}{\brho} + \frac{1}{3}\frac{\balp}{\brho}. 
\end{equation}
So far, combining \eqref{lam-2-1} and \eqref{lam-2-2,3}, we have
\[
\text{Re}\, \lambda^{(2)}_{1} =0, \quad \text{Re}\, \lambda^{(2)}_{2,3} >0.
\]
Thus we need to compute the next coefficient in the expansion \eqref{asym-exp-eigs} for $\lambda_{1}$, namely $\lambda^{(1)}_{1}$, $\lambda^{(0)}_{1}$, $\lambda^{(-1)}_{1}$, and so on. We know that $\lambda^{(3)}_{1} = \lambda^{(2)}_{1}=0$,  $\phi^{(0)}_{1}=(1,0,0)^{\top}$, and $\phi^{(-1)}_{1}=(\tilde{\rho},0,0)$. Let us define $\phi^{(-2)}_{1}=(\rho_{2}, u_{2}, \theta_{2})$, then plugging $\lambda^{(3)}_{1}$, $\lambda^{(2)}_{1}$, $\phi{(0)}_{1}$, and $\phi^{(-1)}_{1}$ into \eqref{eig-rel-iii} we obtain
\[
- B\phi^{(-1)}_{1} +  \big( \lambda^{(1)}_{1}A^{0} + A \big)\phi^{(0)}_{1}  = C \phi^{(-2)}_{1},
\]
or 
\[
-B (\tilde{\rho}, 0, 0 )^{\top} + \big( \lambda^{(1)}_{1}A^{0} + A \big)(1,0,0)^{\top} = C (\rho_{2}, u_{2}, \theta_{2} )^{\top}.
\]
This equation in coordinates reads as 
\[
\frac{2}{3}\frac{\bthe^{2}}{\brho^{2}}\lambda^{(1)}_{1}=0, \quad \frac{2}{3}\frac{\bthe^{2}}{\brho}= \frac{8}{9\brho}\btau_{4}\theta_{2},\quad 0 = \frac{8}{9\brho}\btau_{4}u_{2},
\]
so that we obtain
\begin{equation}\label{lam-1-1}
\lambda^{(1)}_{1}=0, \quad \phi^{(-2)}_{1}=(\rho_{2}, 0, \theta_{2}), \quad \theta_{2}= \frac{3}{4}\frac{\bthe^{2}}{\btau_{4}}>0,
\end{equation}
where $\rho_{2}$ can be taken arbitrarily. 

Next we let $\phi^{(-3)}_{1}=(\rho_{3}, u_{3}, \theta_{3})$. We make use of $\lambda^{(3)}_{1}$, $\lambda^{(2)}_{1}$, $\phi^{(0)}_{1}$, $\phi^{(-1)}_{1}$, and the expressions for $\lambda^{(1)}_{1}$ and $\phi^{(-2)}_{1}$ in \eqref{eig-rel-iv} to obtain
\[
-B \phi^{(-2)}_{1} + A\phi^{(-1)}_{1} + \lambda^{(0)}_{1}A^{0}\phi^{(0)}_{1} = C\phi^{(-3)}_{1},
\]
or equivalently
\[
-B (\rho_{2}, 0, \theta_{2})^{\top} + A (\tilde{\rho}, 0, 0)^{\top} + \lambda^{(0)}_{1}A^{0}(1, 0, 0)^{\top} = C(\rho_{3}, u_{3}, \theta_{3}).
\]
Expressed in coordinates last formula yields
\[
\frac{2}{3}\frac{\bthe^{2}}{\brho^{2}}\lambda^{(0)}_{1} = 0, \quad \frac{2}{3}\frac{\bthe^{2}}{\brho}\tilde{\rho} = \frac{8}{9\brho}\btau_{4}\theta_{3}, \quad -\frac{2}{3}\frac{\balp}{\brho}\theta_{2} = \frac{8}{9\brho}\btau_{4}u_{3}.
\]
Thus we obtain
\begin{equation}\label{lam-0-1}
\lambda^{(0)}_{1}=0, \quad \phi^{(-3)}_{1}=(\rho_{3}, u_{3}, \theta_{3})^{\top},\quad u_{3} = -\frac{3}{4}\frac{\balp}{\btau_{4}}\theta_{2}<0, \quad  \theta_{3}= \frac{3}{4\btau_{4}}\bthe^{2}\tilde{\rho}, 
\end{equation}
with $\rho_{3}$ being arbitrary. Now, we use $\lambda^{(0)}_{1}$ and $\phi^{(-3)}_{1}$ just obtained in \eqref{eig-rel-v} to obtain 
\[
-B\phi^{(-3)}_{1} + A\phi^{(-2)}_{1}+ \lambda^{(-1)}_{1}A^{0}\phi^{(0)}_{1} = C\phi^{(-4)}_{1},
\]
or 
\[
-B(\rho_{3}, u_{3}, \theta_{3})^{\top} + A(\rho_{2}, 0, \theta_{2})^{\top} + \lambda^{(-1)}_{1}A^{0}(1,0,0)^{\top}= C( \rho_{4}, u_{4}, \theta_{4})^{\top}, 
\]
where we have set $\phi^{(-4)}_1= (\rho_{4}, u_{4}, \theta_{4})^{\top}$. By coordinates this equation is expressed as follows,
\[
\begin{aligned}
\frac{2}{3}\frac{\bthe^{2}}{\brho^{2}}\lambda^{(-1)}_{1}&=0, \\ 
-\frac{8}{9}\frac{\bmu}{\brho}\bthe u_{3} + \frac{2}{3}\frac{\bthe^2}{\brho}\rho_{2} + \frac{2}{3}\bthe \theta_{2} &= \frac{8}{9\brho}\btau_{4}\theta_{4}, \\
-\frac{2}{3}\frac{\balp}{\brho}\theta_{3} &= \frac{8}{9\brho}\btau_{4}u_{4}.
\end{aligned}
\]
Thus we obtain
\begin{equation}\label{lam+1-1}
\lambda^{(-1)}_{1} = 0,
\end{equation}
and $u_{4}$ is given in terms of $\theta_{3}$, while $\theta_{4}$ is given as a function of $u_{3}$, $\theta_{2}$, and $\rho_{2}$. Again, $\rho_{4}$ can be choosen arbitrarily. 

Now, we proceed to compute $\lambda^{(-2)}_{1}$. Using that $\lambda^{(-1)}_{1}=0$ and setting $\phi^{(-5)}_{1}=(\rho_{5}, u_{5}, \theta_{5})^{\top}$ in \eqref{eig-rel-vi} we obtain
\[
- B \phi^{(-4)}_{1} + A \phi^{(-3)}_{1} + \lambda^{(-2)}_{1}A^{0}\phi^{(0)}_{1}= C\phi^{(-5)}_{1}, 
\]
or
\[
-B(\rho_{4}, u_{4}, \theta_{4})^{\top} + A (\rho_{3}, u_{3}, \theta_{3})^{\top}+ \lambda^{(-2)}_{1}A^{0}(1, 0, 0)^{\top} = C (\rho_{5}, u_{5}, \theta_{5})^{\top}.
\]
The first of this system of equations reads as
\[
\frac{2}{3}\frac{\bthe^{2}}{\brho}u_{3}+ \lambda^{(-2)}_{1} \frac{2}{3}\frac{\bthe^{2}}{\brho^{2}}=0,
\]
so that
\begin{equation}\label{lam+2-1}
\lambda^{(-2)}_{1}= -\brho u_{3} = \frac{3}{4}\frac{\brho\, \balp}{\btau_{4}}\theta_{2} =\frac{9}{8}\frac{\brho\, \balp}{\btau_{4}}\bthe^{2} >0,
\end{equation}
where we have used \eqref{lam-1-1} and \eqref{lam-0-1} for $u_{3}$ and $\theta_{2}$, respectively.

We summarize the previous computations in the following lemma. 

\begin{lemma}
\label{lemasymp}
Let us consider the eigenvalue problem for the one-dimensional version of system \eqref{Lev-mod}, that is the one given by \eqref{1d-eig-prob-LevMod}. Then its solutions $\lambda_{1}=\lambda_{1}(i \xi)$, $\lambda_{2}=\lambda_{2}(i \xi)$ and $\lambda_{3}=\lambda_{3}(i \xi)$ satisfy the asymptotic expansion \eqref{asym-exp-eigs} as $\vert \xi \vert \to \infty$. In the case of $\lambda_{2}$ and $\lambda_{3}$ they are such that
\[
\Im \lambda^{(3)}_{2}=\Im \lambda^{(3)}_{3}=0,\quad \mbox{and}\quad \Re \lambda^{(2)}_{2}=\Re \lambda^{(2)}_{3}= \frac{2}{3}\frac{\bmu}{\brho} + \frac{1}{3}\frac{\balp}{\brho} >0.
\]
For $\lambda_{1}$, we have
\[
\lambda^{(3)}_{1}=\lambda^{(2)}_{1}= \cdots = \lambda^{(-1)}_{1}=0,\quad \mbox{and} \quad \Re \lambda^{(-2)}_{1}= \frac{9}{8}\frac{\brho \, \balp}{\btau_{4}}\bthe^2 >0.
\]
\end{lemma}

\begin{remark}
In view of Lemma \ref{lemasymp}, we obtain $\Re \lambda_1(\xi) \sim - \theta_0/\xi^2 + \ldots$, suggesting in this fashion that the dissipative structure of regularity loss type expressed in \eqref{pnt-ee-Lev} for the dispersive Navier-Stokes-Fourier system is actually optimal, at least in the one dimensional case.
\end{remark}

\subsection{Isentropic quantum hydrodynamics}
\label{isoQHD}

Quantum hydrodynamics (QHD) models appear in many areas in physics, including the modeling of quantum semiconductor devices \cite{GarC94}, the dynamics of Bose-Einstein condensates \cite{DGPS99} or the description of superfluidity \cite{Land41}. Here we shall consider the isentropic version of the model derived by Gardner \cite{GarC94} to describe the flow of electrons in quantum semiconductor devices. The system of equations under consideration consists of the isentropic Euler equations for the particle density and current density including the quantum Bohm potential and a momentum relaxation term. The momentum equation is highly nonlinear and contains a dispersive term with third-order derivatives. The model reads,
%
%
\begin{align}
\partial_{t}n+\nabla\cdot(n\bfv)&=0,\label{eq:qhm1}\\
\partial_{t}(n \bfv)+\nabla\cdot\left(n \bfv\otimes \bfv\right)+\nabla p(n)-\frac{\epsilon^{2}}{6}n\nabla\left(\frac{\Delta\sqrt{n}}{\sqrt{n}}\right)&=-\frac{n\bfv}{\tau}-\mu\nabla n,\label{eq:qhm2}
\end{align}
where $n \in \R$ denotes the particle density and $\boldsymbol{J} = n \bfv \in \R^3$ is the momentum. The parameter $\tau>0$ represents the relaxation time, $\epsilon>0$ is proportional to Planck's constant $\hbar$ and $\mu>0$ is the interaction constant (see, \cite{JuMi07,LiMar04,PlZ24,SLH23}, for example). Like in classical fluid dynamics, we assume that the pressure satisfies $p^{\prime}(n)>0$. The quantum Bohm potential
\begin{align*}
V(n)=\frac{\epsilon^{2}}{2}\frac{\Delta\sqrt{n}}{\sqrt{n}},
\end{align*}
produces a non-linear dispersive term in the balance of momentum that can be written as 
\begin{align}
n\nabla\left(\frac{\Delta\sqrt{n}}{\sqrt{n}}\right)=\frac{1}{2}\nabla\cdot\left(n\nabla\otimes\nabla\log(n)\right)=\frac{1}{2}\Delta(\nabla n)-2\nabla\cdot\left(\nabla\sqrt{n}\otimes\nabla\sqrt{n}\right).\label{eq:qbp}
\end{align}
It is important to remark that this model is \emph{inviscid}, that is, we are not considering physical nor numerical viscosity terms (see, for example, \cite{JuMi07}), which may appear in some models (cf. \cite{PlZ24,HQS24,JuMi07,BrMe10,LMZ20a}). For these physical considerations and for the mathematical analysis of isentropic QHD models, the reader is referred to \cite{AnMa09,AnMa12,AnMaZh21,LiMar04,Gas01} and the many references therein.

The quasilinear form of system \eqref{eq:qhm1}-\eqref{eq:qhm2} in three space dimensions is 
\begin{align*}
\partial_{t}U + \sum_{i=1}^3 A^{i}(U)\partial_{i}U+Q(U)= \sum_{i,j,k=1}^3 C^{ijk}(U)\partial_{i}\partial_{j}\partial_{k} U +\mathcal{N}(U;\nabla U), 
\end{align*}
where 
\begin{align*}
\sum_{i=1}^3 \xi_{i}A^{i}(U)&= \sum_{i=1}^3 |\xi|\omega_{i}A^{i}(U)=|\xi|\left(\begin{array}{cccc}
\bfv \cdot \omega&n\omega_{1}&n\omega_{2}&n\omega_{3}\\
\left(\tfrac{p^{\prime}(n)+\mu}{n}\right)\omega_{1}&\bfv \cdot \omega&0&0\\
\left(\tfrac{p^{\prime}(n)+\mu}{n}\right)\omega_{2}&0&\bfv \cdot \omega&0\\
\left(\tfrac{p^{\prime}(n)+\mu}{n}\right)\omega_{3}&0&0&\bfv \cdot \omega
\end{array}\right),\\
\sum_{i,j,k=1}^3 \xi_{i}\xi_{j}\xi_{k}C^{ijk}(U)&= \sum_{i,j,k=1}^3 |\xi|^{3}\omega_{i}\omega_{j}\omega_{k}C^{ijk}(U)=|\xi|^{3}\left(\begin{array}{cccc}
0&&0_{1\times 4}&\\
\omega_{1}\tfrac{\epsilon^{2}}{12n}&&&\\
\omega_{2}\tfrac{\epsilon^{2}}{12n}&&0_{3\times 3}&\\
\omega_{3}\tfrac{\epsilon^{2}}{12n}&&&
\end{array}\right),\\
&Q(U)=\left(\begin{array}{c}
0\\
\tfrac{1}{\tau}v_{1}\\
\tfrac{1}{\tau}v_{2}\\
\tfrac{1}{\tau}v_{3}
\end{array}\right),
\end{align*}
and the non-linear terms are 
\begin{align*}
\mathcal{N}(U,\nabla U)=\left(\begin{array}{c}
0\\
-\tfrac{\epsilon^{2}}{3n}\nabla\cdot\left(\nabla\sqrt{n}\otimes\nabla\sqrt{n}\right)
\end{array}\right).
\end{align*}
In this case, the set of admissible states is $\mathcal{O}=\{(n,\bfv)^{\top}\in\mathbb{R}^{4}~:~n>0\}$, and the set of constant equilibrium states is characterized as
\begin{align*}
\mathcal{M}_{\mathrm{eq}}:=\left\lbrace \overline{V}=(\bar{n},\overline{\bfv})^{\top}\in\mathcal{O}~:~\overline{V}~\mbox{is constant and }~\overline{\bfv}=0\right\rbrace.
\end{align*}
Linearizing around a constant state $\overline{V}=(\bar{n},\overline{\bfv})^{\top}\in\mathcal{M}_{\mathrm{eq}}$ yields a system of the form
\begin{equation}
\label{linQHD}
\partial_{t}U + \sum_{i=1}^3 \overline{A}^i \partial_{i}U + \overline{L} U = \sum_{i,j,k=1}^3 \overline{C}^{ijk} \partial_{i}\partial_{j}\partial_{k} U,
\end{equation}
where $\overline{A}^i := A^i(\overline{V})$, $\overline{C}^{ijk} := C^{ijk}(\overline{V})$ and $\overline{L} := DQ(\overline{V})$, being $DQ$ the Jacobian of $Q$ with respect to the state variables $U = (n, \bfv)^\top$. The Fourier transform of the linearized system is
\begin{equation}
\label{FlinQHD}
\widehat{U}_{t}+i|\xi|\Big( \sum_{i=1}^3 \omega_{i}\overline{A}^i + |\xi|^{2} \sum_{i,j,k=1}^3 \omega_{i}\omega_{j}\omega_{k}\overline{C}^{ijk} \Big) \widehat{U}+ \overline{L} \widehat{U}=0.
\end{equation}
Proceeding as in \eqref{Fourieroe} (grouping even and odd order terms) we obtain
\begin{align}
A(\xi)&:= \sum_{i=1}^3 \omega_{i}\overline{A}^i+|\xi|^{2} \sum_{i,j,k=1}^3\omega_{i}\omega_{j}\omega_{k}\overline{C}^{ijk}\nonumber\\
&=\left(\begin{array}{cccc}
0&\bar{n}\omega_{1}&\bar{n}\omega_{2}&\bar{n}\omega_{3}\\
\left(p^{\prime}(\bar{n})+\mu+\tfrac{\epsilon^{2}}{12}|\xi|^{2}\right)\frac{\omega_{1}}{\bar{n}}&0&0&0\\
\left(p^{\prime}(\bar{n})+\mu+\tfrac{\epsilon^{2}}{12}|\xi|^{2}\right)\frac{\omega_{2}}{\bar{n}}&0&0&0\\
\left(p^{\prime}(\bar{n})+\mu+\tfrac{\epsilon^{2}}{12}|\xi|^{2}\right)\frac{\omega_{3}}{\bar{n}}&0&0&0
\end{array}\right),\label{eq:Asymbolqhm}\\
B(\xi)&:=\overline{L} =\left(\begin{array}{cccc}
0&0&0&0\\
0&1/\tau&0&0\\
0&0&1/\tau&0\\
0&0&0&1/\tau
\end{array}\right).\label{eq:Bsymbolqhm}
\end{align}
Observe that, for any $\overline{V}\in\mathcal{M}_{\mathrm{eq}}$, we can define a symbolic symmetrizer for the triplet $(I_4,A(\xi),B(\xi))$ as 
\begin{align*}
S(\xi)=\left(\begin{array}{cccc}
\Theta(\xi)\frac{1}{\bar{n}^{2}}&0&0&0\\
0&1&0&0\\
0&0&1&0\\
0&0&0&1
\end{array}\right),\quad\mbox{where}\quad \Theta(\xi):=p^{\prime}(\bar{n})+\mu+\frac{\epsilon^{2}}{12}|\xi|^{2}.
\end{align*}
Since the characteristic polynomial of $A(\xi)$ is 
\begin{align*}
P(\lambda,\xi)\rvert_{\overline{\bfv}=0}=\left[\left(\overline{\bfv}\cdot\omega-\lambda\right)^{4}-\Theta(\xi)(\overline{\bfv}\cdot\omega-\lambda)^{2}\right]\Bigr\rvert_{\overline{\bfv}=0},
\end{align*}
we deduce that it has three eigenvalues,
\begin{align}
\lambda_{1}(\xi)\Bigr\rvert_{\overline{\bfv}=0}=\overline{\bfv}\cdot\omega\Bigr\rvert_{\overline{\bfv}=0}=0,\label{eq:qhdeigen1}
\end{align}
with algebraic multiplicity equal to two, and
\begin{align}
\lambda_{2,3}(\xi)\Bigr\rvert_{\overline{\bfv}=0}=\left(\overline{\bfv}\cdot\omega\pm\sqrt{\Theta(\xi)}\right)\Bigr\rvert_{\overline{\bfv}=0}=\pm\sqrt{\Theta(\xi)}.\label{eq:qhdeigen2}
\end{align}
Now, we proceed to verify the genuinely coupling condition for \eqref{eq:Asymbolqhm} and \eqref{eq:Bsymbolqhm}. First notice that 
\begin{align*}
\ker B(\xi)=\operatorname{span}\{(1,0,0,0)^{\top}\}.
\end{align*}
Let $W=\left(W_{1},0,0,0\right)^{\top}\in\ker B(\xi)$. If $\lambda(\xi)\in\mathbb{R}$ denotes an arbitrary eigenvalue of $A(\xi)$, the condition $A(\xi)W=\lambda W$ is 
\begin{align*}
\left(\begin{array}{c}
(\overline{\bfv}\cdot\omega) W_{1}\\
\Theta(\xi)\frac{\omega_{1}}{\bar{n}}W_{1}\\
\Theta(\xi)\frac{\omega_{2}}{\bar{n}}W_{1}\\
\Theta(\xi)\frac{\omega_{3}}{\bar{n}}W_{1}
\end{array}\right)=\lambda\left(\begin{array}{c}
W_{1}\\
0\\
0\\
0
\end{array}\right),
\end{align*}
which implies $W_{1}=0$ and thus, the genuinely coupling condition is satisfied. By the equivalence Theorem \ref{equivalencetheo}, we conclude that there exists a compensating matrix symbol for the triplet $(S(\xi),S(\xi)A(\xi),S(\xi)B(\xi))$.

\begin{remark}
\label{remhairyball}
Let us consider the symbol matrix given in \eqref{eq:Asymbolqhm}, for which its two eigenvalues $\lambda_{1}(\xi)$ and $\lambda_{2}(\xi)$ are smooth in $\xi\in \R^{d}\setminus \{ 0 \} $ (see \eqref{eq:qhdeigen1} and \eqref{eq:qhdeigen2}). Thus we only need to find a smooth representation for their respective projections: one way of doing so is by trying to compute the eigenvectors to obtain the matrix representation of the projections. For $\lambda_{1}=0$, an eigenvector $W=(W_{1}(\xi),W_{2}(\xi),W_{3}(\xi),W_{4}(\xi))$ satisfies $W_{1}\equiv 0$ in $\mathbb{R}^{3}\setminus\{0\}$, while $W_{2}$, $W_{3}$ and $W_{4}$  obey the relation, 
\begin{align}
	(\xi_{1},\xi_{2},\xi_{3})\cdot(W_{2}(\xi),W_{3}(\xi),W_{4}(\xi))=0\quad\mbox{for all}\quad\xi\in\mathbb{R}^{3}\setminus\{0\}.\label{eq:retraction}
\end{align}
Moreover, in order for $W$ to properly define an eigenvector, we have to assure that $W^{\prime}(\xi)=(W_{2},W_{3},W_{4})\neq 0$ on $\mathbb{R}^{3}\setminus\{0\}$. Although locally we can construct smoothly such a vector field (see \cite[Theorem 5.1]{AngF25}), a global smooth field in $\R^{d}\setminus \{ 0 \}$ satisfying \eqref{eq:retraction} may not exist. Indeed, this is the case, because otherwise it contradicts the Hairy Ball Theorem \cite[Theorem 8.5.13]{AMR88}: there is not a non-zero continuous vector field $g(\omega)$ defined on $\bbS^{2}$ such that $g(\omega)$ is tangent to $\omega$. Thus, \emph{an explicit smooth formula for the eigenprojection of this eigenvalue does not seem to exist}. Hence in this case, as well as in the example of Section \ref{sec-NSFK}, we compute the compensating matrix by inspection. 
\end{remark}

We begin by noticing that 
\begin{align*}
A_{S}(\xi)&=S^{1/2}(\xi)A(\xi)S^{-1/2}(\xi)=\left(\begin{array}{cccc}
0&\Theta^{1/2}(\xi)\omega_{1}&\Theta^{1/2}(\xi)\omega_{2}&\Theta^{1/2}(\xi)\omega_{3}\\
\Theta^{1/2}(\xi)\omega_{1}&0&0&0\\
\Theta^{1/2}(\xi)\omega_{2}&0&0&0\\
\Theta^{1/2}(\xi)\omega_{3}&0&0&0
\end{array}\right),\\ B_{S}(\xi)&=S^{1/2}(\xi)B(\xi)S^{-1/2}(\xi)=\left(\begin{array}{cccc}
0&0&0&0\\
0&1/\tau&0&0\\
0&0&1/\tau&0\\
0&0&0&1/\tau
\end{array}\right).
\end{align*}
Let us show that the skew symmetric matrix symbol
\begin{align*}
K_{S}(\xi)=\frac{1}{6\Theta^{1/2}(\xi)\tau}\left(\begin{array}{cccc}
0&\omega_{1}&\omega_{2}&\omega_{3}\\
-\omega_{1}&0&0&0\\
-\omega_{2}&0&0&0\\
-\omega_{3}&0&0&0
\end{array}\right),
\end{align*}
satisfies Definition \ref{Matrixsymbol} for the triplet $\left(I_4,A_{S}(\xi),B_{S}(\xi)\right)$. First, observe that 
\begin{align*}
\left[K_{S}(\xi)A_{S}(\xi)\right]^{s}=\frac{1}{6\tau}\left(\begin{array}{cccc}
1&0&0&0\\
0&-\omega_{1}^{2}&-\omega_{1}\omega_{2}&-\omega_{1}\omega_{3}\\
0&-\omega_{1}\omega_{2}&-\omega_{2}^{2}&-\omega_{2}\omega_{3}\\
0&-\omega_{1}\omega_{3}&-\omega_{2}\omega_{3}&-\omega_{3}^{2}
\end{array}\right).
\end{align*}
Let $W=(W_{1},W_{2},W_{3},W_{4})^{\top}$ be an arbitrary vector in $\mathbb{R}^{4}$ and denote by $(\cdot,\cdot)_{\mathbb{R}^{4}}$ the corresponding inner product. Then, 
\begin{align*}
\Big(\big(\big[K_{S}(\xi)A_{S}(\xi)\big]^{s} +B_{S}(\xi)\big) W,W\Big)_{\mathbb{R}^{4}} &=\frac{1}{6\tau}\left( W_{1}^{2}-\left[\omega_{1}W_{2}^{2}+\omega_{2}W_{3}^{2}+\omega_{3}W_{4}^{4}\right] \right) \\
&\quad - \frac{1}{3\tau} \left(\omega_{1}\omega_{2}W_{2}W_{3}+\omega_{1}\omega_{3}W_{2}W_{4}+\omega_{2}\omega_{3}W_{3}W_{4}\right) +\frac{1}{\tau}\left[W_{2}^{2}+W_{3}^{2}+W_{4}^{2}\right]\\
&\geq\frac{1}{6\tau}W_{1}^{2}+\frac{1}{\tau}\left[W_{2}^{2}+W_{3}^{2}+W_{4}^{2}\right]-\frac{1}{2\tau}\left[W_{2}^{2}+W_{3}^{2}+W_{4}^{2}\right]\\
&=\frac{1}{6\tau}W_{1}^{2}+\frac{1}{2\tau}\left[W_{2}^{2}+W_{3}^{2}+W_{4}^{2}\right]\\
&\geq\frac{1}{6\tau}|W|^{2}.
\end{align*}
Therefore, 
\begin{align*}
\left[K_{S}(\xi)A_{S}(\xi)\right]^{s}+B_{S}(\xi)\geq\frac{1}{6\tau}I_4.
\end{align*}
Now, according with formula \eqref{compmatrix} in Theorem \ref{hypequiv}, 
\begin{equation}
K(\xi):=S^{1/2}(\xi)K_{S}(\xi)S^{-1/2}(\xi) =\frac{1}{6\tau}\left(\begin{array}{cccc}
0&\tfrac{\omega_{1}}{\bar{n}}&\tfrac{\omega_{2}}{\bar{n}}&\tfrac{\omega_{3}}{\bar{n}}\\
\tfrac{-\bar{n}\omega_{1}}{\Theta(\xi)}&0&0&0\\
\tfrac{-\bar{n}\omega_{2}}{\Theta(\xi)}&0&0&0\\
\tfrac{-\bar{n}\omega_{3}}{\Theta(\xi)}&0&0&0
\end{array}\right),
\end{equation}
is a compensating matrix symbol for the triplet $\left(S(\xi),A(\xi),B(\xi)\right)$. In particular, notice that
\begin{align*}
K(\xi)S(\xi)=\frac{1}{6\tau}\left(\begin{array}{cccc}
0&\tfrac{\omega_{1}}{\bar{n}}&\tfrac{\omega_{2}}{\bar{n}}&\tfrac{\omega_{3}}{\bar{n}}\\
-\tfrac{\omega_{1}}{\bar{n}}&0&0&0\\
-\tfrac{\omega_{2}}{\bar{n}}&0&0&0\\
-\tfrac{\omega_{3}}{\bar{n}}&0&0&0
\end{array}\right),
\end{align*}
is skew-symmetric for any $\xi\in\mathbb{R}^{3}$ and any $\overline{V} \in\mathcal{M}_{\mathrm{eq}}$.
\begin{remark}
Observe that, since 
\begin{align*}
\frac{1}{\Theta^{1/2}(\xi)}\leq\frac{1}{(p^{\prime}(\bar{n})+\mu)^{1/2}},\quad\mbox{and}\quad\frac{|\xi|}{\Theta^{1/2}(\xi)}\leq\frac{2\sqrt{3}}{\epsilon},
\end{align*}
there are uniform positive constants $C_{1}=C_{1}(\tau,\bar{n},\mu)$ and $C_{2}=C_{2}(\epsilon)$ (only depending on the physical parameters and on the base state under consideration) such that,
\begin{align*}
|K(\xi)|&\leq C_{1},\quad\mbox{for all}\quad\xi\in\mathbb{R}^{3},\\
|\xi K(\xi)|&\leq C_{2},\quad\mbox{for all}\quad\xi\in\mathbb{R}^{3}\setminus\{0\}.
\end{align*}
Moreover, since $B_{S}(\xi)$ is independent of $\xi\in\mathbb{R}^{3}$  we can assure the existence of a positive constant $C_{3}=C_{3}(\tau)$ such that 
\begin{align*}
|B_{S}(\xi)|\leq C_{3},\quad\mbox{for all}\quad\xi\in\mathbb{R}^{3}.
\end{align*}
Therefore, by taking $\theta(\xi) \equiv \tfrac{1}{6\tau}$, we conclude that the assumptions of the equivalence Theorem \ref{equivalencetheo} and Lemma \ref{pnt-ee-full-lem} hold true. 
\end{remark}

We summarize the previous observations/calculations into the following result.
\begin{proposition}
\label{propisoQHD}
The linearized isentropic QHD system \eqref{linQHD} is genuinely coupled, symbol symmetrizable and, therefore, it satisfies the hypotheses of the equivalence Theorem \ref{equivalencetheo}. Moreover, system \eqref{eq:qhm1} - \eqref{eq:qhm2} is strictly dissipative of the standard type.
\end{proposition}
\begin{proof}
Follows directly from the previous observations; in particular the bounds on $|B_{S}(\xi)|$, $|K(\xi)|$ and $|\xi K(\xi)|$ immediately imply the conclusion of Lemma \ref{pnt-ee-full-lem}, upon taking $f(\xi) = \theta(\xi) \equiv 1/\tau$.
\end{proof}

We finish the section by stating the result describing the decay structure of the solutions to the linearized isentropic QHD system \eqref{linQHD}. Its proof is very similar to that of Lemma \ref{lemdecayEFK1d} and we omit it.

\begin{lemma}
\label{lemdecayQHD}
Let us assume that $U= (n, \bfv)^\top$ is a solution of the linear QHD system \eqref{linQHD} with initial condition $(n,\bfv)(x,0)\in \big( H^{s+1}(\R^3)\times \big( H^{s}(\R^3) \big)^{3} \big)\cap \big( L^{1}(\R^3) \big)^{4}$ for $s\geq 0$. Then for each $0\leq \ell \leq s $ the estimate
\begin{equation}
\label{ln-dec-QHD}
\begin{aligned}
\big( \Vert \partial_{x}^{\ell} n(t) \Vert_{1}^{2} + \Vert \partial_{x}^{\ell} \bfv(t) \Vert_{0}^{2} \big)^{1/2} &\leq  Ce^{-c_{1}t} \big( \Vert \partial_{x}^{\ell}n(0) \Vert_{1}^{2}+ \Vert \partial_{x}^{\ell}\bfv(0) \Vert_{0}^{2}  \big)^{1/2}+ \\ & \quad + C(1+t )^{-(1/4+\ell)} \Vert (n,\bfv)(0)\Vert_{L^{1}},
\end{aligned}
\end{equation}
holds for all $t\geq 0$ and some uniform positive constants $C$, $c_{1}$.
\end{lemma}

\section{Physical models violating the assumptions of the equivalence theorem}
\label{secnoway}

In this last Section we examine a couple of models of physical origin which \emph{do not} satisfy the assumptions of the equivalence Theorem \ref{equivalencetheo}. In the first example, the system is not genuinely coupled. In the second, symmetrizability fails, even at the symbol level.
%
%

\subsection{Inviscid heat conducting viscous-capillar fluids in several space dimensions}
\label{secEFKmultid}

In Section \ref{secEFK1d} we verified that the Euler-Fourier-Korteweg system \eqref{EFK1d} in one space dimension, describing \emph{inviscid} fluids (with $\nu = \lambda = 0$), which conduct heat (with thermal conductivity $ \alpha> 0$) and exhibiting capillarity effects ($k > 0$) is genuinely coupled and consequently, by the equivalence Theorem \ref{equivalencetheo}, strictly dissipative. It is remarkable, however, that the same physical model is not genuinely coupled in several space dimensions (that is, when $d \geq 2$).
%

\begin{proposition}
\label{propEFKmdnorgc}
In several space dimensions, $d \geq 2$, the Euler-Fourier-Korteweg system \eqref{EFK-system} is not genuinely coupled.
\end{proposition}
\begin{proof}
For the system in several space dimension, the system in the Fourier spaces reads
\begin{equation}\label{EFK-D-Four}
A^{0}\hU_{t}+\big( i \vert \xi \vert A(\xi) + B(\xi) \big) \hU=0,
\end{equation}
where the generalized flux and the generalized flux matrices have the form: 
\[
A(\xi)= \begin{pmatrix}
\bbu\cdot \omega & \brho \omega^{\top} & 0 \\ \beta(\xi)\omega & \brho\:(\bbu\cdot \omega) I_{d} & \bp_{\theta} \omega \\ 0 & \bthe\:\bp_{\theta}\omega^{\top} & \brho\:\be_{\theta}\:(\bbu\cdot\omega)
\end{pmatrix},\quad \beta(\xi):= \bp_{\rho}+\bk\:\brho\:\vert \xi \vert^{2};
\]
\[
B(\xi)=\vert \xi \vert^{2}\begin{pmatrix}
0 & 0_{1 \times d} & 0 \\ 0_{d \times 1} & 0_{d} & 0_{d \times 1} \\ 0 & 0_{1 \times d} & \balp \vert \omega \vert^{2}
\end{pmatrix}.
\]
As in the one-dimensional case, straightforward computations show that
\[
S(\xi)= \begin{pmatrix}
\beta(\xi)/\brho & 0_{1 \times d}& 0 \\ 0_{d \times 1} & I_{d} & 0_{d \times 1} \\ 0 &0_{1 \times d}& 1/\bthe
\end{pmatrix}
\]
is a symbol symmetrizer for \eqref{EFK-D-Four}. Then for the variable $\hV:=\big(SA^{0}(\xi)\big)^{1/2}\hU$ satisfying 
\[
\hV_{t}+\big( i \vert \xi \vert \tiA(\xi) + \tiB(\xi) \big) \hV=0,
\]	
the matrices $\tiA(\xi)$ and $\tiB(\xi)$ are given by
\begin{equation}\label{tiA-EFK-D}
\tiA(\xi)=\begin{pmatrix}
\bbu\cdot\omega & \beta^{1/2}\omega^\top & 0 \\ \beta^{1/2}\omega & (\bbu\cdot\omega) \:I_{d} & \gamma\:\omega \\ 0 & \gamma\:\omega^{\top} & \bbu\cdot\omega, 
\end{pmatrix}
\end{equation}	
and 
\begin{equation}\label{tiB-EFK-D}
\tiB(\xi) = \vert \xi \vert^{2} \begin{pmatrix}
0 & 0_{1 \times d} & 0 \\ 0_{d \times 1} & 0_{d} & 0_{d \times 1} \\ 0 & 0_{1 \times d} & \balp/\brho\:\be_{\theta},
\end{pmatrix}
\end{equation}
with 
\[
\gamma=\frac{\bthe^{1/2}\bp_{\theta}}{\brho \: \be_{\theta}^{1/2}},
\] 
as before. Observe that $\tiA(\xi)$ have the same structure as the symbol for the Euler equations (see, for instance, Benzoni-Gavage and Serre \cite{BS07}). If $\big( \tirho, \tibu, \tithe \big)^\top \in\R^{d+2}$ is an eigenvector of $\tiA(\xi)$ with associated eigenvalue $\lambda(\xi)$, it must happen that
\[
\begin{aligned}
\bbu\cdot\omega -\lambda(\xi)+\beta(\xi)^{1/2}(\omega\cdot\tibu)&=0,\\
\big( \beta(\xi)^{1/2}\tirho + \tithe\:\gamma\big)\omega + \big(\bbu\cdot\omega-\lambda(\xi) \big)\tibu &=0,\\
\gamma\:(\omega\cdot\tibu) + \big(\bbu\cdot\omega-\lambda(\xi) \big)&=0.
\end{aligned}
\]
%
%
Therefore $\lambda(\xi)=\bbu\cdot\omega$ is an eigenvalue, whose $d-$dimensional eigenspace is spanned by 
\[
\{ \big( \tirho,\tibu,\tithe \big)^\top \in \R^{d+2}:\: \tibu\cdot\omega=0,\: \beta^{1/2}\tirho+\gamma\tithe=0 \},
\]
which we write as 
\[
\text{span}\big\{ \big(-\gamma,0,\beta^{1/2}\big)^\top \big\} \cup \{ \big(0, \tibu,0 \big)^\top \in \R^{d+2}:\: \tibu\cdot\omega=0 \}.
\]
	
By the form of $\tiB(\xi)$ given by \eqref{tiB-EFK-D}, it is easy to see that for each eigenvalue of $\tiA(\xi)$ of the form $\big( 0, \tibu,0 \big)$, with $\tibu\cdot\omega=0$, satisfies
\[
\tiB(\xi)\big(0,\tibu,0\big)^{\top}=0.
\]	
Therefore the triplet $\big(I, \tiA(\xi),\tiB(\xi) \big)$ fails to be genuinely coupled. 	
\end{proof}
	
\subsection{Quantum hydrodynamics equations for inviscid fluids}

Let us consider the full QHD system derived by Gardner \cite{GarC94}, which includes energy exhanges. The description of charged quantum fluids appears in semiconductor physics to describe quantum effects which are manifest at a macroscopic scale (cf. \cite{Jue01,GarC94}). The behavior of a quantum inviscid fluid near thermal equilibrium and in the high temperature limit can be approximated by $O(\hbar^2)$ corrections to the classical formulas of the stress tensor and the energy for a compressible inviscid fluid in space (see Gardner \cite{GarC94}); here $\hbar$ is, as before, the Planck constant. These models are primarily used to study the flow of electrons in quantum semiconductor devices (see, e.g., J\"ungel \cite{Jue01} and the references therein). The quantum hydrodynamic equations (QHD) can be expressed as conservation laws for the electron density $n$, the momentum $n\bfv$ and the total energy $W$ as follows
\begin{align}
	\partial_{t}n+\nabla\cdot(n \bfv)&=0,\label{eq:qhd1}\\
	\partial_{t}(n\bfv)+\nabla\cdot\left(n\bfv\otimes \bfv - \mathbf{P}\right)&=-\frac{n\bfv}{\tau_{p}},\label{eq:qhd2}\\
	\partial_{t}W+\nabla\cdot\left(W\bfv-\mathbf{P}\bfv + \bfq\right)&=-\frac{\left(W-\tfrac{3}{2}n\theta_{0}\right)}{\tau_{w}},\label{eq:qhd3}
\end{align}
for $x \in \R^3$, $t > 0$, where $\bfv \in \R^3$ is the velocity field, $\bfq \in \R^3$ is the heat flux, $\theta_{0}$ is the temperature of the semiconductor lattice and 
\begin{align}
	\mathbf{P}&:=-n\theta I_3+\frac{\hbar^{2}n}{12}\nabla\otimes\nabla\log(n),\label{eq:quantumstress}\\
	W&:=\tfrac{3}{2}n\theta+\frac{1}{2}n|\bfv|^{2}-\frac{\hbar^{2}n}{24}\Delta\log(n).\label{quantumenergy}
\end{align}

Notice that the classical terms in the stress tensor and the total energy correspond to the model of a fluid with an ideal monoatomic gas pressure and internal energy. The right hand sides of \eqref{eq:qhd2} and \eqref{eq:qhd3} represent electron scattering, which is modelled by the standard relaxation time approximation with momentum and energy relaxation times $\tau_{p}$ and $\tau_{w}$, respectively (for further details, see Gardner \cite{GarC94}). It is also to be observed that the relaxation term appearing in \eqref{eq:qhd3} is precisely the relaxation term used in classical hydrodynamical models for semiconductors (cf. \cite{MaRiSc90}). Finally, it is assumed that the heat flux is specified by Fourier's law, which reads,  $\bfq=-\kappa\nabla\theta$.

We now write the three dimensional version of equations \eqref{eq:qhd1}-\eqref{eq:qhd3} in quasilinear form as
\begin{equation}
\label{eq:qquasilinear}
\begin{aligned}
	A^{0}(U)\partial_{t}U + \sum_{i=1}^3 A^{i}(U)\partial_{i}U+Q(U) &= \sum_{i,j=1}^3 B^{ij}(U)\partial_{i}\partial_{j}U + \sum_{i,j,k=1}^3C^{ijk}(U)\partial_{i}\partial_{j}\partial_{k}U\\ &\quad  + \sum_{i,j,k,\ell=1}^3 F^{ijk\ell}(U)\partial_{i}\partial_{j}\partial_{k}\partial_{\ell}U +\mathcal{N}(U;\nabla U;\nabla\otimes\nabla U),
\end{aligned}
\end{equation}
for the state variable $U=(n,\bfv,g)^{\top}\in\mathbb{R}^{5}$ where $	g:=\frac{3}{2}\theta-\frac{\hbar^{2}}{24}\Delta\log(n)$. Notice that every matrix coefficient is of order $5\times 5$ and $Q(U)$ is a vector field in $\mathbb{R}^{5}$.
In order to obtain the evolution equation for $g$ we first take the inner product of equation \eqref{eq:qhd2} with $\bfv$, yielding
\begin{align*}
	\partial_{t}\left(\frac{n|\bfv|^{2}}{2}\right)+\nabla\cdot\left(\frac{1}{2}n|\bfv|^{2}\bfv\right)=(\nabla\cdot\mathbf{P})\cdot \bfv-\frac{n|\bfv|^{2}}{\tau_{p}}.
\end{align*}
By subtracting this last equation from \eqref{eq:qhd3} we get
\begin{align}
	\partial_{t}(ng)+\nabla\cdot(ng\bfv)+\nabla\cdot \bfq=-\frac{1}{\tau_{w}}ng+\frac{1}{\tau_{w}}\frac{3}{2}\theta_{0}+n|\bfv|^{2}\left[\frac{1}{\tau_{p}}-\frac{1}{2\tau_{w}}\right]+\mathbf{P}:\nabla \bfv.\label{eq:qhd4}
\end{align}
The equation for $g$ is obtained once we expand the partial time derivative in \eqref{eq:qhd4} and use the continuity equation \eqref{eq:qhd1}. We apply this last step in equation \eqref{eq:qhd2} in order to obtain a single equation for $\bfv$. Then, the quasilinear form of equations \eqref{eq:qhd1}, \eqref{eq:qhd2} and \eqref{eq:qhd4} is determined by the matrix $A^{0}(U)$ and the symbols of the higher order operators in the equations, given by
\[
\begin{aligned}
	A^{0}&=\left(\begin{array}{ccc}
		1&\\
		&n I_4
	\end{array}\right),\\
	\sum_{i=1}^3 \xi_{i}A^{i}(U)&= \sum_{i=1}^3|\xi|\omega_{i}A^{i}(U)=|\xi|\left(\begin{array}{ccccc}
		\bfv\cdot\omega&n\omega_{1}&n\omega_{2}&n\omega_{3}&0\\
		\tfrac{2}{3}g\omega_{1}&n\bfv\cdot\omega&0&0&\tfrac{2}{3}n\omega_{1}\\
		\tfrac{2}{3}g\omega_{2}&0&n\bfv\cdot\omega&0&\tfrac{2}{3}n\omega_{2}\\
		\tfrac{2}{3}g\omega_{3}&0&0&n\bfv\cdot\omega&\tfrac{2}{3}n\omega_{3}\\
		0&\tfrac{2}{3}ng\omega_{1}&\tfrac{2}{3}ng\omega_{2}&\tfrac{2}{3}ng\omega_{3}&n\bfv\cdot\omega
	\end{array}\right),\\
	\sum_{i,j=1}^3\xi_{i}\xi_{j}B^{ij}(U)&= \sum_{i,j=1}^3 |\xi|^{2}\omega_{i}\omega_{j}B^{ij}(U)=\left(\begin{array}{cc}
		0_{4\times 4}&\\
		&\tfrac{2}{3}\kappa
	\end{array}\right),\\
	\sum_{i,j,k=1}^3\xi_{i}\xi_{j}\xi_{k}C^{ijk}(U)&= \sum_{i,j,k=1}^3 |\xi|^{3}\omega_{i}\omega_{j}\omega_{k}C^{ijk}(U)=|\xi|^{3}\left(\begin{array}{cccc}
		0 & 0_{1\times 4}&\\
		\omega_{1}\tfrac{\hbar^{2}}{18}&\\
		\omega_{2}\tfrac{\hbar^{2}}{18}&0_{3\times 4}&\\
		\omega_{3}\tfrac{\hbar^{2}}{18}&\\
		0 &0_{1\times 4}&
	\end{array}\right),\\
	\sum_{i,j,k,\ell=1}^3\xi_{i}\xi_{j}\xi_{k}\xi_{\ell}F^{ijk\ell}(U)&= \sum_{i,j,k,\ell=1}^3|\xi|^{4}\omega_{i}\omega_{j}\omega_{k}\omega_{\ell}F^{ijk\ell}(U)=|\xi|^{4}\left(\begin{array}{ccc}
		0 &0_{1\times 4}\\
		\tfrac{\kappa\hbar^{2}}{36}\tfrac{1}{n}&0_{1\times 4}&
	\end{array}\right),
\end{aligned}
\]
the vector field $Q$ is
\begin{align*}
	Q(U)=\left(\begin{array}{c}
		0\\
		\tfrac{1}{\tau_{p}}n\bfv\\
		\tfrac{1}{\tau_{w}}ng-\tfrac{1}{\tau_{w}}\tfrac{3}{2}n\theta_{0}+n|\bfv|^{2}\left[\tfrac{1}{2\tau_{w}}-\tfrac{1}{\tau_{p}}\right]
	\end{array}\right),
\end{align*}
and where $\mathcal{N}(U;\nabla U;\nabla\otimes\nabla U)$ comprises the nonlinear terms (which we do not write here for the sake of brevity). 

Let $\mathcal{O}$ be the open convex set in $\mathbb{R}^{5}$ defined as $\mathcal{O}:=\{(n,\bfv,g)^{\top}\in\mathbb{R}^{5}~:~n,\theta>0\}$ and consider the set $\mathcal{M}_{\mathrm{eq}}$ of constant equilibrium states $\overline{V}$, that is, 
\begin{align*}
	\mathcal{M}_{\mathrm{eq}}:=\left\lbrace \overline{V}\in\mathcal{O}~:~\overline{V}~\mbox{is constant and}~Q(\overline{V})=0\right\rbrace.
\end{align*}
This set is characterized as all the constant states $(\bar{n},\bar{\bfv},\bar{g})^{\top}$ for which $\bar{\bfv}=0$ and $\bar{g}=\tfrac{3}{2}\theta_{0}$. By the linearization of \eqref{eq:qquasilinear} around $\overline{V}\in\mathcal{M}_{\mathrm{eq}}$ we mean the system 
\[
	\overline{A}^{0}\partial_{t}U+ \sum_{i=1}^3 \overline{A}^i\partial_{i}U+\overline{L}U= \sum_{i,j=1}^3 \overline{B}^{ij}\partial_{i}\partial_{j}U+ \sum_{i,j,k=1}^3 \overline{C}^{ijk}\partial_{i}\partial_{j}\partial_{k}U + \sum_{i,j,k,\ell=1}^3\overline{F}^{ijk\ell}\partial_{i}\partial_{j}\partial_{k}\partial_{\ell}U,
\]
where $\overline{A}^i := A^i(\overline{V})$, $\overline{C}^{ijk} := C^{ijk}(\overline{V})$, etc. and $\overline{L} := DQ(\overline{V})$, being $DQ$ the Jacobian of $Q$ with respect to the state variables $U$. Upon application of the Fourier transform to the linearized system we find 
\[
\overline{A}^0\widehat{U}_{t}+i|\xi|A(\xi)\widehat{U}
	+B(\xi) \widehat{U}=0,
\]
where, as in \eqref{Fourieroe}, we have set
\begin{align*}
	A(\xi)&:=\sum_{i=1}^3 \omega_{i}\overline{A}^i+|\xi|^{2} \sum_{i,j,k=1}^3\omega_{i}\omega_{j}\omega_{k}\overline{C}^{ijk}
,\\
	B(\xi)&=\overline{L}+|\xi|^{2} \sum_{i,j=1}^3\omega_{i}\omega_{j}\overline{B}^{ij}-|\xi|^{4}\sum_{i,j,k,\ell=1}^3\omega_{i}\omega_{j}\omega_{k}\omega_{\ell}\overline{F}^{ijk\ell}.
\end{align*}
%
Observe that, for any $\overline{V}=(\bar{n},\bar{\bfv},\bar{g})^{\top}\in\mathcal{M}_{\mathrm{eq}}$, the pair $(\overline{A}^0,A(\xi))$ accepts only a symbolic symmetrizer, which can be given as the diagonal matrix
\begin{align*}
	S_{1}(\xi):=\left(\begin{array}{ccc}
		\tfrac{1}{\bar{n}}\left(\tfrac{2}{3}\bar{g}+|\xi|^{2}\tfrac{\hbar^{2}}{18}\right)&0_{1 \times 3}&0\\
		0_{3 \times 1}&I_3&0_{3 \times 1}\\
		0&0_{1 \times 3}&1/{\bar{g}}
	\end{array}\right).
\end{align*}
Similarly, a symbolic symmetrizer for the pair $\big(I_5,(\overline{A}^{0})^{-1}A(\xi)\big)$ can be given as
\begin{align*}
	S_{2}(\xi):=\left(\begin{array}{ccc}
		\tfrac{1}{\bar{n}^{2}}\left(\tfrac{2}{3}\bar{g}+|\xi|^{2}\tfrac{\hbar^{2}}{18}\right)& 0_{1 \times 3}&0\\
		0_{3 \times 1}&I_3&0_{3 \times 1}\\
		0&0_{1 \times 3}&1/{\bar{g}}
	\end{array}\right),\quad\mbox{for any}\quad \overline{V}\in\mathcal{O}.
\end{align*}
However, the symbols $A(\xi)$ and $B(\xi)$ cannot be simultaneously symmetrized, even with a symbolic symmetrizer. Let us prove this assertion. 

\begin{proposition}
The QHD system \eqref{eq:qhd1} - \eqref{eq:qhd3}, when linearized around a constant equilibrium state $\overline{V} \in \mathcal{M}_{\mathrm{eq}}$, is not symbol symmetrizable (and, therefore, it does not satisfy one of the assumptions of the equivalence Theorem \ref{equivalencetheo}).
\end{proposition}
\begin{proof}
For any $\overline{V}\in\mathcal{M}_{\mathrm{eq}}$, it holds that 
\begin{align*}
	B(\xi)=\left(\begin{array}{ccccc}
		0&0&0&0&0\\
		0&\tfrac{\bar{n}}{\tau_{p}}&0&0&0\\
		0&0&\tfrac{\bar{n}}{\tau_{p}}&0&0\\
		0&0&0&\tfrac{\bar{n}}{\tau_{p}}&0\\
		-\tfrac{|\xi|^{4}}{\bar{n}}\bar{\kappa}\tfrac{\hbar^{2}}{36}&0&0&0&\tfrac{n}{\tau_{w}}+|\xi|^{2}\tfrac{2}{3}\bar{\kappa}\end{array}\right).
\end{align*}
Now let us assume the existence of a symbolic symmetrizer $S(\xi)$ for the pair $(A(\xi),B(\xi))$. Then, in particular, for any $\overline{V}\in\mathcal{M}_{\mathrm{eq}}$, it holds that $S(\xi)B(\xi)$ is symmetric for all $\xi\in\mathbb{R}^{3}\setminus\{0\}$. Let $s_{ij}(\xi)$ denote the components of $S(\xi)$ and set 
\begin{align*}
	b :=\frac{\bar{n}}{\tau_{p}},\quad G(\xi):=-\frac{|\xi|^{4}}{\bar{n}}\bar{\kappa}\frac{\hbar^{2}}{36}\quad\mbox{and}\quad R(\xi):=\frac{\bar{n}}{\tau_{w}}+|\xi|^{2}\frac{2}{3}\bar{\kappa}.
\end{align*}
The condition of symmetrizability implies that 
\begin{equation}
	\label{eq:symB}
	\begin{aligned}
		s_{12}(\xi)b&=s_{25}(\xi)G(\xi), &s_{13}(\xi)b&=s_{35}(\xi)G(\xi),\\
		s_{14}(\xi)b&=s_{45}(\xi)G(\xi), &s_{15}(\xi)R(\xi)&=s_{55}(\xi)G(\xi),\\
		s_{23}(\xi)b&=s_{23}(\xi)b, &s_{24}(\xi)b&=s_{24}(\xi)b,\\
		s_{25}(\xi)R(\xi)&=s_{25}(\xi)b, &s_{34}(\xi)b&=s_{34}(\xi)b,\\
		s_{35}(\xi)R(\xi)&=s_{35}(\xi)b, &s_{45}(\xi)R(\xi)&=s_{45}(\xi)b.
	\end{aligned}
\end{equation}
From the last equations we have that
\begin{align*}
	&s_{25}(\xi)\left(R(\xi)-b\right)=0,\quad s_{35}(\xi)\left(R(\xi)-b\right)=0,\\
	&s_{45}(\xi)\left(R(\xi)-b\right)=0\quad\forall~\xi\in\mathbb{R}^{3}\setminus\{0\}.
\end{align*}
But since $R(\xi)-b\neq 0$ for any $\xi\in\mathbb{R}^{3}\setminus\{0\}$, it follows that 
\begin{align*}
	s_{25}(\xi)=s_{35}(\xi)=s_{45}(\xi)=0\quad\forall~\xi\in\mathbb{R}^{3}\setminus\{0\}.
\end{align*}
Since $b\neq 0$ we also have that 
\begin{align*}
	s_{12}(\xi)=s_{13}(\xi)=s_{14}(\xi)=0\quad\forall~\xi\in\mathbb{R}^{3}\setminus\{0\}.
\end{align*}
Now, set $H(\xi):=\tfrac{2}{3}\bar{g}+|\xi|^{2}\tfrac{\hbar^{2}}{18}$. The non-trivial symmetrizability conditions for $S(\xi)A(\xi)$ read as follows
\begin{align*}
	&\left(s_{11}(\xi)\bar{n}+s_{15}(\xi)\tfrac{2}{3}\bar{n}\bar{g}\right)\omega_{1}=H(\xi)\left(\omega_{1},\omega_{2},\omega_{3}\right)\cdot\left(s_{22}(\xi),s_{23}(\xi),s_{24}(\xi)\right),\\
	&\left(s_{11}(\xi)\bar{n}+s_{15}(\xi)\tfrac{2}{3}\bar{n}\bar{g}\right)\omega_{2}=H(\xi)\left(\omega_{1},\omega_{2},\omega_{3}\right)\cdot\left(s_{23}(\xi),s_{33}(\xi),s_{34}(\xi)\right),\\
	&\left(s_{11}(\xi)\bar{n}+s_{15}(\xi)\tfrac{2}{3}\bar{n}\bar{g}\right)\omega_{3}=H(\xi)\left(\omega_{1},\omega_{2},\omega_{3}\right)\cdot\left(s_{24}(\xi),s_{34}(\xi),s_{44}(\xi)\right),\\
	&\left(s_{15}(\xi)+s_{55}(\xi)\tfrac{2}{3}\bar{g}\right)\omega_{1}=\tfrac{2}{3}\left(\omega_{1},\omega_{2},\omega_{3}\right)\cdot\left(s_{22}(\xi),s_{23}(\xi),s_{24}(\xi)\right),\\
	&\left(s_{15}(\xi)+s_{55}(\xi)\tfrac{2}{3}\bar{g}\right)\omega_{2}=\tfrac{2}{3}\left(\omega_{1},\omega_{2},\omega_{3}\right)\cdot\left(s_{23}(\xi),s_{33}(\xi),s_{34}(\xi)\right),\\
	&\left(s_{15}(\xi)+s_{55}(\xi)\tfrac{2}{3}\bar{g}\right)\omega_{3}=\tfrac{2}{3}\left(\omega_{1},\omega_{2},\omega_{3}\right)\cdot\left(s_{24}(\xi),s_{34}(\xi),s_{44}(\xi)\right).
\end{align*}
Since these equations are valid for any $\omega\in\mathbb{S}^{2}$, we conclude that 
\begin{align*}
	s_{11}(\xi)+s_{15}(\xi)\tfrac{2}{3}\bar{g}=\tfrac{3H(\xi)}{2\bar{n}}\left(s_{15}(\xi)+s_{55}(\xi)\tfrac{2}{3}\bar{g}\right)
\end{align*}
holds for any $\xi\in\mathbb{R}^{3}\setminus\{0\}$. By equations \eqref{eq:symB}, 
\begin{align*}
	s_{15}(\xi)=s_{55}(\xi)\frac{G(\xi)}{R(\xi)},
\end{align*}
and therefore, 
\begin{align*}
	R(\xi)s_{11}(\xi)=P(|\xi|)s_{55}(\xi),
\end{align*}
where $P(|\xi|)$ is a sixth order polynomial in $|\xi|$ given as
\begin{align*}
	P(|\xi|)=\frac{3}{2}\frac{H(\xi)}{\bar{n}}G(\xi)+H(\xi)R(\xi)g-\frac{2}{3}\bar{g}G(\xi).
\end{align*}
Observe that $R(\xi)>0$ and, by the positive definiteness of $S(\xi)$, $s_{11}(\xi)>0$ and $s_{55}(\xi)>0$. Therefore, 
\begin{align*}
	P(|\xi|)s_{55}(\xi)>0\quad\forall~\xi\in\mathbb{R}^{3}\setminus\{0\},
\end{align*}
but since the dominant term of $P(|\xi|)$ has a negative coefficient, the contradiction follows.
\end{proof}
%
%
	
\section*{Acknowledgements}
	
This research was motivated by the works of S. Kawashima and J. Humpherys. R. G. Plaza is grateful to S. Kawashima for an insightful conversation during the HYP2012 meeting in Padova, Italy, and to J. Humpherys for many useful discussions along the years. The work of F. Angeles and R. G. Plaza was supported by CONAHCyT, M\'exico, grant CF-2023-G-122. F. Angeles was also partially supported by the C\'{a}tedra Extraordinaria IIMAS. The work of J. M. Valdovinos was partially supported by CONAHCyT, through a scholarship for graduate studies, grant no. 712874. 

	
\def\cprime{$'$}


	
	
	


\end{document}